\documentclass[a4paper,oneside,english,american]{amsart}
\usepackage[T1]{fontenc}
\usepackage[latin9]{inputenc}
\usepackage{color}
\usepackage{amstext}
\usepackage{amsthm}
\usepackage{amssymb}
\usepackage{graphicx}
\usepackage{esint}
\usepackage[all]{xy}
\PassOptionsToPackage{normalem}{ulem}
\usepackage{ulem}

\makeatletter

\numberwithin{equation}{section}
\theoremstyle{plain}
\newtheorem{thm}{\protect\theoremname}[section]
  \theoremstyle{plain}
  \newtheorem{conjecture}[thm]{\protect\conjecturename}
  \theoremstyle{plain}
  \newtheorem{prop}[thm]{\protect\propositionname}
  \theoremstyle{definition}
  \newtheorem{defn}[thm]{\protect\definitionname}
  \theoremstyle{remark}
  \newtheorem{rem}[thm]{\protect\remarkname}

\newcounter{myparagraph}[subsection]
\newcommand{\myparagraph}{\refstepcounter{myparagraph}}
\renewcommand{\themyparagraph}{{\arabic{section}.\arabic{subsection}.\alph{myparagraph}}}
\newcommand*{\para}[1]{\vskip0.3cm\noindent\hspace{-3pt}
\myparagraph{\bf \themyparagraph.{{\,#1}}}}

\newcounter{myobs}[thm]

\newenvironment{myitem}{\begin{list}{}{
\setlength{\leftmargin}{0.8cm}
\setlength{\itemindent}{-0.5cm}
\setlength{\itemsep}{2pt}
}}{\end{list}}

\def\hsp#1{{\hspace{ #1 pt}}}

\usepackage[all]{xy}

\usepackage{color}
\usepackage{amsthm}
\usepackage{bm}
\usepackage{graphicx, color, ulem}

\usepackage[labelformat=empty]{caption}

\makeatother

\usepackage{babel}
  \addto\captionsamerican{\renewcommand{\conjecturename}{Conjecture}}
  \addto\captionsamerican{\renewcommand{\definitionname}{Definition}}
  \addto\captionsamerican{\renewcommand{\propositionname}{Proposition}}
  \addto\captionsamerican{\renewcommand{\remarkname}{Remark}}
  \addto\captionsamerican{\renewcommand{\theoremname}{Theorem}}
  \addto\captionsenglish{\renewcommand{\conjecturename}{Conjecture}}
  \addto\captionsenglish{\renewcommand{\definitionname}{Definition}}
  \addto\captionsenglish{\renewcommand{\propositionname}{Proposition}}
  \addto\captionsenglish{\renewcommand{\remarkname}{Remark}}
  \addto\captionsenglish{\renewcommand{\theoremname}{Theorem}}
  \providecommand{\conjecturename}{Conjecture}
  \providecommand{\definitionname}{Definition}
  \providecommand{\propositionname}{Proposition}
  \providecommand{\remarkname}{Remark}
\providecommand{\theoremname}{Theorem}

\begin{document}
\selectlanguage{english}%
\global\long\def\ve{\mathcal{\varepsilon}}
\global\long\def\st{\;\mathcal{\text{s.t.}}\;}
\global\long\def\vp{\mathcal{\varphi}}
\global\long\def\c{\circ}

\global\long\def\cO{\mathcal{O}}

\global\long\def\cA{\mathcal{A}}

\global\long\def\cB{\mathcal{B}}

\global\long\def\cD{\mathcal{D}}

\global\long\def\cE{\mathcal{E}}

\global\long\def\cF{\mathcal{F}}

\global\long\def\cM{\mathfrak{M}}

\global\long\def\cU{\mathcal{U}}
\global\long\def\cV{\mathcal{V}}
\global\long\def\cW{\mathcal{W}}

\global\long\def\cN{\mathcal{N}}
\global\long\def\cH{\mathcal{H}}

\global\long\def\idR{1_{\mathbb{R}^{n}}}

\global\long\def\id{\mathrm{id}}

\global\long\def\bR{\mathbb{R}}

\global\long\def\mbR{\mathbb{R}}
\global\long\def\mbC{\mathbb{C}}

\global\long\def\mbP{\mathbb{P}}

\global\long\def\mbZ{\mathbb{Z}}

\global\long\def\bS{\bar{S}}
\global\long\def\bT{\bar{T}}

\global\long\def\tPd{\widetilde{\mathbb{P}}_{\Delta}}

\selectlanguage{american}%
\def\m{\hsp{-4}}

\def\n{\hsp{-2}}

~

\begin{flushright}
\par\end{flushright}

\title{Mirror Symmetry of Calabi-Yau Manifolds Fibered by (1,8)-Polarized
Abelian Surfaces}

\author{Shinobu Hosono and Hiromichi Takagi}
\begin{abstract}
We study mirror symmetry of a family of Calabi-Yau manifolds fibered
by (1,8)-polarized abelian surfaces with Euler characteristic zero.
By describing the parameter space globally, we find all expected boundary
points (LCSLs), including those correspond to Fourier-Mukai partners.
Applying mirror symmetry at each boundary point, we calculate Gromov-Witten
invariants ($g\leq2$) and observe nice (quasi-)modular properties
in their potential functions. We also describe degenerations of Calabi-Yau
manifolds over each boundary point. 
\end{abstract}

\maketitle
\selectlanguage{english}%
{\small{}\tableofcontents{}}{\small \par}

\selectlanguage{american}%

\section{\textbf{Introduction}}

Since the discovery of mirror symmetry of Calabi-Yau manifolds and
its surprising applications to Gromov-Witten invariants \cite{Candelas},
moduli spaces of Calabi-Yau manifolds as well as geometry of Calabi-Yau
manifolds are attracting attentions from both mathematics and physics.
After several decades from the discovery, a large number of interesting
Calabi-Yau manifolds are now known and many of them have been studied
in details. In this paper, to explore mirror symmetry in terms of
interesting Calabi-Yau manifolds, we will study a special type of
Calabi-Yau threefolds which have fibrations by abelian surfaces and
have vanishing Euler numbers. 

Mirror symmetry of Calabi-Yau threefolds exchanges the so-called $A$-side
(related to Hodge decomposition $H^{1,1}(M))$ of the (stringy) moduli
space of a Calabi-Yau manifold $M$ with the $B$-side (related to
$H^{2,1}(M^{\vee}))$ of a mirror Calabi-Yau manifold $M^{\vee}$.
More precisely $A$ and $B$ sides of $M$ are interchanged to $B$
and $A$ sides of $M^{\vee}$, which implies the isomorphisms $H^{1,1}(M)\simeq H^{2,1}(M^{\vee})$
and $H^{2,1}(M)\simeq H^{1,1}(M^{\vee})$. In most examples of Calabi-Yau
threefolds, since they have small Hodge numbers $h^{1,1}(M)$ while
large $h^{2,1}(M)$, our descriptions of mirror symmetry of the moduli
spaces are restricted to one part of the moduli space related to $H^{1,1}(M)\simeq H^{2,1}(M^{\vee})$.
For Calabi-Yau threefolds having vanishing Euler numbers, since the
equality $h^{1,1}(M)=h^{2,1}(M)$ holds, we can expect to describe
the entire moduli spaces if $h^{1,1}(M)$ is small. Our Calabi-Yau
threefolds provide simple but non-trivial examples of such Calabi-Yau
manifolds. 

Calabi-Yau threefolds which we will study in this paper are known
for long since the investigations of explicit equations of $(1,d)$-polarized
abelian surfaces by Gross and Popescu \cite{GP-A,GP-CYI}. Abelian
surfaces with $(1,d)$-polarization have embeddings in $\mbP^{d-1}$,
if $d\geq5$, by the polarization $\mathcal{L}$ which is very ample.
It was found \cite{GP-CYI} that, in many cases, \textcolor{black}{a
$(1,d)$-polarized abelian surface $(A,\mathcal{L})$ is embedded
in $\mathbb{P}^{d-1}$ by the canonical theta functions of the space
$\Gamma(A,\mathcal{L})$. By taking a part of the equations of $A$
in $\mathbb{P}^{d-1}$ and taking a suitable small resolution, a smooth
Calabi-Yau threefold is obtained, which contains $A$ as a fiber of
an abelian fibration. }Among several examples, we will be concerned
with a Calabi-Yau threefold fibered by $(1,8)$-polarized abelian
surfaces, for which we have $h^{1,1}=h^{2,1}=2$. Following \cite{GP-CYI},
we denote this Calabi-Yau threefold by $V_{8,w}^{1}$. 

By construction, Calabi-Yau threefold $V_{8,w}^{1}$ admits a free
action by Heisenberg group $\cH_{8}=\langle\sigma,\tau\rangle$, which
acts on $V_{8,w}^{1}$ as $\mbZ_{8}\times\mbZ_{8}$. Quotients $V_{8,w}^{1}/\mbZ_{8}\times\mbZ_{8}$
and $V_{8,w}^{1}/\mbZ_{8}$ by a subgroup $\mbZ_{8}\subset\mbZ_{8}\times\mbZ_{8}$
are studied as examples of Calabi-Yau threefolds of non-trivial fundamental
groups. In particular, the Brauer group of $V_{8,w}^{1}$ was calculated
to study its relation to non-trivial fundamental group in \cite{GP}.
The following conjecture has arisen in these calculations:
\begin{conjecture}[Gross-Pavanelli \cite{Pav,GP} ]
\label{conj:GrossPav}Mirror of the Calabi-Yau manifold $V_{8,w}^{1}$
is given by $V_{8,w}^{1}/\mbZ_{8}$ with a subgroup $\mbZ_{8}\subset\mbZ_{8}\times\mbZ_{8}$.
Then mirror of the quotient $V_{8,w}^{1}/\mbZ_{8}$ is given by $V_{8,w}^{1}/\mbZ_{8}\times\mbZ_{8}$. 
\end{conjecture}
This conjecture was partly confirmed by showing that $V_{8,w}^{1}$
and $V_{8,w}^{1}/\mbZ_{8}\times\mbZ_{8}$ are derived equivalent \cite{Sch}
(see also \cite{Bak}) to each other; this is consistent to the fact
that these two Calabi-Yau manifolds have the same mirror manifold
$V_{8,w}^{1}/\mbZ_{8}$. In this paper, constructing families $\cV_{\mbZ_{8}}^{1}$
and $\cV_{\mbZ_{8}\times\mbZ_{8}}^{1}$ of Calabi-Yau manifolds for
$V_{8,w}^{1}/\mbZ_{8}$ with $\mbZ_{8}=\langle\tau\rangle$ and $V_{8,w}^{1}/\mbZ_{8}\times\mbZ_{8}$,
respectively, we will answer affirmatively to the above conjecture
(\textbf{Proposition \ref{prop:MS-Result-Main}}). 

Actually, mirror symmetry of Calabi-Yau manifold $V_{8,w}^{1}$ was
first studied\textit{ locally} near a special boundary point by Pavanelli
\cite{Pav} by calculating Gromov-Witten invariants of $V_{8,w}^{1}$
assuming that $V_{8,w}^{1}$ is self-mirror. We extend his local calculations
to global ones by making families of Calabi-Yau manifolds $\cV_{\mbZ_{8}}^{1}\rightarrow\mbP_{\Delta}$
and $\cV_{\mbZ_{8}\times\mbZ_{8}}^{1}\rightarrow\mbP_{\Delta}$ over
a toric variety $\mbP_{\Delta}$ of dimension two. We will find in
Section \ref{sec:P-Delta-global} that there are degeneration points
$A,B,C$ and $A',B',C'$ on a suitable resolution of $\mbP_{\Delta}$,
where we observe the following mirror correspondences:
\begin{equation}
\;A\;\;\leftrightarrow\;\;V_{8,w}^{1}\,\,,\quad B\;\;\leftrightarrow\;\;V_{8,w}^{1}/\mbZ_{8}\times\mbZ_{8}\,\,,\quad C\;\;\leftrightarrow\;\;V_{8,w}^{1}/\mbZ_{8}\,\label{eq:ABC-CY-intro}
\end{equation}
and $A',B',C'$ corresponding to birational models of each. We confirm
these correspondences by calculating Gromov-Witten invariants of stable
maps up to genus $g=2$ for $A,B$ and to $g=1$ for $C$. 

From the calculations of Gromov-Witten invariants, we will find that
the generating functions of these invariants are written in terms
of quasi-modular forms in a similar way to the case of rational elliptic
surfaces \cite{HSS,HST1}. For $M=A,B,C$, we introduce counting functions
by 
\[
Z_{g,n}^{M}(q)=\sum_{d\geq0}N_{g}^{M}(d,n)q^{d},\,\,\,\,\Big(Z_{g,n}^{M}(q)=\sum_{d\geq0}N_{g}^{M}(d,n)q^{2d}\,\,\text{for }C\Big)
\]
for Gromov-Witten invariants of Calabi-Yau manifolds which correspond
to $A,B,C$ by (\ref{eq:ABC-CY-intro}). We remark that Calabi-Yau
manifolds $V_{8,w}^{1},V_{8,w}^{1}/\mbZ_{8}\times\mbZ_{8}$ and $V_{8,w}^{1}/\mbZ_{8}$
have fibrations by abelian surfaces. Then the invariants $N_{g}^{M}(d,n)$
are related to counting numbers of genus $g$ curves of degree $d$
which intersect $n$ times with a fiber abelian surface, i.e., $n$-sections
if $g=0$ (cf. \cite{HSS}). 
\begin{conjecture}[\textbf{Observation \ref{obs:ZA_01}, equations (\ref{eq:Zg0B-A-relation}),
(\ref{eq:Zg0C-A-relation})}]
The generating functions $Z_{0,n}^{A}(q)$ have the following forms
\[
Z_{0,n}^{A}(q)=P_{0,n}^{A}(E_{2},E_{4},E_{6})\left(\frac{64}{\bar{\eta}(q)^{8}}\right)^{n},
\]
where $\bar{\eta}(q):=\Pi_{n\geq1}(1-q^{n})$ and $P_{0,n}^{A}$ are
quasi-modular forms of weight $4(n-1)$ in terms of Eisenstein series
$E_{2}(q),E_{4}(q)$ and $E_{6}(q)$ with $P_{0,1}^{A}=1$. The generating
functions $Z_{0,n}^{M}(q)$ for $M=B,C$ are given by 
\[
Z_{0,n}^{B}(q)=\frac{1}{64}Z_{0,n}^{A}(q^{8}),\quad Z_{0,n}^{C}(q)=\frac{1}{8}Z_{0,n}^{A}(q^{2}).
\]

\end{conjecture}
For higher genus calculations, we use the so-called BCOV holomorphic
anomaly equation \cite{BCOV1,BCOV2} to determine $Z_{g,n}^{M}(q)$.
For lower $g$ and $n$, we find the following forms of $Z_{g,n}^{M}(q)$,
which we state as a conjecture in general.
\begin{conjecture}[\textbf{Observations \ref{obs:Z11-A}, \ref{obs:Z11-B}, Conjectures
\ref{conj:Zgn-A}, \ref{conj:Zgn-B}}]
\label{conj:Main-R3}The generating functions $Z_{g,n}^{M}(q)$ $(M=A,B)$
are written by quasi-modular forms;
\[
Z_{g,n}^{A}(q)=P_{g,n}^{A}(E_{2},S,T,U)\left(\frac{64}{\bar{\eta}(q)^{8}}\right)^{n},\;\;\;Z_{g,n}^{B}(q)=P_{g,n}^{B}(E_{2},S,T,U)\left(\frac{1}{\bar{\eta}(q^{8})^{8}}\right)^{n},
\]
where $P_{g,n}^{A}$ and $P_{g,n}^{B}$ are polynomials of degree
2(g+n-1) of Eisenstein series $E_{2}$ and $S:=\theta_{3}(q)^{4},T:=\theta_{3}(q^{2})^{4},U:=\theta_{3}(q)^{2}\theta_{3}(q^{2})^{2}$
with the theta function $\theta_{3}(q)=\sum_{n\in\mbZ}q^{n^{2}}$. 
\end{conjecture}
We verify the above conjecture for $g=0,n\leq9$ and $g=1,n\leq3$.
Also, in Subsection \ref{sub:Fg2-AB}, we verify this for $g=2$ and
$n\leq2$ by using BCOV recursion relations. For $M=C$, genus two
calculations are slightly different from the cases $A$ and $B$.
Because of this, we couldn't determine unknown parameters completely.
However we observe from the calculations for $g=0,1$ that some simplifications
occur as in the following form (\textbf{Conjecture \ref{conj:Zgn-C}});
\[
Z_{g,n}^{C}(q)=P_{g,n}^{C}(E_{2}(q^{2}),E_{4}(q^{2}),E_{6}(q^{2}))\left(\frac{8}{\bar{\eta}(q^{2})^{8}}\right)^{n},
\]
where $P_{g,n}^{C}$ are quasi-modular forms of weight $4(g+n-1)$. 

It should be noted that Calabi-Yau manifolds $V_{8,w}^{1}$ and $V_{8,w}^{1}/\mbZ_{8}\times\mbZ_{8}$
are derived equivalent by fiberwise Fourier-Mukai transformations
\cite{Bak,Sch}. The generating functions $Z_{g,n}^{A}(q)$ and $Z_{g,n}^{B}(q)$
count $n$-sections (of genus $g$), hence under the Fourier-Mukai
transformations, these counting problems correspond to suitable moduli
problems associated with stable sheaves of rank $n$ on the dual Calabi-Yau
manifold. Our results on $Z_{g,n}^{A}(q)$ and $Z_{g,n}^{B}(q)$ suggest
that there exist nice moduli spaces of sheaves on both Calabi-Yau
manifolds $V_{8,w}^{1}$ and $V_{8,w}^{1}/\mbZ_{8}\times\mbZ_{8}$.
Actually, we will observe some simplifications in Conjecture \ref{conj:Main-R3}
when $n=1$ and for $g=0,1,2$; which we summarize in general as 
\begin{conjecture}[\textbf{Conjecture \ref{conj:Conj.Zg1-AB}}]
When $n=1$, Conjecture \ref{conj:Main-R3} simplifies to 
\[
Z_{g,1}^{A}(q)=P_{g,1}(E_{2},S,T,U)\frac{64}{\bar{\eta}(q)^{8}},\;\;\;Z_{g,1}^{B}(q)=P_{g,1}(E_{2},S,T,U)\frac{1}{\bar{\eta}(q^{8})^{8}},
\]
where $P_{g,1}=P_{g,1}^{A}=P_{g,1}^{B}$ is a polynomial of $E_{2},S,T$
and $U$ of degree $2g$. 
\end{conjecture}
The above two conjectures are reminiscent of the quasi-modular forms
we encountered in a similar geometric setting of rational elliptic
surfaces \cite{HST1}. In the latter case, we have a modular anomaly
equation which enables us to determine $P_{g,n}$ recursively and
also a nice closed formula for $Z_{g,1}(q)$ \cite[Thm.4.7]{HST2}.
However, it seems that things are more complicated in the present
case. 

Studying degenerations of the families over $A,B$ and $C$ is also
interesting, since we should be able to apply explicitly several ideals
of geometric mirror symmetry such as SYZ geometric mirror construction
\cite{SYZ} and Gross-Siebert program \cite{GS1,GS2}. We find in
\textbf{Propositions~\ref{prop:Degen-B},~\ref{prop:Degen-C},~\ref{prop:actions-on-ABC}}
that the type of degenerations are the same for $A,B$ and $C$, but
the group actions of $\mbZ_{8}\times\mbZ_{8}$ differ for these three
degenerations. Also in \textbf{Proposition \ref{prop:Conection-Matrices}},
we solve the connection problem of local solutions of Picard-Fuchs
equations at each degeneration point. We observe that the connection
matrix $U_{AB}$ has a simple interpretation from the fact that $A$
and $B$ correspond to Fourier-Mukai partners $V_{8,w}^{1}$ and $V_{8,w}^{1}/\mbZ_{8}\times\mbZ_{8}$,
respectively. More detailed and global analysis of the geometric mirror
symmetry via degenerations are left for future investigations.

~

Below we briefly describe the construction of this paper. In Section
2, we will summarize $(1,8)$-polarized abelian surfaces and their
embedding into $\mbP^{7}$, which give rise to the Calabi-Yau manifold
$V_{8,w}^{1}$. Known properties of $V_{8,w}^{1}$ are also summarized
to be used in the subsequent sections. In Section 3, we start with
{}{a complete intersection of four quadratics in $\mathbb{P}^{7}$
and define $V_{8,w}^{1}$ as a small resolution of it. Since $V_{8,w}^{1}$
contains parameters $w=[w_{0},w_{1},w_{2}]\in\mbP_{w}^{2}$, we obtain
a family $\mathcal{V}^{1}$ of $V_{8,w}^{1}$ over $\mathbb{P}_{w}^{2}$.}
We describe a certain symmetry of the four quadratic equations, and
show that there are two possible families $\cV_{\mbZ_{8}}^{1}\rightarrow\mbP_{\Delta}$
and $\cV_{\mbZ_{8}\times\mbZ_{8}}^{1}\rightarrow\mbP_{\Delta}$ as
a quotient of the family $\cV^{1}\rightarrow\mbP_{w}^{2}$ by the
symmetry. In Section 4, we describe discriminant loci of these families
over $\mbP_{\Delta}$ and find the degeneration points $A,B,C$ and
$A',B',C'$ of the families. In Section 5, we reproduce Picard-Fuchs
differential equations satisfied by period integrals of the families,
and find that they are identical for $\cV_{\mbZ_{8}}^{1}\rightarrow\mbP_{\Delta}$
and $\cV_{\mbZ_{8}\times\mbZ_{8}}^{1}\rightarrow\mbP_{\Delta}$. We
determine Gromov-Witten invariants ($g=0,1$) using mirror symmetry
near the boundary points $A$ and $A'$. In Section 6, noticing tangential
intersections of a component of the discriminant with a boundary divisor,
we find the degeneration point $B$. Applying mirror symmetry to $B$,
we find Gromov-Witten invariants ($g=0,1,2$) of $V_{8,w}^{1}/\mbZ_{8}\times\mbZ_{8}$.
In Section 7, we find the boundary points $C$ and $C'$ in $\mbP_{\Delta}$.
Calculating genus one Gromov-Witten invariants, and comparing the
potential function $F_{1}^{C}$ with $F_{1}^{A}$, we find that these
two points do not represent degenerations of the same family. The
results from Section 5 to Section 7 are summarized in Proposition
\ref{prop:MS-Result-Main}. In Section 8, we will calculate connection
matrices which relate the local solutions around $A,B,C$; and confirm
that the integral structures from $A,B$ and $C$ differ from each
other. For future investigations, we finally determine the degenerations
of Calabi-Yau manifolds (up to the quotients by finite groups) over
$A,B$ and $C$. A brief summary and related topics are discussed
in Section 9. In Appendices A to F, we describe formulas which we
use (and also derive) in the text. 

~

~

\noindent \textbf{Acknowledgements:} The authors would like to thank
Daisuke Inoue for explaining his recent results \cite{Inoue} to them.
S.H. would like to thank Atsushi Kanazawa and Daisuke Inoue for bringing
his attention to the work \cite{Pav}. {}{The authors
are grateful to the referees as well as editors for making variable
comments and suggestions which improved the presentation of this paper.}
This work is supported in part by Grant-in Aid Scientific Research
(C 20K03593, A 18H03668 S.H. and C 16K05090 H.T.).

~

\vskip3cm

\section{\textbf{Calabi-Yau manifolds fibered by abelian surfaces }}

\subsection{Calabi-Yau complete intersections}

Here we summarize minimal generalities on abelian surfaces with $(1,d)$
polarization and Calabi-Yau threefolds fibered by such surfaces following
the reference \cite{GP-CYI}. 

\para{} Let $(A,\mathcal{L})$ be a general (1,8) polarized abelian
surface. It is known \cite[\S 10.4]{BL1} that the linear system $|\mathcal{L}|$
admits an embedding of $A$ into $\mathbb{P}(H^{0}(\mathcal{L})^{\vee})=\mathbb{P}^{7}$
with its image of degree 16. There is a natural morphism from $A$
to the dual abelian surface $\hat{A}$, $\phi_{\mathcal{L}}:A\rightarrow\hat{A}$
defined by $x\mapsto t_{x}^{*}\mathcal{L}\otimes\mathcal{L}^{-1}$
with the translation $t_{x}$ by $x\in A$. Then the kernel $K(\mathcal{L})$
of $\phi_{\mathcal{L}}$ is isomorphic to $\mathbb{Z}_{8}\times\mathbb{Z}_{8}$.
Since we have $t_{x}^{*}\mathcal{L}\simeq\mathcal{L}$ for $x\in K(\mathcal{L})$,
one would expect the corresponding linear action of $H^{0}(\mathcal{L})$;
but actually the resulting action is given by a linear representation
of of a central extension of the group $K(\mathcal{L})$. The central
extension is known to be isomorphic to the Heisenberg group
\[
\cH_{8}=\langle\sigma,\tau\mid\sigma^{8}=\id,\tau^{8}=\id,[\sigma,\tau]=\sigma\tau\sigma^{-1}\tau^{-1}=\xi\id\rangle,
\]
and the linear representation acts on the homogeneous coordinates
$x_{i}$~(i=0,..,7) of $\mathbb{P}(H^{0}(\mathcal{L})^{\vee})$ as
\begin{equation}
\sigma(x_{i})=x_{i-1},\,\,\,\tau(x_{i})=\xi^{-i}x_{i}\,\,\,\,\,\,(\xi^{8}=1).\label{eq:sigma-tau-def}
\end{equation}

\para{} If we require $\mathcal{L}$ to be symmetric, i.e., $(-1)_{A}^{*}\mathcal{L}\simeq\mathcal{L}$,
then the vector space $H^{0}(\mathcal{L})$ becomes invariant under
the involution:
\[
\iota(x_{i})=x_{-i}\,\,\,\,\,\,\,\,\,\,(i\in\mathbb{Z}_{8})
\]
We decompose $\mathbb{C}^{8}$ of $\mathbb{P}^{7}=\mathbb{P}(\mathbb{C}^{8})$
into the $(+1)$ and $(-1)$ eigenspaces of this involution, and denote
the corresponding projective subspaces by $\mathbb{P}_{+}$ and $\mathbb{P}_{-}$.
The eigenspace $\mathbb{P}_{-}^{2}:=\mathbb{P}_{-}$ has the following
form,
\begin{equation}
\begin{aligned}\mathbb{P}_{-}^{2}=\left\{ [0,y_{1},y_{2},y_{3},0,-y_{3},-y_{2},-y_{1}]\in\mathbb{P}^{7}\right\} ,\end{aligned}
\label{eq:P2-}
\end{equation}
 and plays a role in defining the Calabi-Yau spaces which we will
study.

\para{} Let $\cH_{8}':=\langle\sigma^{4},\tau^{4}\rangle\subset\cH_{8}$,
and consider $\cH_{8}'$-invariant quadrics $H^{0}(\cO_{\mbP^{7}}(2))^{\cH_{8}'}$.
The group $\cH_{8}$ acts on these $\cH_{8}'$-invariant quadrics.
Then the space of invariants decomposes into three isomorphic and
irreducible 4-dimensional representations of $\cH_{8}$ as follows:
\[
H^{0}(\mathcal{O}_{\mathbb{P}^{7}}(2))^{\cH_{8}'}=\bigoplus_{i=0}^{2}\langle F_{i},\sigma F_{i},\sigma^{2}F_{i},\sigma^{3}F_{i}\rangle
\]
where 
\[
F_{0}=x_{0}^{2}+x_{4}^{2},\,\,\,\,\,\,F_{1}=x_{1}x_{7}+x_{3}x_{5},\,\,\,\,\,\,F_{2}=x_{2}x_{6}.
\]

\begin{prop}[{\cite[Remark 6.1]{GP-CYI}}]
 For each $y\in\mathbb{P}_{-}^{2}$, define polynomials in the subspace
of $H^{0}(\mathcal{O}_{\mathbb{P}^{7}}(2))^{\cH_{8}'}$ by 
\[
f:=y_{1}y_{3}F_{0}-y_{2}^{2}F_{1}+(y_{1}^{2}+y_{3}^{2})F_{2},\,\,\,\,\,\,\sigma(f),\,\,\,\,\sigma^{2}(f),\,\,\,\,\,\,\sigma^{3}(f).
\]
 These polynomials vanish along the $\cH_{8}$ orbit of $y\in\mathbb{P}_{-}^{2}$
in $\mathbb{P}^{7}$. 
\end{prop}
\para{} The group $\cH_{8}'$ acts on $\mathbb{P}_{-}^{2}$ as $\mathbb{Z}_{2}\times\mathbb{Z}_{2}$
generated by $(y_{1},y_{2},y_{3})\mapsto(-y_{3},-y_{2},-y_{1})$ and
$(y_{1},y_{2},y_{3})\mapsto(-y_{1},y_{2},-y_{3})$. We define the
quotient $\mathbb{P}_{\omega}^{2}=\mathbb{P}_{-}^{2}/\mathbb{Z}_{2}\times\mathbb{Z}_{2}$
by the relation 
\begin{equation}
[w_{0},w_{1},w_{2}]:=[2y_{1}y_{3},-y_{2}^{2},y_{1}^{2}+y_{3}^{2}].\label{eq:omega2y}
\end{equation}
 Using these invariants, we write the above four quadratic equations
as 

\begin{equation}
\begin{alignedat}{3} &  & f_{1}(\omega,x) &  &  & =\frac{w_{0}}{2}(x_{0}^{2}+x_{4}^{2})+w_{1}(x_{1}x_{7}+x_{3}x_{5})+w_{2}x_{2}x_{6},\\
 &  & f_{2}(\omega,x) &  &  & =\frac{w_{0}}{2}(x_{1}^{2}+x_{5}^{2})+w_{1}(x_{2}x_{0}+x_{4}x_{6})+w_{2}x_{3}x_{7},\\
 &  & f_{3}(\omega,x) &  &  & =\frac{w_{0}}{2}(x_{2}^{2}+x_{6}^{2})+w_{1}(x_{3}x_{1}+x_{5}x_{7})+w_{2}x_{4}x_{0},\\
 &  & f_{4}(\omega,x) &  &  & =\frac{w_{0}}{2}(x_{3}^{2}+x_{7}^{2})+w_{1}(x_{4}x_{2}+x_{6}x_{0})+w_{2}x_{5}x_{1}.
\end{alignedat}
\label{eq:def-eqs-f}
\end{equation}

\para{} Four quadratic equations define Calabi-Yau complete intersections
in $\mbP^{7}$. 
\begin{defn}
For each $w\in\mathbb{P}_{w}^{2}$, we define a variety 
\begin{equation}
V_{8,w}:=\left\{ f_{1}(w,x)=\cdots=f_{4}(w,x)=0\right\} \,\subset\,\mathbb{P}^{7}.\label{eq:defeq-V1}
\end{equation}
\end{defn}
\begin{thm}[{\cite[Theorem 6.5]{GP-CYI}}]
 For general $w\in\mathbb{P}_{w}^{2}$, the variety $V_{8,w}$ is
a $(2,2,2,2)$ complete intersection Calabi-Yau variety which is singular
exactly at 64 ODPs. These 64 ODPs are given by the $\cH_{8}$-orbit
of $y\in\mathbb{P}_{-}^{2}$ in $\mathbb{P}^{7}$ for $y$ given by
(\ref{eq:omega2y}) .
\end{thm}
For general $w\in\mathbb{P}_{w}^{2}$, the Calabi-Yau variety $V_{8,w}$
is a pencil of (1,8)-polarized abelian surfaces \cite[Theorem 6.7]{GP-CYI}.
Blowing up $V_{8,w}$ along a smooth (1,8)-polarized abelian surface
$A$, we obtain a small resolution $V_{8,w}^{2}\rightarrow V_{8,w}$
with 64 exceptional $\mathbb{P}^{1}$s for the 64 ODPs. We denote
by $V_{8,w}^{1}\rightarrow V_{8,w}$ the small resolution obtained
by flopping the 64 $\mathbb{P}^{1}$s in $V_{8,w}^{2}$. \textcolor{black}{We
refer to \cite[p.213]{GP-CYI} for details.}
\begin{thm}[{\cite[Theorem 6.9]{GP-CYI}}]
 Let $V_{8,w}^{1}\rightarrow V_{8,w}$ be the small resolution above
for a general $w\in\mathbb{P}_{w}^{2}$. Then \end{thm}
\begin{enumerate}
\item $V_{8,w}^{1}$ has a fibration over $\mathbb{P}^{1}$ with the fiber
$(1,8)$-polarized abelian surfaces,
\item the Hodge numbers are given by $h^{1,1}(V_{8,w}^{1})=h^{2,1}(V_{8,w}^{1})=2$.
\end{enumerate}
It is known that the resolution $V_{8,w}^{2}$ also has an abelian
surface fibration over $\mathbb{P}^{1}$. 

\para{} \label{para:divisors- H1A1-and-H2A2}When we will discuss
mirror symmetry, we will need the cubic forms and also linear forms
on $H^{2}$. Here we summarize these topological invariants for $V_{8,w}^{i}$. 

Let $H_{1}$ be the pullback of of the hyperplane section of $V_{8,w}$,
and $A_{1}$ be the class of a fiber abelian surface in $V_{8,w}^{1}$.
Then we have 
\begin{equation}
\begin{aligned}H_{1}^{3}=H_{1}^{2}A_{1}=16, &  &  & H_{1}A_{1}^{2}=A_{1}^{3}=0\\
c_{2}(V_{8,w}^{1})H_{1}=64, &  &  & c_{2}(V_{8,w}^{1})A_{1}=0.
\end{aligned}
\label{eq:yukawaV1}
\end{equation}
\textcolor{black}{where $c_{2}(V_{8,w}^{1})H_{1}=64$ follows from
the Riemann-Roch theorem and $c_{2}(V_{8,w}^{1})A_{1}=0$ follows
since $A_{1}$ is the class of a fiber of an abelian fibration.} Also
the ample cone $Amp(V_{8,w}^{1})\subset H^{2}(V_{8,w}^{1},\mathbb{R})$
is generated by $H_{1},A_{1}$. 

For the other resolution $V_{8,w}^{2}$, we write the birational map
$\phi:V_{8,w}^{2}\dashrightarrow V_{8,w}^{1}$ and define the pullbacks
by $H_{2}=\phi^{*}(H_{1})$ and $A_{2}=\phi^{*}(A_{1})$, which generate
$H^{2}(V_{8,w}^{1},\mathbb{R})$. \textcolor{black}{Then the cubic
forms and linear forms of $V_{8,w}^{2}$ are} given by\textcolor{black}{
\begin{equation}
\begin{aligned}H_{2}^{3}=H_{2}^{2}A_{2}=16, &  &  & H_{2}A_{2}^{2}=0,\,\,\,\,A_{2}^{3}=-64\\
c_{2}(V_{8,w}^{2})H_{2}=64, &  &  & c_{2}(V_{8,w}^{2})A_{2}=128,
\end{aligned}
\label{eq:yulawaV2}
\end{equation}
since we have $A_{2}^{3}=A_{1}^{3}-n_{0}$ in terms of the number
$n_{0}=64$ of flopping curves. The number $128$
in the second line follows from a relation $c_{2}(V_{8,w}^{2})A_{2}=c_{2}(V_{8,w}^{1})A_{1}+2\,n_{0}$
which we derive from Riemann-Roch theorem by showing 
$\chi(\mathcal{O}_{V_{8,w}^{2}}(A_{2}))=\chi(\mathcal{O}_{V_{8,w}^{1}}(A_{1}))$.
The ample cone of $V_{8,w}^{2}$ is generated by} 
\[
H_{2},\,\,\,\,\tilde{A_{2}}:=2H_{2}-A_{2},
\]
\textcolor{black}{(see \cite[Prpo.6.14]{GP-CYI}}) for which we have
$H_{2}^{3}=H_{2}^{2}\tilde{A}_{2}=16$, $H_{2}\tilde{A}_{2}^{2}=\tilde{A_{2}^{3}}=0$
and also $c_{2}(V_{8,w}^{2})H_{2}=64$, $c_{2}(V_{8,w}^{2})\tilde{A}_{2}=0$.

\subsection{More on the small resolutions $V_{8,w}^{1}$ and $V_{8,w}^{2}$}

We summarize known properties on the abelian surface fibration $X:=V_{8,w}^{1}\rightarrow\mathbb{P}^{1}$. 

\para{} \label{para: V1-summary} About the fibration $V_{8,w}^{1}\rightarrow\mathbb{P}^{1}$
for a general $w\in\mathbb{P}_{w}^{2}$, the following facts are known
in \cite{GP-CYI,GP};
\begin{enumerate}
\item There are exactly $64$ sections $\sigma_{k}$ given by the exceptional
curves of the flop.
\item Every smooth fiber $F$ is a $(1,8)$-polarized abelian surface with
its polarization $\mathcal{L}=\cO_{V_{8,w}^{1}}(H_{1})\vert_{F}$.
The 64 points $F\cap\sigma_{k}$ are exactly the kernel $K(\mathcal{L})$
of the polarization $\mathcal{L}$.
\item There are exactly 8 singular fibers, each of which is the elliptic
translation scroll obtained from an elliptic normal curve $E$ in
$\mathbb{P}^{7}$ with a point $e\in E$. 
\item The group $\cH_{8}$ acts freely on $V_{8,w}^{1}$, making the 64
sections into a single orbit. The action restricted on each smooth
fiber coincides with that of the kernel $K(\mathcal{L})$ of the polarization.
On a singular fiber, which is a translation scroll over an elliptic
curve $E$, the action is represented by the natural translations
by an 8-torsion point of $E$. 
\end{enumerate}
{}{In this paper, we denote by $H_{X}:=H_{1}$ and $A_{X}:=A_{1}$,
respectively, the restriction of the hyperplane class of $\mbP^{7}$
and the fiber class of the fibration $X=V_{8,w}^{1}\rightarrow\mbP^{1}$.
Also we denote by $\sigma_{X}$ and $\ell$ one of the 64 sections
and a line in a singular fiber, respectively. Then, from the relations
\[
H_{X}.\ell=1,\,\,H_{X}.\sigma_{X}=0;\quad A_{X}.\ell=0,\,\,A_{X}.\sigma=1,
\]
we see that $H_{X}$ and $A_{X}$ generate $Pic(V_{8,w}^{1})$ modulo
torsions. We denote by $E_{X}$ the class of an elliptic curve in
(3) above. Then we have $H_{X}.E_{X}=8$ since elliptic curves $E$
in (3) are $\mathcal{H}_{8}$-invariant curves of degree $8$ in $\mathbb{P}^{7}$
\cite[Thm.3.1]{GP-A}. }

We note that the action of $\cH_{8}$ on $\mathbb{P}^{7}$ may be
regarded as that of $\mathbb{Z}_{8}\times\mathbb{Z}_{8}$ generated
by $\sigma,\tau$.

\para{} We take a subgroup $\mathbb{Z}_{8}\subset\mathbb{Z}_{8}\times\mathbb{Z}_{8}$.
Since $\mbZ_{8}\times\mbZ_{8}$ acts freely on $V_{8,w}^{1}$, we
have three Calabi-Yau manifolds
\begin{equation}
V_{8,w}^{1},\,\,\,\,\,V_{8,w}^{1}/\mathbb{Z}_{8}\times\mathbb{Z}_{8}\,\,\,\text{and}\,\,V_{8,w}^{1}/\mathbb{Z}_{8}\label{eq:3CYquotients}
\end{equation}
with the same hodge numbers. As we summarized in Conjecture \ref{conj:GrossPav},
it is conjectured in \cite{GP,Pav} that these three Calabi-Yau manifolds
are related by mirror symmetry. This conjecture has nicely been supported
by the following theorem:
\begin{thm}[{\cite[Theorem 4.1]{Sch}}]
The two Calabi-Yau manifolds $V_{8,w}^{1}$ and $V_{8,w}^{1}/\mathbb{Z}_{8}\times\mathbb{Z}_{8}$
are derived equivalent.
\end{thm}
Note that this theorem is consistent with Conjecture \ref{conj:GrossPav}
from the viewpoint of homological (or categorical) mirror symmetry
\cite{Ko}. In this paper, we will find that the subgroup $\mbZ_{8}=\langle\tau\rangle\subset\mbZ_{8}\times\mbZ_{8}$
is a suitable choice for the conjecture to hold. Then, we will show
an affirmative answer to Conjecture \ref{conj:GrossPav} by finding
degenerations of Calabi-Yau manifolds $V_{8,w}^{1}/\mbZ_{8}$ and
$V_{8,w}^{1}/\mbZ_{8}\times\mbZ_{8}$ where the conjectured mirror
symmetry arises. 

\vskip3cm

\section{\textbf{Families }$\protect\cV^{1},\,\,\protect\cV^{1}/\mathbb{Z}_{8},$
\textbf{and $\protect\cV^{1}/\mathbb{Z}_{8}\times\mathbb{Z}_{8}$
over }$\mathbb{P}_{\omega}^{2}$}

{}{The small resolutions $V_{8,w}^{1}$ of the $(2,2,2,2)$
complete intersection (\ref{eq:defeq-V1}) in $\mbP^{7}$ form a family
$\mathcal{V}^{1}$ of over an open set of $\mathbb{P}_{w}^{2}$. In
what follows, we simply write this family by $\mathcal{V}^{1}\rightarrow\mathbb{P}_{w}^{2}$
with understanding that the actual family is defined over the set
of general points of $\mathbb{P}_{w}^{2}$. The projective space $\mathbb{P}_{w}^{2}$
here should be considered as a compactification of the parameter space
of the family. }

\subsection{Symmetry of the family $\mathcal{V}^{1}\rightarrow\mathbb{P}_{\omega}^{2}$ }

As described in \ref{para: V1-summary}, the Heisenberg group $\cH_{8}$
(or $\mathbb{Z}_{8}\times\mathbb{Z}_{8}$) acts on each fiber of the
family $\mathcal{V}^{1}\rightarrow\mathbb{P}_{w}^{2}$. Actually,
this action extends as a symmetry of the family to a larger group
$\cN\cH_{8}$ in 
\begin{equation}
1\rightarrow\cH_{8}\rightarrow\cN\cH_{8}\rightarrow SL_{2}(\mathbb{Z}_{8})\rightarrow1.\label{eq:NH8-def}
\end{equation}
Here the group $\cN\cH_{8}$ is the normalizer of $\cH_{8}$ in $GL(\mbC^{8})$,
and is generated by 
\[
S=\frac{1}{2\sqrt{2}}\left(\xi^{-\frac{1}{2}(i-j)^{2}}\right)_{0\leq i,j\leq7},\,\,\,T=\frac{1}{2\sqrt{2}}\left(\xi^{-ij}\right)_{0\leq i,j\leq7}
\]
and $\sigma,\tau\in\cH_{8}$, where $\xi=e^{2\pi i/8}$. The adjoint
action of $g\in\cN\cH_{8}$ on $\cH_{8}/[\cH_{8},\cH_{8}]\simeq\mbZ_{8}\times\mbZ_{8}$
by $\sigma^{a}\tau^{b}\mapsto g\sigma^{a}\tau^{b}g^{-1}$ gives the
surjective homomorphism to $SL_{2}(\mathbb{Z}_{8})$, which is described
by $S\mapsto\mathtt{S}:=\left(\begin{smallmatrix}1 & 1\\
0 & 1
\end{smallmatrix}\right)$ and $T\mapsto\mathtt{T}:=\left(\begin{smallmatrix}0 & -1\\
1 & 0
\end{smallmatrix}\right)$ for $S,T\in\mathcal{NH}_{8}$. {}{Let $g.x$ represent
the linear action of $g\in\cN\cH_{8}$ on $x\in\mathbb{C}^{8}$. We
define the following representation $\rho$ which comes from the linear
actions of $S,T\in\cN\cH_{8}$ on $y\in\mathbb{P}_{-}^{2}\subset\mathbb{P}^{7}=\mathbb{P}(\mathbb{C}^{8})$
through the relation (\ref{eq:omega2y}). }
\begin{defn}
\label{def:rho-matrix}We define a representation $\rho:\cN\cH_{8}/\cH_{8}\rightarrow GL(\mathbb{C}^{3})$
by

\[
\rho(S)=\frac{\xi}{2}\left(\begin{matrix}\xi^{7} & 2 & \xi^{3}\\
1 & 0 & 1\\
\xi^{3} & 2 & \xi^{7}
\end{matrix}\right),\,\,\,\rho(T)=\frac{\xi^{6}}{2}\left(\begin{matrix}1 & 2 & 1\\
1 & 0 & -1\\
1 & -2 & 1
\end{matrix}\right),\,\,\,\,\rho(h)=\id\,\,\text{for }h\in\cH_{8}.
\]
\end{defn}
\begin{rem}
{}{(1) It should be noted that we impose the condition
$\rho(h)=\mathrm{id}$ for $h\in\mathcal{H}_{8}$. (2) Since the relation
(\ref{eq:omega2y}) is a projective relation, we can only determine
the matrices $\rho(S)$ and $\rho(T)$ up to constant factors, say
$a,b,$ for these matrices respectively. We write $\rho(S)=a\,\rho(S)_{0},\rho(T)=b\,\rho(T)_{0}$
with normalizing $\rho(S)_{0}$ and $\rho(T)_{0}$ so that their (1,2)
entries are equal to 1. By the condition $\rho(\cH_{8})=\id$, we
can find that $(a,b)=(\xi,\xi^{6}),(\xi^{3},1),(\xi^{5},\xi^{2})$
and $(\xi^{7},\xi^{4})$ are the only possible values for the constants
(see \cite{HT-math-c} for calculations). Here, depending the choice
of $(a,b)$, the image of $\rho$ varies; it is isomorphic to $SL_{2}(\mathbb{Z}_{8})/(\mathbb{Z}_{8})^{\times}$
for $(\xi^{3},1),(\xi^{7},\xi^{4})$ and $SL_{2}(\mathbb{Z}_{8})/\left\{ 1,5\right\} $
for $(\xi,\xi^{6}),(\xi^{5},\xi^{2})$ (see the proposition below
for this result and the definition of $\{1,5\}$). In the above definition,
we have chosen $(\xi,\xi^{6})$ so that we have a larger image.}$\hfill\square$
\end{rem}
The following proposition describes the action of $\cN\cH_{8}$ on
the family $\mathcal{V}^{1}$. 
\begin{prop}
\label{prop:NH-action-on-f} ~~

\begin{myitem}

\item{$(1)$} The group $\cN\cH_{8}$ acts linearly on the defining
equations (\ref{eq:def-eqs-f}) by 
\begin{equation}
f_{i}(\rho(g).\omega,g.x)=\sum_{j}c_{ij}(g)f_{j}(\omega,x),\label{eq:g-act-f}
\end{equation}
where $g.x$ represents the natural linear action of $g\in\cN\cH_{8}$
on $x\in\mathbb{C}^{8}$, and the group homomorphism $R(g)=\left(c_{ij}(g)\right)$
is determined by 
\begin{equation}
\left(\begin{smallmatrix}0 & 0 & 0 & 1\\
1 & 0 & 0 & 0\\
0 & 1 & 0 & 0\\
0 & 0 & 1 & 0
\end{smallmatrix}\right),\left(\begin{smallmatrix}1 & 0 & 0 & 0\\
0 & \xi^{6} & 0 & 0\\
0 & 0 & \xi^{4} & 0\\
0 & 0 & 0 & \xi^{2}
\end{smallmatrix}\right),\frac{\xi}{2}\times\left(\begin{smallmatrix}\xi^{7} & \xi^{6} & \xi^{3} & \xi^{6}\\
\xi^{6} & \xi^{7} & \xi^{6} & \xi^{3}\\
\xi^{3} & \xi^{6} & \xi^{7} & \xi^{6}\\
\xi^{6} & \xi^{4} & \xi^{6} & \xi^{7}
\end{smallmatrix}\right),\,\frac{\xi^{6}}{2}\times\left(\begin{smallmatrix}1 & 1 & 1 & 1\\
1 & \xi^{6} & \xi^{4} & \xi^{2}\\
1 & \xi^{4} & 1 & \xi^{4}\\
1 & \xi^{2} & \xi^{4} & \xi^{6}
\end{smallmatrix}\right)\label{eq:R-matrix-apxA}
\end{equation}
for $g=\sigma,\tau,S,T$, respectively, with the relation $R(gh)=R(g)R(h)$.

\item{$(2)$} The kernel of $\rho:\cN\cH_{8}\rightarrow GL(\mathbb{C}^{3})$
is given by $\left\langle \cH_{8},u_{1},u_{5}\right\rangle $ with
$u_{i}:=(S^{8-i}T)^{3}$. \item{$(3)$} The image of $\rho:\cN\cH_{8}\rightarrow GL(\mathbb{C}^{3})$
is isomorphic to $SL_{2}(\mathbb{Z}_{8})/\left\{ 1,5\right\} $, where
$\{1,5\}$ is the subgroup of the units $(\mathbb{Z}_{8})^{\times}=\left\{ 1,3,5,7\right\} $
{}{with $(\mathbb{Z}_{8})^{\times}\,\mathrm{id}_{2}\subset SL_{2}(\mathbb{Z}_{8})$.}

\end{myitem}\end{prop}
\begin{proof}
(1) We verify the claim by direct evaluations which we leave to the
reader (see \cite{HT-math-c}). (2),(3) By definition, there is a
group homomorphism $SL_{2}(\mathbb{Z}_{8})\simeq\mathcal{NH}_{8}/\mathcal{H}_{8}\rightarrow\mathrm{Im}\,\rho$.
The units $(\mathbb{Z}_{8})^{\times}\subset SL_{2}(\mathbb{Z}_{8})$
are written by $(\mathtt{S}^{7}\mathtt{T})^{3},(\mathtt{S}^{5}\mathtt{T})^{3},(\mathtt{S}^{3}\mathtt{T})^{3}$
$(\mathtt{S}\mathtt{T})^{3}$ for $1,3,5,7$, respectively. Evaluating
the corresponding matrices, we find that $\rho((S^{7}T)^{3})=\rho((S^{3}T)^{3})=\id$.
Then by counting the elements in the image directly as $384/2$$(=|SL_{2}(\mathbb{Z}_{8})/\left\{ 1,5\right\} |$)
in \cite{HT-math-c}, we conclude the claim. \end{proof}
\begin{defn}
We introduce the factor group $G_{\rho}:=\cN\cH_{8}/\left\langle \cH_{8},u_{1},u_{5}\right\rangle $
by 
\[
1\rightarrow\left\langle \cH_{8},u_{1},u_{5}\right\rangle \rightarrow\cN\cH_{8}\overset{q}{\rightarrow}G_{\rho}\rightarrow1
\]
with ${{}u_{i}:=(S^{8-i}T)^{3}}$. The group $G_{\rho}$
is isomorphic to $\mathrm{Im}\,\rho\subset GL(\mathbb{C}^{3})$ and
also to $SL_{2}(\mbZ_{8})/\left\{ 1,5\right\} $. 
\end{defn}
The symmetry relation (\ref{eq:g-act-f}) clearly entails the isomorphisms
\begin{equation}
V_{8,\omega}\simeq V_{8,\rho(g)\omega},\label{eq:V=00003DgV}
\end{equation}
which were observed in \cite{GP-CYI}.

\subsection{Degenerations of the family $\protect\cV^{1}$}

A general fiber $V_{8,\omega}^{1}$ of the family $\mathcal{V}^{1}\rightarrow\mathbb{P}_{\omega}^{2}$
has a fibration over $\mathbb{P}^{1}$ by (1,8)-polarized abelian
surfaces. Gross and Popescu \cite{GP-CYI} describe the family $\mathcal{V}^{1}\rightarrow\mathbb{P}_{\omega}^{2}$
as a fibration of (1,8) abelian surfaces over a conic bundle over
$\mathbb{P}_{w}^{2}$. Studying degenerations of the family, they
showed rationality of the moduli space $\mathcal{M}_{(1,8)}^{lev}$
of (1,8)-polarized abelian surfaces. Here for our later purposes,
we summarize their description on the discriminant loci of the family. 

The variety $V_{8,\omega}$ is given as a complete intersection of
four quadrics in $\mathbb{P}^{7}$. Depending on the degenerations
of the quadrics, the following three different components of the discriminant
are recognized:
\begin{equation}
\begin{aligned} & D_{s}=\left\{ 2w_{1}^{4}-w_{0}w_{2}(w_{0}^{2}+w_{2}^{2})=0\right\} ,\,\,\,\,L_{1}=\left\{ w_{1}(w_{0}^{2}-w_{2}^{2})=0\right\} ,\\
 & \;\;\;\;\;\;L_{2}=\left\{ \left((w_{0}+w_{2})^{4}-(w_{0}-w_{2})^{4}\right)dis_{0}=0\right\} ,
\end{aligned}
\label{eq:disDsL1L2}
\end{equation}
where we define 
\[
dis_{0}:=\left((w_{0}+w_{2})^{4}-(2w_{1})^{4}\right)\left((w_{0}-w_{2})^{4}+(2w_{1})^{4}\right).
\]
The group $\cN\cH_{8}$ acts on these discriminant loci through the
representation $\rho$ in Proposition \ref{prop:NH-action-on-f}.
It is easy to see that these three components are invariant under
the actions $S$ and $T$. While $D_{s}$ is irreducible, $L_{1}$
and $L_{2}$ consist of lines which are exchanged under $S$ and $T$
. Then, due to the symmetry relation (\ref{eq:V=00003DgV}), it is
sufficient to see the degenerations over $D_{s}$, and the lines $\left\{ w_{1}=0\right\} $,
$\left\{ w_{0}=0\right\} $ in $L_{1}$ and $L_{2}$, respectively. 
\begin{prop}
\label{prop:V1-degeneration-loci}Over the three discriminant loci,
$V_{8,\omega}$ degenerates as follows:\begin{myitem}

\item{$(1)$} Over a general point on $L_{1}$, $V_{8,w}$ degenerates
to a join of two elliptic quartic normal curves.

\item{$(2)$} Over a general point on $L_{2}$, $V_{8,w}$ has 72$($=64+8$)$
ordinary points but still possesses a pencil of abelian surfaces.

\item{$(3)$} Over a general point on $D_{s}$, the conic bundle over
$\mathbb{P}_{w}^{2}$ degenerates and a pencil of abelian surfaces
of $V_{8,w}$ breaks down accordingly.

\end{myitem}\end{prop}
\begin{proof}
We refer \cite[Thm. 6.8]{GP-CYI} for the proofs.
\end{proof}
The reason why we extract the 8 lines in $\left\{ dis_{0}=0\right\} $
from the 12 lines in $L_{2}=:L_{2}'\cup\left\{ dis_{0}=0\right\} $
is based on the following property:
\begin{prop}
\label{prop:ODP-8}Over general points on the 8 lines in $\left\{ dis_{0}=0\right\} $,
the additional 8 ordinary double points in Proposition \ref{prop:V1-degeneration-loci}
(2) are in a single orbit of $\mbZ_{8}\simeq\langle\tau\rangle$.\end{prop}
\begin{proof}
Our proof is based on explicit calculations of the additional 8 ordinary
double points (ODP) for each lines. For example, for a general point
on the line $(w_{0}+w_{2})+2w_{1}=0$ (resp. $i\sqrt{i}(w_{0}-w_{2})+2w_{1}=0$),
we find, by evaluating the Jacobian ideal of the complete intersections,
that the additional 8 ODPs are given by 
\[
\mbZ_{8}\cdot[1,1,1,1,1,1,1,1]\,\,\,\,\,(\text{resp. }\mbZ_{8}\cdot[1,\eta,-i,\eta,1,-\eta,-i,-\eta])
\]
where $\mbZ_{8}=\langle\tau\rangle$ and $\eta:=(-1)^{\frac{1}{8}}$. 
\end{proof}
We can also verify that on the 4 lines in $L_{2}'$, the numbers of
the $\tau$-orbits vary. We will interpret the above proposition when
we will calculate genus one Gromov-Witten invariants in Sections \ref{sec:MSbyPF-A},
\ref{sec:MS-B} and \ref{sec:MS-more-C}.

\subsection{Parameter space of the family $\mathcal{V}^{1}$}

It is easy to see that the variety $V_{8,\omega}^{1}$ degenerates
to 16 $\mathbb{P}^{3}$s (see Subsection \ref{sub:degen-A-B} for
details) at the point of intersection $\left\{ w_{0}=0\right\} \cap\left\{ w_{1}=0\right\} $.
This type of degenerations are hallmarks for mirror symmetry of a
family to its mirror Calabi-Yau manifolds. In order to study degenerations
of this type, we will describe subgroups of $\cN\cH_{8}$ which act
on the family $\cV^{1}\rightarrow\mbP_{w}^{2}$.

\para{Abelian subgroup $G_0$.} \label{para:G0} Knowing that a special
degeneration appears at the point $[0,0,1]\in\mathbb{P}_{w}^{2}$,
let us introduce the following subgroup $G_{0}\subset G_{\rho}$. 
\begin{defn}
We denote the isotropy subgroup of the point $[0,0,1]$ by 
\[
F_{0}=\left\{ g\in\cN\cH_{8}\mid\rho(g).[0,0,1]=[0,0,1]\right\} ,
\]
and define the corresponding subgroup $G_{0}:=F_{0}/\left\langle \cH_{8},u_{1},u_{5}\right\rangle $
in $G_{\rho}$. 
\end{defn}
To describe $G_{0}$ as a subgroup of $G_{\rho}=\cN\cH_{8}/\langle\cH_{8},u_{1},u_{5}\rangle$,
let us denote the classes of $S$ and $T$ by $\bar{S}$ and $\bar{T}$,
respectively. They may be written by $\bar{S}=\rho(S)$ and $\bar{T}=\rho(T)$
under the isomorphism $G_{\rho}\simeq\mathrm{Im}\,\rho$ in Proposition
\ref{prop:NH-action-on-f} (3). 
\begin{prop}
\label{prop:The-group-G0}The group $G_{0}$ is described by $G_{0}=\left\langle \bS\bT\bS,(\bS\bT)^{3}\right\rangle \simeq\mbZ_{8}\times\mbZ_{2}$
with 
\[
\bS\bT\bS=\left(\begin{smallmatrix}\xi^{7} & 0 & 0\\
0 & 1 & 0\\
0 & 0 & \xi^{3}
\end{smallmatrix}\right),\,\,\,(\bS\bT)^{3}=-1
\]
 and fixes all coordinate points $[1,0,0],[0,1,0]$,$[0,0,1]$ of
$\mbP_{w}^{2}$.\end{prop}
\begin{proof}
It is straightforward to verify the claimed properties by using the
matrix forms $\bar{S}=\rho(S),\bar{T}=\rho(T)$. See \cite{HT-math-c}. 
\end{proof}
The group $G_{0}$ was considered in \cite{Pav} to define a local
quotient $\mathbb{C}^{2}/G_{0}$ of the affine space $\mathbb{C}^{2}\subset\mathbb{P}_{w}^{2}$
centered at $[0,0,1]$. We will consider this quotient globally for
the family $\cV^{1}\rightarrow\mbP_{w}^{2}$ in the next section. 

\para{Subgroup $G_{max}$.} A slightly larger group $G_{max}$ arises
naturally. To describe it, we recall our definition $R(g)=\left(c_{ij}(g)\right)_{1\leq i,j\leq4}$
in (\ref{eq:g-act-f}). 
\begin{defn}
We define 
\[
F_{max}:=\left\{ g\in\cN\cH_{8}\mid R(g)\,\text{ is a projectively permutation matrix}\right\} ,
\]
where by ``projectively permutation matrix'' we mean a permutation
matrix with non-vanishing entries 1 replaced by non-vanishing complex
numbers.
\end{defn}
From the matrices $R(g)$ given in (\ref{eq:R-matrix-apxA}), we see
that $\cH_{8}\subset F_{mat}$. It is also easy to verify that $R(u_{i})=(-1)^{(i-1)/2}\id$
for the units $u_{i}$ ($i=1,3,5,7$). Hence we have the following
subgroup of $G_{\rho}$:
\[
G_{max}:=F_{max}/\left\langle \cH_{8},u_{1},u_{5}\right\rangle .
\]
One may note that the group $F_{max}$ (or $G_{max})$ is the largest
group which preserves projectively the form of period integral (\ref{eq:period-Int-def})
which we will study in the following sections. 
\begin{prop}
\label{prop:Gmax}We have $G_{max}=\left\langle \bS\bT\bS,(\bS\bT)^{3},\bS^{4}\right\rangle $
and $G_{0}$ is a normal subgroup of $G_{max}$ with index two.\end{prop}
\begin{proof}
Once we find generators of $G_{max}$, the claimed properties are
easy to verify; e.g., the index follows immediately from $G_{max}/G_{0}=\left\langle \bS^{4}\right\rangle $
and $\bS^{4}=\left(\begin{smallmatrix}0 & 0 & -1\\
0 & -1 & 0\\
-1 & 0 & 0
\end{smallmatrix}\right)$. We find generators in \cite{HT-math-c} and verifies the normality
there.
\end{proof}
\para{Other subgroups of $\cN\cH_8$.} For our analysis of the family
$\mathcal{V}^{1}$ over $\mathbb{P}_{w}^{2}$, we will mostly consider
the corresponding family over $\mathbb{P}_{w}^{2}/G_{0}$. However,
for completeness, let us introduce two more subgroups of $G_{\rho}$;
\begin{equation}
G_{1}=\left\langle \bS^{2},\bT\bS^{2}\bT,(\bS\bT)^{3}\right\rangle ,\,\,\,\,\,\,G_{1}^{ex}:=\left\langle \bS^{2},\bT\bS\bT\right\rangle .\label{eq:G1-ex}
\end{equation}
We can verify directly that $G_{1}$ is a normal subgroup $G_{1}^{ex}$
with index two $(|G_{1}|=32)$. We also verify that $G_{max}$ is
a normal subgroup of $G_{1}^{ex}=\left\langle G_{max},\bS^{2}\right\rangle $
with index two. We summarize the containments of these groups in the
following diagram:
\begin{equation}
\begin{matrix}\xymatrix{ &  & \ar@{^{(}->}[dl]\,\,\,\left\{ e\right\} \,\,\,\ar@{_{(}->}[dr]\\
 & \ar@{^{(}->}[d]^{\mathbb{Z}_{2}}\,\,\,\,\,\,\,G_{0}\,\,\,\, &  & \ar@{_{(}->}^{\mathbb{Z}_{2}}[d]\ar@{_{(}->}[dr]^{\frak{S}_{3}}\;\;G_{1}\;\;\\
 & \,\,\,\,\,G_{max}\,\,\,\,\ar@{^{(}->}[rr]^{\mathbb{Z}_{2}} &  & \;\;G_{1}^{ex}\;\;\ar@{^{(}->}[r] & \;\;G_{\rho} & \hsp{-30}\simeq SL_{2}(\mathbb{Z}_{8})/\left\{ 1,5\right\} 
}
\end{matrix}\label{eq:Group-Diag}
\end{equation}
In the above diagram, $\xymatrix{H\ar@{^{(}->}[r]^{K} & G}
$ indicates that $H\vartriangleleft G$ with the factor group $G/H\simeq K$. 

\para{Invariants of $G_1$.} Since the group $G_{1}$ is a finite
group, it is easy to derive the following proposition (see \cite{HT-math-c}
for the derivations):
\begin{prop}
\label{prop:invariants-G1}For the group $G_{1}$, the following properties
hold:

\begin{myitem}

\item{$(1)$} When acting on $\mathbb{C}[w_{0},w_{1},w_{2}]$ through
the representation $\rho$, we have 
\[
\mathbb{C}[w_{0},w_{1},w_{2}]^{G_{1}}=\mathbb{C}[(w_{0}+w_{2})^{4},(w_{0}-w_{2})^{4},(2w_{1})^{4},(2w{}_{1})^{2}(w_{0}^{2}-w_{2}^{2})^{2}].
\]
\item{$(2)$} $G_{\rho}$ acts as permutations on the first three
invariants in $(1)$, 
\[
A:=(w_{0}+w_{2})^{4},\,\,B:=-(w_{0}-w_{2})^{4},\,\,C:=-(2w_{1})^{4},
\]
and on the fourth invariant $E:=(2w{}_{1})^{2}(w_{0}^{2}-w_{2}^{2})^{4}$
as the sign representation of $\frak{S}_{3}$.

\end{myitem}
\end{prop}
We count the homogeneous degrees of the invariants $A,B,C$ and $E$
are $4$ and $6$, and note that they satisfy a single relation $E^{2}-ABC=0$.
Using this relation, we see the isomorphism $\mathrm{Proj}\,\mbC[A,B,C,E]\simeq\mbP^{2}$. 

\para{Quotients of ${\mathbb P}_w^2$.} To each subgroup $G$ in the
diagram (\ref{eq:Group-Diag}), we have the corresponding quotient
$\mathbb{P}_{w}^{2}/G$. We will study in detail the quotient $\mbP_{\Delta}:=\mbP_{w}^{2}/G_{0}$
in the next section. Here we briefly describe quotients by other groups
to see their relations to $\mbP_{\Delta}$. First, using the invariants
$A,B,C$ (and $E$) in Proposition \ref{prop:invariants-G1} (1),
we have $\mathbb{P}_{w}^{2}/G_{1}\simeq\mathbb{P}^{2}$. Then, noting
that the factor group $G_{1}^{ex}/G_{1}\simeq\mathbb{Z}_{2}$ acts
on $A,B,C$ by exchanging $A$ and $B$, and making the invariants
$(A-B)^{2}$ and $A+B$, we arrive at the quotient $\mathbb{P}_{w}^{2}/G_{1}^{ex}\simeq\mathbb{P}(2,1,1)$.
Similarly for the factor group $G_{\rho}/G_{1}\simeq\frak{S}_{3}$,
we note that this group acts on $A,B,C$ as a natural permutations.
Then the three elementary symmetric polynomials describe the quotient
$\mathbb{P}_{w}^{2}/G_{\rho}\simeq\mathbb{P}(3,2,1).$ 

\begin{equation}
\begin{matrix}\xymatrix{ &  & \ar@{->}[dl]\,\,\,\mathbb{P}_{w}^{2}\,\,\,\ar@{->}[dr]\\
 & \ar@{->}[d]^{/\mathbb{Z}_{2}}\,\,\,\,\mathbb{P}_{\Delta} &  & \ar@{->}[d]_{/\mathbb{Z}_{2}}\;\;\;\mathbb{P}^{2}\;\;\ar@{->}[dr]^{/\frak{S}_{3}}\\
 & \,\,\,\,\,\mathbb{P}_{\Delta}/\mathbb{Z}_{2}\ar@{->}[rr]^{/\mathbb{Z}_{2}} &  & \mathbb{P}(2,1,1)\ar@{->}[r] & \mathbb{P}(3,2,1)\;\;
}
\end{matrix}\label{eq:moduli-Diag}
\end{equation}

Although we will not use in the following sections, we observe that
the discriminant loci in (\ref{eq:disDsL1L2}) take simple forms in
terms of the invariants $A,B,C$:
\[
\begin{aligned}D_{s}=\left\{ A+B+C=0\right\} ,\;\;L_{1}=\left\{ ABC=0\right\} ,\,\\
L_{2}=\left\{ (A+B)(B+C)(C+A)=0\right\} .\;\;\;\;\;
\end{aligned}
\]

\para{Families.} Let us note that the Heisenberg group $\cH_{8}$
acts on each fiber of the family $\mathcal{V}^{1}\rightarrow\mathbb{P}_{w}^{2}$
by sending a point $([x],[w])\in\cV^{1}$ to $([g.x],[w])\in\cV^{1}.$
Since this is a projective action in $\mbP^{7}$, the group $\cH_{8}$
reduces to $\mbZ_{8}\times\mbZ_{8}$. We will take a subgroup $\mbZ_{8}=\langle\tau\rangle\subset\mbZ_{8}\times\mbZ_{8}$,
and by considering fiberwise quotients by $\mbZ_{8}$ and $\mbZ_{8}\times\mbZ_{8}$,
respectively, we define families of Calabi-Yau manifolds, 
\begin{equation}
\cV_{\mbZ_{8}}^{1}\rightarrow\mbP_{w}^{2},\;\;\;\cV_{\mbZ_{8}\times\mbZ_{8}}^{1}\rightarrow\mbP_{w}^{2}\label{eq:Z8-Z8Z8-families-Pw}
\end{equation}
over general points of $\mathbb{P}_{w}^{2}$ (cf. the lead paragraph
of this section). 

Let us consider the following subgroups $\widetilde{G}_{0},\widetilde{G}_{max}$
which are generated by the specified elements in $\cN\cH_{8}$;
\[
\widetilde{G}_{0}:=\left\langle STS,(ST)^{3}\right\rangle ,\,\,\,\,\,\,\,\widetilde{G}_{max}:=\left\langle STS,(ST)^{3},S^{4}\right\rangle .
\]
Note that these define natural lifts of the groups $G_{0}$ and $G_{max}$
to $\cN\cH_{8}$. 
\begin{prop}
\label{prop:G0-Gmax-sequences} $(1)$ The order of $\widetilde{G}_{0}$
is $64$, and we have 
\[
\begin{matrix}1\rightarrow\left\langle -1,\tau^{4}\right\rangle \rightarrow\widetilde{G}_{0}\rightarrow G_{0}\rightarrow1.\end{matrix}
\]
$(2)$ The order of $\widetilde{G}_{max}$ is 512, and we have 
\[
1\rightarrow\left\langle -1,\sigma^{4},\tau^{4},{{}-i\,u_{5}}\right\rangle \rightarrow\widetilde{G}_{max}\rightarrow G_{max}\rightarrow1,
\]
where $u_{5}=(SSST)^{3}$ corresponds to the unit of $SL_{2}(\mbZ_{8})$. \end{prop}
\begin{proof}
Our proofs are based on explicit matrix calculations. Since they are
straightforward, we refer to \cite{HT-math-c} for the calculations
to verify the claimed properties. \end{proof}
\begin{prop}
By taking quotients of the families (\ref{eq:Z8-Z8Z8-families-Pw})
by $\widetilde{G}_{0}$, we have the corresponding families 
\[
\cV_{\mbZ_{8}}^{1}/\widetilde{G}_{0}\rightarrow\mbP_{\Delta},\;\;\;\cV_{\mbZ_{8}\times\mbZ_{8}}^{1}/\widetilde{G}_{0}\rightarrow\mbP_{\Delta},
\]
 over $($general points of$)$ $\mbP_{\Delta}:=\mbP_{w}^{2}/G_{0}$. \end{prop}
\begin{proof}
The group $\left\langle -1,\tau^{4}\right\rangle $ acts trivially
on the fibers of $\cV_{\mbZ_{8}}^{1}$ and $\cV_{\mbZ_{8}\times\mbZ_{8}}^{1}$.
Hence we have the claimed families over general points of $\mbP_{w}^{2}/G_{0}$.\end{proof}
\begin{rem}
In a similar way to the above proposition, one might expect to have
a family over $\mbP_{w}^{2}/G_{max}$. However, since the unit $u_{5}$
in $\left\langle -1,\sigma^{4},\tau^{4},{{}-i\,u_{5}}\right\rangle $
does not belong to $\cH_{8}$, the family $\cV_{\mbZ_{8}\times\mbZ_{8}}^{1}\rightarrow\mbP_{w}^{2}$
does not reduce to a family over $\mbP_{w}^{2}/G_{max}$. Here, for
convenience to readers, we present the explicit form of $u_{5}$;
\[
\hsp{110}u_{5}=\frac{1}{i}\left(\begin{smallmatrix}0 & 0 & 0 & 0 & 1 & 0 & 0 & 0\\
0 & -1 & 0 & 0 & 0 & 0 & 0 & 0\\
0 & 0 & 0 & 0 & 0 & 0 & 1 & 0\\
0 & 0 & 0 & -1 & 0 & 0 & 0 & 0\\
1 & 0 & 0 & 0 & 0 & 0 & 0 & 0\\
0 & 0 & 0 & 0 & 0 & -1 & 0 & 0\\
0 & 0 & 1 & 0 & 0 & 0 & 0 & 0\\
0 & 0 & 0 & 0 & 0 & 0 & 0 & -1
\end{smallmatrix}\right).\hsp{100}\square
\]
\end{rem}
\begin{defn}
For the sake of notational simplicity, we will write the quotient
families $\cV_{\mbZ_{8}}^{1}/\widetilde{G}_{0}$ and $\cV_{\mbZ_{8}\times\mbZ_{8}}^{1}/\widetilde{G}_{0}$
by 
\begin{equation}
\cV_{\mbZ_{8}}^{1}\rightarrow\mbP_{\Delta}\,\,\,\text{and }\,\,\,\cV_{\mbZ_{8}\times\mbZ_{8}}^{1}\rightarrow\mbP_{\Delta},\label{eq:Z8-Z8Z8-families-PD}
\end{equation}
respectively, unless otherwise mentioned. 
\end{defn}
No confusion with (\ref{eq:Z8-Z8Z8-families-Pw}) should arise in
the above definition since we will mostly confine ourselves to these
families over $\mbP_{\Delta}$ in the following sections. 

~

\vskip3cm

\section{\label{sec:P-Delta-global}\textbf{Parameter space $\mathbb{P}_{\Delta}$}}

In the last section, we have obtained two families of Calabi-Yau manifolds
$\cV_{\mbZ_{8}}^{1}$ and $\cV_{\mbZ_{8\times\mbZ_{8}}}^{1}$ over
the same parameter space $\mbP_{\Delta}=\mbP_{w}^{2}/G_{0}$. Here
we describe the quotient $\mathbb{P}_{\Delta}$ and its resolution
to study mirror symmetry from the families.

\subsection{Toric variety $\protect\mbP_{\Delta}$ }

As we see in Proposition \ref{prop:The-group-G0}, the group $G_{0}$
acts on $\mbP_{w}^{2}$ diagonally giving a toric variety $\mbP_{\Delta}$
associated to a lattice polytope $\Delta$. To describe $\Delta$,
we start with following invariant monomials 
\[
w_{1}^{8},\,\,w_{0}^{3}w_{1}^{4}w_{2},\,\,w_{0}w_{1}^{4}w_{2}^{3},\,\,w_{0}^{8},\,\,w_{0}^{6}w_{2}^{2},\,\,w_{0}^{4}w_{2}^{4},\,\,w_{0}^{2}w_{2}^{6},\,\,w_{2}^{8}.
\]
We read integral vectors $v_{1},v_{2},\cdots,{{}v_{8}}\in\mbZ^{3}$
for the above monomials in order, and set $\cA:=\left\{ v_{1},v_{2},\cdots,{{}v_{8}}\right\} $.
We introduce a lattice $M=\mbZ(\cA-{{}v_{3}})$. The lattice
polytope $\Delta$ is an integral polytope $\mathrm{Conv(\cA})$ in
$M\otimes\mbR$. More concretely, taking an integral basis $\left(\begin{smallmatrix}-1\\
4\\
{{}-3}
\end{smallmatrix}\right),\left(\begin{smallmatrix}-2\\
0\\
2
\end{smallmatrix}\right)$ for the lattice $M\simeq\mbZ^{2}$, we have 
\[
\Delta=\mathrm{Conv.\left(\left(\begin{matrix}-1\\
-3
\end{matrix}\right),\left(\begin{matrix}1\\
0
\end{matrix}\right),\left(\begin{matrix}-1\\
1
\end{matrix}\right)\right)},
\]
where three lattice points corresponds to $w_{0}^{8},w_{1}^{8}$ and
$w_{2}^{8}$, respectively. 

From the normal fan $\Sigma:=\cN_{\Delta}$, we read two $A_{1}$
singularity at the origin corresponding to the vertices $w_{0}^{8},w_{2}^{8}$,
and $A_{3}$ singularity corresponding to the vertex $w_{1}^{8}$.
For example, the affine chart which corresponds to the vertex $w_{2}^{8}$
is given by 
\begin{equation}
\mathrm{Spec}\,\mbC[\frac{w_{1}^{8}}{w_{2}^{8}},\frac{w_{0}^{2}}{w_{2}^{2}},\frac{w_{0}w_{1}^{4}}{w_{2}^{5}}].\label{eq:affine-A1}
\end{equation}
 In Fig.\ref{fig:Fig1}, we describe a toric resolution $\widetilde{\mbP}_{\Delta}\rightarrow\mbP_{\Delta}$
with introducing an affine coordinate
\begin{equation}
x=\frac{1}{4}\frac{w_{0}^{2}}{w_{2}^{2}},\,\,\,y=-2\frac{w_{1}^{4}}{w_{0}w_{2}^{3}}\label{eq:xy-definition}
\end{equation}
for the resolution of the $A_{1}$-singularity (\ref{eq:affine-A1}).
Here we have determined the numerical factors, $\frac{1}{4}$ and
$-2$, in favor of mirror symmetry which we will describe in the next
section. The exceptional divisor of the blowing-up is written by $E$. 

\begin{figure}
\includegraphics[scale=0.33]{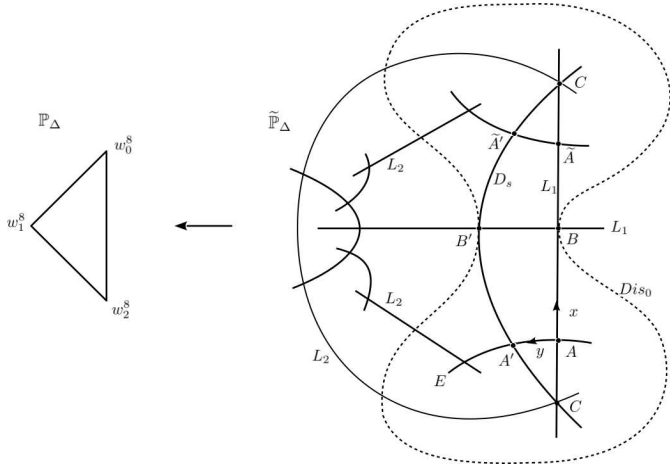}

\caption{Figure 1. \label{fig:Fig1}$\protect\mbP_{\Delta}$ and a resolution
$\widetilde{\protect\mbP}_{\Delta}$. The components $L_{1},L_{2},D_{s}$
and $Dis_{0}$ of the discriminant are also depicted in the right. }
\end{figure}

\subsection{Discriminant loci\label{sub:Disc-L1L2D}}

Our family over $\mbP_{\Delta}$ is either $\cV_{\mbZ_{8}}^{1}$ or
$\cV_{\mbZ_{8}\times\mbZ_{8}}^{1}$, whose general fibers are given
by $V_{8,w}^{1}/\left\langle \tau\right\rangle $ or $V_{8,w}^{1}/\left\langle \sigma,\tau\right\rangle $,
respectively. In either case, since the quotients are taken by free
group actions, the degenerations occur at the same loci as the $(2,2,2,2)$
complete intersection $V_{8,w}$ summarized in Proposition \ref{prop:V1-degeneration-loci}.
We have depicted in Fig.\ref{fig:Fig1} schematically the components
of the proper transforms of the discriminant loci. For simplicity
we use the same letters $L_{1},L_{2},D_{s}$ for the proper transforms,
but it should be clear in the context. Using the affine coordinate
(\ref{eq:xy-definition}), they are given by
\[
\begin{aligned} & \qquad\qquad\qquad\qquad\quad D_{s}=\left\{ 1+4x+y=0\right\} \\
 & L_{1}=\left\{ w_{1}=0\right\} \cup\left\{ 4x-1=0\right\} ,\;\;L_{2}=\left\{ w_{0}w_{2}=0\right\} \cup\left\{ 4x+1\right\} \cup Dis_{0},
\end{aligned}
\]
where the component $Dis_{0}$ is defined by $dis_{0}=0$ with 
\[
dis_{0}=(1-4x)^{4}-256xy(1+4x+y).
\]
The component $Dis_{0}$ will play an important role when describing
mirror symmetry at genus one.

\subsection{Degeneration points }

In the next section, we will study the degenerations of the family
over the resolution $\tPd$ in terms of period integrals of the family.
In the context of mirror symmetry, we are mostly interested in special
boundary points called large complex structure limits (LCSLs) which
are characterized by certain distinguished monodromy properties of
period integrals (see \cite{MoLCSL} for a precise definition and
also Appendix \ref{sec:App-Canonical-Form}). It turns out that there
are many LCSLs in $\tPd$ where mirror symmetry emerges in nice forms
(see Proposition \ref{prop:MS-Result-Main} for our final interpretations).
In Fig.\ref{fig:Fig1}, we have named them $A,A';B,B';C$ and also
$\tilde{A},\tilde{A}'$. 

\para{$A$ and $A'$.} Here and in what follows, we will use the affine
coordinate $(x,y)$ introduced in (\ref{eq:xy-definition}). This
affine coordinate arises as one of the affine coordinate of the blow-up
of the $A_{1}$ singularity at the vertex $w_{2}^{8}$ of $\mbP_{\Delta}$,
where three components of the discriminant $L_{1},L_{2}$ and $D_{s}$
intersect. After the blow-up, the intersection splits into three points
on the exceptional divisor $E$. Two of them, $A$ and $A'$, are
LCSLs which were studied locally in \cite{Pav}. 

\para{Symmetry $w_0\leftrightarrow w_2$.} The points $A,A',\tilde{A},\tilde{A}',B,B'$
and $C$ in Fig.\ref{fig:Fig1} all give rise to LCSLs, which we will
study in detail in Sections \ref{sec:MSbyPF-A}, \ref{sec:MS-B} and
\ref{sec:MS-more-C}. In terms of the affine coordinate $x,y$ in
(\ref{eq:xy-definition}), they are given by
\[
\begin{aligned}A=(0,0),A'=(0,-1);\,\,B=(\frac{1}{4},0),B'=(\frac{1}{4},-2);\,\,C=(-\frac{1}{4},0).\end{aligned}
\]
The coordinates of $\tilde{A}$ and $\tilde{A}'$ are given by $(\frac{1}{16}\frac{1}{x},\frac{1}{4}\frac{y}{x})=(0,0)$
and $(0,-1)$, respectively. In fact, all the boundary divisors are
invariant under the following involution: 
\begin{equation}
(x,y)\mapsto(\frac{1}{16}\frac{1}{x},\frac{1}{4}\frac{y}{x}),\label{eq:inv-sym}
\end{equation}
which comes from the symmetry of $\tPd$ under the exchange $w_{0}\leftrightarrow w_{2}$.
Actually, this symmetry is represented by the action of $\bS^{4}$
in $G_{max}$ described in Proposition \ref{prop:Gmax}. The points
$B,B'$ and $C$ are fixed under this involution, while $A$ and $A'$
are transformed to $\tilde{A}$ and $\tilde{A}'$, respectively. One
might consider the quotient $\mbP_{\Delta}/\left\langle \bS^{4}\right\rangle =\mbP_{w}^{2}/G_{max}$
as a parameter space of the families, but as we saw in Proposition
\ref{prop:G0-Gmax-sequences}, the quotient does not admit the corresponding
family over $\mbP_{\Delta}/\left\langle \bS^{4}\right\rangle $. 

\vskip2cm

\section{\label{sec:MSbyPF-A}\textbf{Mirror symmetry from the boundary points
$A$ and $A'$}}

Associated to the family $\cV_{\mbZ_{8}}^{1}\rightarrow\mbP_{\Delta}$
(resp. $\cV_{\mbZ_{8}\times\mbZ_{8}}^{1}\rightarrow\mbP_{\Delta}$)
we have the local system $R^{3}\pi_{*}\mbC_{\cV_{\mbZ_{8}}^{1}}$(resp.
$R^{3}\pi_{*}\mbC_{\cV_{\mbZ_{8}\times\mbZ_{8}}^{1}}$) over $\mbP_{\Delta}$.
These local systems result in a system of differential equations (Picard-Fuchs
equations) of the same forms, which were studied locally in \cite{Pav}.
We will study the resulting Picard-Fuchs equations globally, and find
in Section \ref{sec:Degens-CY} that a difference between the two
local systems appears in the integral structures for the solutions
of the Picard-Fuchs equations (i.e., in the integral variation of
Hodge structures). Also, we will recognize the difference between
the two families when we calculate genus one Gromov-Witten potentials
(see Remark \ref{rem:Conifold-A} and Remark \ref{rem:Conifold-Factor-C}).

\subsection{Picard-Fuchs equations}

As discovered first in \cite{Candelas}, we can find mirror symmetry
in calculating genus zero Gromov-Witten invariants from the period
integrals which we determine by solving Ficard-Fuchs equations. 

\para{Period integrals.} Since the both families $\cV_{\mbZ_{8}}^{1}$
and $\cV_{\mbZ_{8}\times\mbZ_{8}}^{1}$ come from the same family
$\cV^{1}\rightarrow\mbP_{w}^{2}$ of the Heisenberg-invariant $(2,2,2,2)$
complete intersections in $\mbP^{7}$, we can express the period integrals
following \cite{Griffiths}:
\begin{equation}
\Pi_{\gamma}(w):=\int_{\gamma}Res\left(\frac{w_{2}^{4}}{f_{1}(w,x)f_{2}(w,x)f_{3}(w,x)f_{4}(w,x)}d\mu\right),\label{eq:period-Int-def}
\end{equation}
where $d\mu:={\displaystyle \sum_{i=0}^{7}(-1)^{i}dx_{0}\wedge\cdots\widehat{dx_{i}}\cdots\wedge dx_{7}}$
and $\gamma$ is an integral cycle of a fixed fiber of the family.
Period integral in this form often appears when describing mirror
symmetry; there, we combine the integral over a cycle with the residue
to an integral over a tubular cycle $T(\gamma)$ of the zero locus
$f_{1}=\cdots=f_{4}=0$. It is straightforward to evaluate the period
integral over a tubular cycle
\[
T(\gamma_{0}):=\left\{ [x]\in\mbP^{7}\mid\left|x_{i}/x_{0}\right|=\ve\,\,(i=1,\cdots,7)\right\} .
\]
The following results are obtained in \cite{Pav}:
\begin{prop}
$(1)$ The integral $\Pi_{\gamma_{0}}(w)$ can be evaluated in a closed
formula. 

\noindent$(2)$ In the affine coordinate $x,y$ of (\ref{eq:xy-definition}),
the integral $\Pi_{\gamma_{0}}(w)=:w_{0}(x,y)$ has the following
power series expansion,
\begin{equation}
\omega_{0}(x,y)=1+8x^{2}+16xy+160x^{2}y+16xy^{2}+88x^{4}+1536x^{3}y+\cdots.\label{eq:sol-w0}
\end{equation}
\end{prop}
\begin{proof}
Since all calculations are now standard in literatures (see e.g. \cite{BatCox}),
we only sketch them. We first write quadric equations as 
\[
\frac{1}{w_{2}}\frac{1}{x_{i+1}x_{i+5}}f_{i}(w,x)=1-P_{i}(w,x)\,\,\,(i=1,...,4)
\]
in terms of Laurent polynomials $P_{i}$, and evaluate the residues
in the coordinate $x_{0}=1$ by making geometric series, 
\[
\Pi_{\gamma_{0}}(w)=\int_{T(\gamma_{0})}\sum_{n_{1},...,n_{4}\geq0}P_{1}(w,x)^{n_{1}}\cdots P_{4}(w,x)^{n_{4}}\frac{dx_{1}\cdots dx_{7}}{x_{1}\cdots x_{7}}.
\]
We need to have careful analysis to formulate a closed formula which
we refer to \cite[III.9]{Pav}. However it is straightforward to have
the series expansion up to considerably higher order in $x,y$ (say
total degree 50) which is sufficient for our purpose. 
\end{proof}
From the series expansion (\ref{eq:sol-w0}) up to sufficiently high
degrees, we can determine Picard-Fuchs differential operators which
annihilate the period integral. The following $\cD_{2}$ and $\cD_{3}$
were first determined in \cite{Pav}.
\begin{prop}
\label{prop:PicardFuchs-D2-D3}The period integral (\ref{eq:sol-w0})
satisfy the following set of differential equations of Fuchs type:
\begin{equation}
\cD_{2}w(x,y)=0,\,\,\,\,\,\,\cD_{3}w(x,y)=0,\label{eq:PFeqs-D2-D3}
\end{equation}
where $\cD_{2}$ and $\cD_{3}$ are the second and third order differential
operators, see Appendix \ref{sec:AppendixA-PF} for their explicit
forms. 
\end{prop}
\para{Characteristic variety.} It is straightforward to determine
the singular loci (characteristic variety) of the above differential
operators; we obtain $\left\{ dis_{k}=0\right\} $ with 
\[
\begin{alignedat}{1}dis_{0}:=(1-4x)^{4}-256xy(1+4x+y),\quad\quad\quad\quad\quad\\
dis_{1}:=1+4x+y,\,\,\,\,dis_{2}:=1+4x\;\;\;and\;\;\;dis_{3}:=1-4x,
\end{alignedat}
\]
in addition to the coordinate lines $\left\{ x=0\right\} ,\left\{ y=0\right\} $.
We note that each component of the characteristic variety corresponds
exactly to one of the degenerations summarized in Subsection \ref{sub:Disc-L1L2D}. 

\para{Picard-Fuchs equations at $A'$.} The differential operators
$\cD_{2}$ and $\cD_{3}$ are easily transformed to the affine coordinate
$(x_{1},y_{1})$ centered at the boundary point $A'$ by the relation
\begin{equation}
(x_{1},y_{1})=(x,-4x-y-1).\label{eq:xy-x1y1-inv}
\end{equation}
The local geometry around $A$ and $A'$ is summarized in Fig.\ref{fig:AandA'}.
We denote by $\cD_{2}'$ and $\cD_{3}'$, respectively, the resulting
differential operators from $\cD_{2}$ and $\cD_{3}$.
\begin{prop}
\label{prop:PFeqs-A-A'}The local solutions of $\cD_{2}'\omega=\cD_{3}'\omega=0$
around $(x_{1},y_{1})=(0,0)$ have exactly the same forms as those
of $\cD_{2}\omega=\cD_{3}\omega=0$ around $(x,y)=(0,0)$. \end{prop}
\begin{proof}
Let us denote the simple replacements of variables by $\cD_{k}(x_{1},y_{1}):=\cD_{k}\vert_{{x\rightarrow x_{1}\atop y\rightarrow y_{1}}}$
for $k=2,3$. Calculating the coordinate transformations $\cD_{2}',\cD_{3}'$,
we find the following relations
\[
\cD_{2}(x_{1},y_{1})=y_{1}^{2}\cD_{2}',\,\,\,\,\cD_{3}(x_{1},y_{1})=y_{1}\cD_{3}'-5(1+4x_{1}+2y_{1})\theta_{y_{1}}\cD_{2}'-5y_{1}\cD_{2}'.
\]
These relations implies the claimed property.
\end{proof}
We remark, for later use, that there is the following invariant relation
\[
y(1+4x+y)=y_{1}(1+4x_{1}+y_{1})
\]
under the coordinate change $(x_{1},y_{1})=(x,-4x-y-1)$. 

\begin{figure}
\includegraphics[scale=0.5]{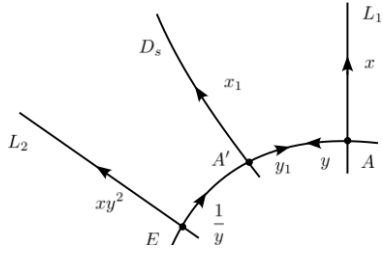}

\caption{\label{fig:AandA'}Figure 2. Boundary points $A$ and $A'$}
\end{figure}

\subsection{Griffiths-Yukawa couplings \label{sub:Griffiths-Yukawa-couplings}}

Let us denote by $\Omega(w)$ the holomorphic three form on a general
fiber $V_{8,w}^{1}/\mbZ_{8}$ of the family $\cV_{\mbZ_{8}}^{1}$
over $\mbP_{\Delta}$, by which we express the period integral (\ref{eq:period-Int-def})
as $\Pi_{\gamma}(w)=\int_{\gamma}\Omega(w)$. 

\para{Near the point $A$.} We define the so-called Griffiths-Yukawa
couplings by 
\begin{equation}
C_{ijk}^{A}:=-\int_{V_{8,w}^{1}/\mbZ_{8}}\Omega(w)\frac{\partial^{3\;}}{\partial z_{A}^{i}\partial z_{A}^{j}\partial z_{A}^{k}}\Omega(w)\label{eq:Cijk-A}
\end{equation}
where $(z_{A}^{1},z_{A}^{2})=(x,y)$ is the affine coordinate centered
at the degeneration point $A$. Using the Picard-Fuchs equations (see
e.g. \cite{HKTY1}), we have 
\begin{prop}
\label{prop:GY-Cijk-A}Up to a common constant d, the Griffiths-Yukawa
couplings $C_{111}^{A},C_{112}^{A},C_{122}^{A}$ and $C_{222}^{A}$
are determined in order as follows:
\[
\frac{d\,P_{111}}{x^{3}dis_{0}\,dis_{1}^{2}\,dis_{2}\,dis_{3}},\,\,\,\,\frac{d\,P_{112}}{x^{2}y\,dis_{0}\,dis_{1}^{2}\,dis_{3}},\,\,\,\,\frac{d\,P_{122}}{xy\,dis_{0}\,dis_{1}^{2}\,dis_{3}},\,\,\,\,\frac{d\,P_{222}}{y^{2}dis_{0}\,dis_{1}^{2}},
\]
where $P_{ijk}=P_{ijk}(x,y)$ are polynomials given in Appendix \ref{sec:Appendix-Griffiths-Yukawa}
{}{$($see also \cite{HT-math-c}}$)$.
\end{prop}
In Proposition $\ref{prop:qYukawa-A}$, we determine the overall constant
$d$ in $C_{ijk}^{A}$ by finding mirror symmetry of the family $\cV_{\mbZ_{8}}^{1}$
to a Calabi-Yau manifold $V_{8,w}^{1}$. 

\para{Near the point $A'$.} \label{para:iso-A-A'}In the same way
as above, we define the Griffiths-Yukawa couplings $C_{ijk}^{A'}$
in terms of the affine coordinate $(z_{A'}^{1},z_{A'}^{2})=(x_{1},y_{1})$
centered at $A'$. As above, they are determined by the Picard-Fuchs
equations $\cD_{2}'\omega=\cD_{3}'\omega=0$ up to a normalization.
However, due to the properties in Proposition \ref{prop:PFeqs-A-A'},
we see that the isomorphism $C_{ijk}^{A'}=C_{ijk}^{A}$, i.e., $C_{ijk}^{A'}(x_{1},y_{1})=C_{ijk}^{A}(x_{1},y_{1}).$

\subsection{Mirror symmetry }

For a boundary point $P=D_{1}\cap D_{2}$ given as a normal crossing
divisors, by solving the Picard-Fuchs equations around $P$, we can
study the local monodromies around boundary divisors $D_{1},D_{2}$.
Mirror symmetry of the family $\cV_{\mbZ_{8}}^{1}$ to a Calabi-Yau
manifold $M$ arises at a special boundary point called the large
complex structure limit (LCSL), which is characterized{}{{}
by }unipotent monodromy properties around the boundary divisors; these
special forms of monodromies define the so-called monodromy weight
filtration on the third cohomology space $H^{3}(V_{8,w}^{1}/\mbZ_{8},\mathbb{R})$.
As one of the consequences of mirror symmetry, we observe an isomorphism
between the weight monodromy filtration and the corresponding filtration
in the hard Lefschetz theorem applied for the mirror Calabi-Yau manifold
$M$ (see e.g. \cite[Sect.2]{HT-birat}). 
\begin{prop}
\label{prop:MS-A-A'}The boundary points $A$ and $A'$, respectively,
of the Picard-Fuchs equations over $\mbP_{\Delta}$ are mirror symmetric
to Calabi-Yau manifolds{}{{} $X:=V_{8w}^{1}$ and $X':=V_{8,w}^{2}$
i}n the above sense. 
\end{prop}
We can verify the above proposition from the weight monodromy filtrations
which we introduce by using the canonical forms (see Appendix \ref{sec:App-Canonical-Form})
of local solutions around $A$ and $A'$. Here we remark that there
is no distinction between $A$ and $A'$ as far as local properties
are concerned as we observed in Proposition \ref{prop:PFeqs-A-A'}.
Similarly, there is no distinction for the filtrations coming from
the hard Lefschetz theorem for $V_{8,w}^{1}$ and its birational model
$V_{8,w}^{2}$. In the above proposition, we have fixed one way of
the mirror identification and we will retain this in what follows.

\subsection{Mirror symmetry and Gromov-Witten invariants}

Mirror symmetry observed in Proposition \ref{prop:MS-A-A'} can be
confirmed by extracting their quantum cohomology from the Griffiths-Yukawa
couplings (\ref{eq:Cijk-A}) expanded near the boundary points. This
was first achieved in \cite{Pav}, however we reproduce it here in
order to fit the results into our global study of the family over
$\mbP_{\Delta}$. 

\para{Mirror maps.} Near the boundary point $A$, it is straightforward
to solve the Picard-Fuchs equations $\cD_{2}\omega=\cD_{3}\omega=0$
in the forms of power series with logarithmic singularities around
the boundary divisors $\left\{ x=0\right\} $ and $\left\{ y=0\right\} $.
The solutions consist of six independent power series, which have
the following leading logarithmic singularities: 
\[
\begin{alignedat}{4} & \omega_{0}=1, & \;\; & \omega_{1}=\log x, & \;\; & \omega_{2,1}=(\log x)^{2}+2(\log x)(\log y),\\
 &  & \;\; & \omega_{2}=\log y, & \;\; & \omega_{2,2}=(\log x)^{2}, & \hsp{-50} & \omega_{3}=(\log x)^{3}+3(\log x)^{2}(\log y).
\end{alignedat}
\]
This structure of the logarithmic singularities in the solutions actually
defines the weight monodromy filtration $W_{0}\subset W_{2}\subset W_{4}\subset W_{6}$,
see Appendix \ref{sec:App-Canonical-Form}. For our purpose here,
we only need to determine the explicit forms of the solutions $\omega_{0},\omega_{1},\omega_{2}$.
Note that the solution $\omega_{0}$ is unique by the leading behavior
$\omega_{0}=1$, which is given by (\ref{eq:sol-w0}). On the other
hand the solutions $\omega_{1},\omega_{2}$ are determined up to adding
constant multiples of $\omega_{0}$; we set these solutions as 
\[
\omega_{1}=\omega_{0}\log x+\omega_{1}^{reg}+c_{1}\omega_{0},\,\,\,\omega_{2}=\omega_{0}\log y+\omega_{2}^{reg}+c_{2}\omega_{0},
\]
where $\omega_{k}^{reg}$ represent power series with no constant
terms. 
\begin{defn}
\label{def:MirrorMap-A}We define mirror map by the inverse relations
$x=x(q_{1},q_{2}),$ $y=y(q_{1},q_{2})$ of 
\[
q_{1}=e^{\frac{\omega_{1}}{\omega_{0}}}=C_{1}x\,\exp\big(\frac{\omega_{1}^{reg}}{\omega_{0}}\big),\,\,\,\,q_{2}=e^{\frac{\omega_{2}}{\omega_{0}}}=C_{2}y\,\exp\big(\frac{\omega_{2}^{reg}}{\omega_{0}}\big),
\]
where $C_{k}:=e^{c_{k}}$. Also we often write $q_{k}=e^{t_{k}}(k=1,2)$
with introducing $t_{k}$. 
\end{defn}
\para{Griffiths-Yukawa couplings.} Mirror symmetry arises from the
family $\cV_{\mbZ_{8}}^{1}$ if we combine Griffiths-Yukawa couplings
in Proposition \ref{prop:GY-Cijk-A} with the mirror maps near the
boundary point \cite{Candelas}. We recall that the so-called quantum
corrected Yukawa couplings are defined by 
\begin{equation}
Y_{abc}^{A}=\big(\frac{1}{\omega_{0}}\big)^{2}\sum_{i,j,k}C_{ijk}^{A}\frac{dz_{A}^{i}}{dt_{a}}\frac{dz_{A}^{j}}{dt_{b}}\frac{dz_{A}^{k}}{dt_{c}},\label{eq:Yijk-A}
\end{equation}
where $(z_{A}^{1},z_{A}^{2}$) represents the mirror map $(x,y)=(x(q_{1},q_{2})$,
$y(q_{1},q_{2}))$ at the boundary point $A$, and the substitution
of the mirror map is assumed in the r.h.s. Similarly we have the quantum
corrected Yukawa couplings $Y_{abc}^{A'}$ with $C_{ijk}^{A'}$ and
the mirror map $z_{A'}^{i}$ defined around the other boundary point
$A'$.
\begin{prop}
\label{prop:qYukawa-A}When we set $d=1$ in $C_{ijk}^{A}$ and $C_{1}=C_{2}=1$
in the definition of mirror map, we have 
\[
\begin{alignedat}{2} & Y_{111}^{A}=16+512\,q_{1}q_{2}+22528\,q_{1}^{2}q_{2}\cdots, & \,\, & Y_{112}^{A}=16+512\,q_{1}q_{2}+11264\,q_{1}^{2}q_{2}\cdots,\\
 & Y_{122}^{A}=0+512\,q_{1}q_{2}+5632\,q_{1}^{2}q_{2}\cdots, &  & Y_{222}^{A}=0+64\,q_{2}+512\,q_{1}q_{2}+64\,q_{2}^{2}\cdots.
\end{alignedat}
\]
Also, we have exactly same form for the corresponding expansions of
$Y_{abc}^{A'}$.
\end{prop}
Quantum corrected Yukawa couplings are related to the so-called genus
zero Gromov-Witten potential $F_{0}^{A}$ , by $Y_{abc}^{A}=\frac{\partial^{3}F_{0}^{A}}{\partial t^{a}\partial t^{b}\partial t^{c}}.$
From this potential function, we read Gromov-Witten invariants and
also classical cubic forms (\ref{eq:yukawaV1}) in the following form:
\[
F_{0}^{A}=\frac{16}{3!}\,t_{1}^{3}+\frac{16}{2!}t_{1}^{2}t_{2}+\sum_{d_{1},d_{2}}N_{0}^{A}(d_{1},d_{2})q_{1}^{d_{1}}q_{2}^{d_{2}},
\]
where the summation over the bi-degrees $\beta:=(d_{1},d_{2})$ is
restricted to $d_{1},d_{2}\geq0$ and $(d_{1},d_{2})\not=(0,0)$ (see
\ref{para:potF-kappa-gen} of Section \ref{sec:MS-more-C} for a more
invariant form of $F(t)$). 
\begin{prop}
Assuming mirror symmetry, the numbers $N_{0}^{A}(\beta)$ are Gromov-Witten
invariants of stable maps to $V_{8,w}^{1}$ with $\beta\not=0\in H_{2}(V_{8,w}^{1},\mbZ)$.
The bi-degree $(d_{1},d_{2})$ represents the degrees $(\beta.H_{X},\beta.A_{X})$
with respect to the divisor classes $H_{X}=H_{1}$ and $A_{X}=A_{1}$
introduced in $($\ref{eq:yukawaV1}$)$ and \ref{para: V1-summary}. 
\end{prop}
Similar results hold for the numbers $N_{0}^{A'}(d_{1},d_{2})$ with
$V_{8,w}^{1}$ being replaced by the birational model $V_{2,w}^{2}$,
and divisors $H_{1},A_{1}$ replaced by $H_{2},\tilde{A}_{2}$ in
\ref{para:divisors- H1A1-and-H2A2}. 

It is convenient to define BPS numbers $n_{0}(\beta):=n_{0}(d_{1},d_{2})$
by the relations 
\[
N_{0}^{A}(\beta)=\sum_{k|\beta}\frac{1}{k^{3}}n_{0}^{A}(\beta/k),
\]
which remove the contributions from the so-called multiple covers
\cite{AM2,Vo} in $N_{0}^{A}(\beta)=N_{0}^{A}(d_{1},d_{2})$. In Table
\ref{tab:Table-A} (0), we list the resulting BPS numbers from $F_{0}^{A}$. 

Results in this subsection were first obtained in \cite{Pav} verifying
that some of BPS numbers coincides with rational curves on $V_{8,w}^{1}$.
The identification of the two boundary points $A$ and $A'$, respectively,
with the birational models $V_{8,w}^{1}$ and $V_{8,w}^{2}$ is justified
by observing the number of flopping curves in $n_{0}^{A}(0,1)=n_{0}^{A'}(0,1)=64$,
and also a ``sum-up relation'' 
\begin{equation}
n_{0}^{(2,2,2,2)}(d)=\sum_{d_{2}}n_{0}^{A}(d,d_{2})=\sum_{d_{2}}n_{0}^{A'}(d,d_{2}),\label{eq:sumup-rel-AA'}
\end{equation}
which reproduce the BPS numbers in Tables E1 (1) of Appendix \ref{sec:Appendix-2222}
for a smoothing of the singular Calabi-Yau variety $X_{2,2,2,2}^{sing}$
arising from the contractions of 64 curves;\def\contRi{\begin{xy}
(-15,0)*++{V^1_{8,w}}="Vi",
( 15,0)*++{V^2_{8,w}}="Vii",
(  0,-13)*++{\;\;\;\;X_{2,2,2,2}^{sing}\;.}="Zi",
(-11,-7)*++{\,^{\,_{64\,\mathbb{P}^1}}},
( 11,-7)*++{\,^{\,_{64\,\mathbb{P}^1}}},
\ar@{<-->}^{} "Vi";"Vii"
\ar "Vi";"Zi"
\ar "Vii";"Zi"
\end{xy} }
\begin{equation}
\begin{matrix}\contRi\end{matrix}\label{eq:Birat-i}
\end{equation}

\subsection{Counting functions by quasi-modular forms}

Actually, mirror symmetry of Calabi-Yau manifolds which have abelian
surface fibration was first studied in the case of fiber product of
two rational elliptic surfaces, i.e., Schoen's Calabi-Yau threefolds
\cite{HSS,HST1}; there it was found that some part of Gromov-Witten
potential are expressed by quasi-modular forms coming from elliptic
curves. It is interesting to observe that the Gromov-Witten potential
of $V_{8,w}^{1}$ has similar properties. 

Let us define $q$-series $Z_{0,n}(q)\,(n=1,2,...)$ by 
\begin{equation}
F_{0}^{A}(q,p)=\frac{16}{3!}\,t_{1}^{3}+\frac{16}{2!}t_{1}^{2}t_{2}+\sum_{n\geq1}Z_{0,n}(q)p^{n},\label{eq:F0-A}
\end{equation}
where $F_{0}^{A}(q,p)=F_{0}^{A}(q_{1},q_{2})$. By definition of the
bi-degree $(d_{1},d_{2})$, the $q$-series $Z_{0,n}^{A}(q)=\sum_{d\geq0}N_{0}^{A}(d,n)q^{d}$
counts the Gromov-Witten invariants related to curves which intersects
with the fiber class $n$-times, i.e., $n$-sections of the fibration
$V_{8,w}^{1}\rightarrow\mbP^{1}$. In particular, since it holds that
\[
Z_{0,1}^{A}(q)=\sum_{d\geq0}N_{0}^{A}(d,1)q^{d}=\sum_{d\geq0}n_{0}^{A}(d,1)\,q^{d},
\]
the $q$-series $Z_{0,1}(q)$ counts BPS numbers of the sections.
We can observe the following property from the table of BPS numbers:

\begin{Xobs}{  \label{obs:ZA_01}The $q$-series $Z_{0,1}(q)$ has
a closed form given by 
\[
Z_{0,1}^{A}(q)=\frac{64}{\bar{\eta}(q)^{8}},
\]
where $\bar{\eta}(q)=\Pi_{n\geq1}(1-q^{n})$. Moreover, we have
\[
Z_{0,n}^{A}(q)=P_{0,n}^{A}(E_{2},E_{4},E_{6})\left(\frac{64}{\bar{\eta}(q)^{8}}\right)^{n},
\]
where $P_{0,n}^{A}(E_{2},E_{4},E_{6})$ are quasi-modular forms of
weight $4(n-1)$ in terms of Eisenstein series $E_{2},E_{4}$ and
$E_{6}$. }\end{Xobs}

We have verified the above properties by calculating $Z_{0,n}(q)$
up to sufficiently large degree $d\le70$ and for $n\leq9$. Below
are explicit forms of the resulting polynomials $P_{0,n}^{A}$ for
lower $n$;
\[
\begin{alignedat}{2}P_{0,2}^{A}= &  & \,\, & \frac{1}{4608}(8\,E_{2}^{2}+E_{4}),\\
P_{0,3}^{A}= &  &  & \frac{1}{2654208}(14E{}_{2}^{4}+7E{}_{2}^{2}E_{4}+E{}_{4}^{2}+2E_{2}E_{6}),\\
P_{0,4}^{A}= &  &  & \frac{1}{2^{26}3^{7}}(3008E{}_{2}^{6}+2808E{}_{2}^{4}E_{4}+1128E{}_{2}^{2}E{}_{4}^{2}+125E{}_{4}^{3}\\
\quad\;\, &  &  & \hsp{90}+1120E{}_{2}^{3}E_{6}+528E_{2}E_{4}E_{6}+31E_{6}^{2}).
\end{alignedat}
\]
{}{The forms of $P_{0,n}^{A}\,(5\leq n\leq9)$ can be
found in \cite{HT-math-c}. }

\subsection{Counting sections by geometry of singular fibers}

It is easy to identify the number $64$ with the number of the sections
of $X=V_{8,w}^{1}\rightarrow\mbP^{1}$. The appearance of the $\eta$-function
in the denominator reminds us of similar counting formulas for a rational
elliptic surface in Schoen's Calabi-Yau manifolds \cite{HSS,HST1},
which came from 12 singular fibers of Kodaira's $I_{1}$ type. In
the present case, the singular fibers consist of 8 elliptic translation
scrolls,
\[
S_{k}=\bigcup_{p\in E_{k}}\langle p,p+e_{k}\rangle\,\,\,\,\,(k=1,...,8)
\]
which is described in the part \ref{para: V1-summary} (2), where
$\langle p,p+e_{k}\rangle$ represents a line passing two points $p$
and $p+e_{k}$$(e_{k}\in E_{k})$. Clearly, the power 8 in the denominator
of $Z_{0,1}^{A}(a)=\frac{64}{\bar{\eta}(q)^{8}}$ should be explained
by the number of singular fibers. Since $N_{0}^{A}(\beta)=n_{0}^{A}(\beta)$
holds for the classes $\beta=(\beta\cdot H_{X},\beta\cdot A_{X})=(d,1)$,
the function $\frac{64}{\bar{\eta}(q)^{8}}$ should count the numbers
of rational curves coming from the 8 translation scrolls. Let us fix
a section $\sigma$ and denote chains of lines contained in {}{{} translation
scrolls $S_{k}\,(k=1,...,8)$ }by
\[
L_{1}^{(1)},....,L_{n_{1}}^{(1)};L_{1}^{(2)},....,L_{n_{2}}^{(2)};\cdots;L_{1}^{(8)},....,L_{n_{8}}^{(8)},
\]
where $\sigma$ intersect at one point with some line $L_{n_{i}}^{(k)}$
{}{in the chain $L_{1}^{(k)},...,L_{n_{k}}^{(k)}$}.
These chains of rational curves could explain the counting function
$\frac{1}{\bar{\eta}(q)^{8}}$, if a configuration 
\[
\sigma_{X}\cup L_{1}^{(k)}\cup L_{2}^{(k)}\cup\cdots\cup L_{n_{k}}^{(k)}\,\text{with a fixed }k
\]
had a contribution $p(n_{k})(:=$the number of partitions of $n_{k}$)
to $n_{0}(d,1)=n_{0}(n_{1}+n_{2}+\cdots+n_{8},1)$. However, as one
easily recognize, a naive counting from this configuration is $n_{k}$
instead of $p(n_{k})$, while $p(n_{k})=n_{k}$ holds for $n_{k}\leq3$.
In fact, in \cite{Pav}, $n_{0}(d,1)$ for $d=n_{1}+\cdots+n_{8}\leq3$,
i.e., the numbers $n_{0}(1,1)=8\times64,\,n_{0}(2,1)=44\times64,\,n_{0}(3,1)=192\times64$
are explained by studying Gromov-Witten theory for the above configurations.
However, for $d>3$, there are missing configurations or contributions
to explain $n_{0}(\beta)=n_{0}(d,1)$. We hope that we will come to
this problem in a future work. 

\begin{table}
{\scriptsize{}
\[
\begin{array}{c}
\hsp{-5}\begin{array}{|c|ccccccccc}
\hline \m j\diagdown i\m & 0 & 1 & 2 & 3 & 4 & 5 & 6 & 7 & 8\\
\hline 0 & 0 & 0 & 0 & 0 & 0 & 0 & 0 & 0 & 0\\
1 & 64 & 512 & 2816 & 12288 & 46464 & 157696 & 493056 & 1441792 & 3989568\\
2 & 0 & 0 & 4096 & 98304 & 1220608 & 10813440 & 76775424 & 464322560 & 2480783360\\
3 & 0 & 0 & 2816 & 195072 & 6301056 & 124829696 & 1772620032 & 19764707328 & 183168532288\\
4 & 0 & 0 & 0 & 98304 & 10567680 & 478740480 & 13238665216 & 261369036800 & 4018366742528\\
5 & 0 & 0 & 0 & 12288 & 6301056 & 728901120 & 40797528064 & 1437499588608 & 36413468765248\\
6 & 0 & 0 & 0 & 0 & 1220608 & 478740480 & 58763759616 & 3812602150912 & 160955539341312\\
: & : & : & : & : & : & : & : & : & :
\end{array}\end{array}
\]
}{\scriptsize \par}

(0) Genus zero BPS numbers $n_{0}^{A}(i,j)$ 

{\scriptsize{}
\[
\begin{array}{c}
\hsp{-5}\begin{array}{|c|cccccccccc}
\hline \m j\diagdown i\m & 0 & 1 & 2 & 3 & 4 & 5 & 6 & 7 & 8 & 9\\
\hline 0 & 0 & 0 & 0 & 0 & 0 & 0 & 0 & 0 & 8 & 0\\
1 & 0 & 0 & 0 & 0 & 896 & 8192 & 50688 & 245760 & 1024768 & 3807232\\
2 & 0 & 0 & 0 & 0 & 2240 & 163840 & 3262464 & 39583744 & 358221888 & 2646376448\\
3 & 0 & 0 & 0 & 0 & 1920 & 966656 & 47699968 & 1190572032 & 19998058112 & 255958884352\\
4 & 0 & 0 & 0 & 0 & 4640 & 2097152 & 248373248 & 12249595904 & 363641449120 & 7661339901952\\
5 & 0 & 0 & 0 & 0 & 1920 & 2311168 & 619652096 & 57765986304 & 2933367424384 & 98849134485504\\
6 & 0 & 0 & 0 & 0 & 2240 & 2097152 & 833927680 & 141594460160 & 12170949409472 & 649231850569728\\
7 & 0 & 0 & 0 & 0 & 896 & 966656 & 619652096 & 190239614976 & 27968653028864 & 2375812698800128\\
: & : & : & : & : & : & : & : & : & : & :
\end{array}\end{array}
\]
}{\scriptsize \par}

(1) Genus one BPS numbers $n_{1}^{A}(i,j)$

{\scriptsize{}
\[
\begin{array}{|c|ccccccccccc}
\hline \m j\diagdown i\m & 0 & 1 & 2 & 3 & 4 & 5 & 6 & 7 & 8 & 9 & 10\\
\hline 0 & 0 & 0 & 0 & 0 & 0 & 0 & 0 & 0 & 0 & 0 & 0\\
1 & 0 & 0 & 0 & 0 & 0 & 0 & 256 & 4096 & 37568 & 247808 & 1332480\\
2 & 0 & 0 & 0 & 0 & 0 & 0 & 12288 & 524288 & 10758016 & 145555456 & 1490690048\\
3 & 0 & 0 & 0 & 0 & 0 & 0 & 69120 & 10266624 & 465354496 & 11849469952 & 209498467584\\
4 & 0 & 0 & 0 & 0 & 0 & 0 & 147456 & 75792384 & 6951182592 & 311243702272 & 8956370026496\\
5 & 0 & 0 & 0 & 0 & 0 & 0 & 302592 & 264691712 & 47403040000 & 3616972675072 & 164270290355456\\
6 & 0 & 0 & 0 & 0 & 0 & 0 & 364544 & 531759104 & 171942100864 & 21836284264448 & 1539526624321536\\
7 & 0 & 0 & 0 & 0 & 0 & 0 & 302592 & 674427904 & 359333519296 & 74782408261632 & 8139227500850176\\
: & : & : & : & : & : & : & : & : & : & : & :
\end{array}
\]
}{\scriptsize \par}

(2) Genus two BPS numbers $n_{2}^{A}(i,j)$

~

\caption{Table 1. \label{tab:Table-A} BPS numbers $n_{g}^{A}(i,j)$ of $X=V_{8,w}^{1}$
for $g=0,1,2$. These are equal to $n_{g}^{A'}(i,j)$ of $V_{8,w}^{2}$.}
\end{table}

\subsection{\label{sub:F1-A-A'}Genus one Gromov-Witten potentials $F_{1}^{A}$
and $F_{1}^{A'}$}

Using mirror symmetry, we can extend Observation \ref{obs:ZA_01}
to genus one Gromov-Witten potentials $F_{1}^{A}$ and $F_{1}^{A'}$.
Here we start with recalling the general form of the genus one Gromov-Witten
potential of a Calabi-Yau threefold $M$ proposed in \cite{BCOV1}. 

\para{Genus one potential $F_1$.} Suppose $M$ is a Calabi-Yau threefold
with rank$(Pic(M))=2$, and that we have a family of mirror Calabi-Yau
manifolds over some parameter space with a boundary point (LCSL) at
the origin of an affine coordinate $(x_{1},x_{2})$. Then using mirror
map defined near the origin, we have the genus one potential function
$F_{1}^{M}$ of $M$ in the following form: 
\begin{equation}
F_{1}^{M}=\frac{1}{2}\log\Big\{\big(\frac{1}{\omega_{0}}\big)^{3+h^{1,1}-\frac{\chi(M)}{12}}\frac{\partial(x_{1},x_{2})}{\partial(t_{1},t_{2})}dis_{0}^{r_{0}}\prod_{k\geq1}dis_{k}^{r_{k}}\times\prod_{i}x_{i}^{-1-\frac{c_{2}.H_{i}}{12}}\Big\},\label{eq:F1-X-formula}
\end{equation}
where $\chi(M)$ is the Euler number of $M$ and $h^{1,1}=h^{1,1}(M)$,
and also $c_{2}.H_{i}$ are the linear forms described in \ref{para:divisors- H1A1-and-H2A2}.
The notation $dis_{k}$ represents the components of the discriminant
of the family over the parameter space. Among them, $dis_{0}$ is
reserved to represent the component where the most general degenerations
of the fibers appear. The powers $r_{k}$ are unknown parameters which
we need to determine from some additional data. 

In the present case of $V_{8,w}^{1}$, we use the topological invariants
$\chi(V_{8,w}^{1})=0,c_{2}.H_{1}=64$ and $c_{2}.A_{1}=0$ (and the
same numbers of the corresponding invariants for $V_{8,w}^{2})$.
The unknown parameters $r_{k}$ can be fixed by knowing some of vanishing
results on Gromov-Witten invariants. 
\begin{prop}
Near the boundary point $A$, assuming mirror symmetry, we have the
potential function
\[
F_{1}^{A}(q,p)=\frac{1}{2}\log\Big\{\big(\frac{1}{\omega_{0}}\big)^{3+2}\frac{\partial(x,y)}{\partial(t_{1},t_{2})}dis_{0}^{-\frac{1}{6}}\,dis_{1}^{-1}\,dis_{2}^{-\frac{2}{3}}\,dis_{3}^{-1}\times x^{-1-\frac{64}{12}}y^{-1}\Big\}
\]
which gives genus one Gromov-Witten potential of $V_{8,w}^{1}$. For
the other boundary point $A'$, we have isomorphic form $F_{1}^{A'}$
with the same parameters $r_{0}$ and $r_{k}$. \end{prop}
\begin{rem}
\label{rem:Conifold-A}From the above expression of $F_{1}^{A}$,
we can convince ourselves that we are working on the family $\cV_{\mbZ_{8}}^{1}\rightarrow\mbP_{\Delta}$
to describe mirror symmetry. This comes from the well-known observation
(see e.g. \cite{AM1}) that the power of $dis_{0}^{-\frac{1}{6}}$
is determined by 
\begin{equation}
-\frac{1}{6}\times(\text{the number of ODPs}),\label{eq:Conifold-Factor}
\end{equation}
where we count the number of OPDs which appear generically on fiber
Calabi-Yau manifolds over the principal component $\left\{ dis_{0}=0\right\} $
of the discriminant. The number (\ref{eq:Conifold-Factor}) is often
called \textit{a conifold factor.} Recall that we showed in Proposition
\ref{prop:ODP-8} that over $\left\{ dis_{0}=0\right\} $ there appear
8 ODPs in $V_{8,w}^{1}$ which lies on a single $\tau$-orbit. Namely,
there appears $1$ ODP for the free quotient $V_{8,w}^{1}/\mbZ_{8}$
with $\mbZ_{8}=\langle\tau\rangle$. $\hfill\square$
\end{rem}
We read genus one Gromov-Witten invariants from the following expansion
of $F_{1}^{A}$ with $d_{1},d_{2}\geq0$ and $(d_{1},d_{2})\not=(0,0)$:
\begin{equation}
F_{1}^{A}(q,p)=-\frac{c_{2}.H_{X}}{24}\log q-\frac{c_{2}.A_{X}}{24}\log p+\sum_{d_{1},d_{2}}N_{1}^{A}(d_{1},d_{2})q^{d_{1}}p^{d_{2}}.\label{eq:F1-A}
\end{equation}
The BPS numbers $n_{1}^{A}(\beta):=n_{1}^{A}(i,j)$ may be introduced
through the following relations to genus one Gromov-Witten invariants
$N_{1}^{A}(\beta):=N_{1}^{A}(d_{1},d_{2})$: 
\[
N_{1}^{A}(\beta)=\sum_{k|\beta}\frac{1}{k}\left\{ n_{1}^{A}(\beta/k)+\frac{1}{12}n_{0}^{A}(\beta/k)\right\} .
\]
BPS numbers are conjectured to be integer invariants coming from certain
counting problems of curves in $X=V_{8,w}^{1}$ with homology class
$\beta$ \cite{GV,HST2}. For example, $n_{1}^{A}(8,0)$ counts the
numbers of elliptic curves in the eight singular fibers of $V_{8,w}^{1}$.
To determine the parameters in the above proposition, we have required
vanishings $n_{1}^{A}(i,j)=0$ for $i=0,..,3.$ In Table \ref{tab:Table-A}
(1), we list the resulting BPS numbers form $F_{1}^{A}(q,p)$. 

\para{$q$-series $Z_{1,1}(q)$ via quasi-modular forms.} We define
genus one $q$-series $Z_{1,n}(q)$ by arranging the expansion (\ref{eq:F1-A})
into 
\[
F_{1}^{A}=-\frac{c_{2}.H_{X}}{24}\log q+Z_{1,0}(q)+\sum_{n\geq1}Z_{1,n}^{A}(q)\,p^{n}.
\]
The $q$-series $Z_{1,n}^{A}(q)$ is a generating series of genus
one Gromov-Witten invariants $N_{1}^{A}(d,n)$. For $n=0$, we observe
that non-vanishing invariants appear only at 
\[
N_{1}^{A}(8d,0)=\sum_{k|d}\frac{1}{k}n_{1}^{A}(8d/k,0)
\]
 $(d=1,2,...)$ which comes from elliptic curves in singular fibers
of abelian fibration $V_{8,w}^{1}\to\mbP^{1}$ (see \ref{para: V1-summary}).
Including the first term $-\frac{c_{2}.H_{X}}{24}\log q=-\frac{8}{3}\log q$
in the definition $\widehat{Z}_{1,0}(q)=-\frac{c_{2}.H_{X}}{24}\log q+$$Z_{1,0}^{A}(q)$,
we have the $q$-series 
\[
\widehat{Z}_{1,0}(q)=-\frac{c_{2}.H_{X}}{24}\log q-8\,\log\,\bar{\eta}(q^{8})=-8\log\,\eta(q^{8})
\]
in terms of Dedekind's $\eta$-function $\eta(q)=q^{\frac{1}{24}}\Pi_{n\geq1}(1-q^{n})$.
Corresponding to Observation \ref{obs:ZA_01}, we verify the following
property up to sufficiently high degrees of $q$.

\begin{Xobs}{\label{obs:Z11-A}The q-series $Z_{1,1}^{A}(q)$ is
expressed by quasi-modular forms as 
\[
Z_{1,1}^{A}(q)=P_{1,1}^{A}(q)\,\frac{64}{\bar{\eta}(q)^{8}}
\]
where $P_{1,1}^{A}(q)$ is a polynomial 
\[
P_{1,1}^{A}=\frac{1}{288}\left\{ E_{2}^{2}+E_{2}(S+2T+4U)+4(S+2T-5U)^{2}\right\} ,
\]
of Eisenstein series $E_{2}=E_{2}(q)$ and elliptic theta functions
\[
S=\theta_{3}(q)^{4},\quad T=\theta_{3}(q^{2})^{4},\quad U=\theta_{3}(q)^{2}\theta_{3}(q^{2})^{2}
\]
with $\theta_{3}(q):=\Sigma_{n\in\mbZ}q^{n^{2}}$.}\end{Xobs}

We have verified similar polynomial expressions of $P_{1,n}^{A}$
for $n<4$ (see Appendix \ref{sec:Appendix-Pgn-AB} {}{and
\cite{HT-math-c}}). We conjecture the following form of $q$-series
$Z_{g,n}^{A}(q)$ in terms of quasi-modular forms in general. 
\begin{conjecture}
\label{conj:Zgn-A}The $q$-series $Z_{g,n}^{A}(q)\,(n\geq1)$ is
expressed by 
\[
Z_{g,n}^{A}(q)=P_{g,n}^{A}(E_{2},S,T,U)\left(\frac{64}{\bar{\eta}(q)^{8}}\right)^{n},
\]
where $P_{g,n}^{A}$ is polynomial of degree $2(g+n-1)$ of $E_{2},S,T,U$. 
\end{conjecture}
In Subsection \ref{sub:Fg2-AB}, by using mirror symmetry, we will
verify the above conjecture for $g=2$ and lower $n$. Here, we remark
that we have the same results for $A'$ because of the isomorphisms
in Proposition \ref{prop:PFeqs-A-A'} between $A$ and $A'$. 

\para{Contraction to $(2,2,2,2)\subset \mbP^7$.} We observed in (\ref{eq:sumup-rel-AA'})
that BPS numbers of a smoothing of the singular Calabi-Yau variety
$X_{2,2,2,2}^{sing}$ arises from a ``sum-up relation''. This relation
holds also at genus one;
\begin{equation}
n_{1}^{(2,2,2,2)}(d)=\sum_{d_{2}}n_{1}^{A}(d,d_{2})=\sum_{d_{2}}n_{1}^{A'}(d,d_{2}),\label{eq:sumup-rel-g1-AA'}
\end{equation}
where $n_{1}^{(2,2,2,2)}(d)$ can be found in \cite{HKTY2} (see also
Appendix \ref{sec:Appendix-2222}). 

\vskip3cm

\section{\label{sec:MS-B}\textbf{Exploring the parameter space $\protect\mbP_{\Delta}$
globally}}

In this section and the subsequent section, we shall study the boundary
points $B,B'$ and also $C$ (and $C'$) described in Section \ref{sec:P-Delta-global}.
We will find that all aspect of mirror symmetry of $V_{8,w}^{k}\,(k=1,2)$
and their free quotients are encoded in a single system of Picard-Fuchs
equations (\ref{eq:PFeqs-D2-D3}). The entire picture of mirror symmetry
will be summarized in Proposition \ref{prop:MS-Result-Main}. As in
the preceding section, for brevity, we will retain the notation $\mbP_{\Delta}$
even if we will make suitable resolutions.

\subsection{Blowing-up at the boundary points $B$, $B'$ }

The points $B$ and $B'$ are symmetric under the involution (\ref{eq:xy-x1y1-inv}),
and there is no difference in the local analysis around $B$ and $B'$.
Because of this we will restrict our attentions mostly to $B$. 

As is clear in the form of the discriminant $dis_{0}$, the two divisors
$\left\{ dis_{0}=0\right\} $ and $\left\{ y=0\right\} $ intersect
at $B$ with 4th-order tangency. To have power series solutions of
the Picard-Fuchs equation, we blow-up at this point successively four
times; and find that the point $\tilde{B}$ shown in Figure 3 has
properties of a LCSL. In a similar way, we verify that $\tilde{B}'$
is also a LCSL. In what follows, we will use $B$ and $B'$ instead
of $\tilde{B}$ and $\tilde{B}'$ for brevity of notation. 

\begin{figure}
\includegraphics[scale=0.5]{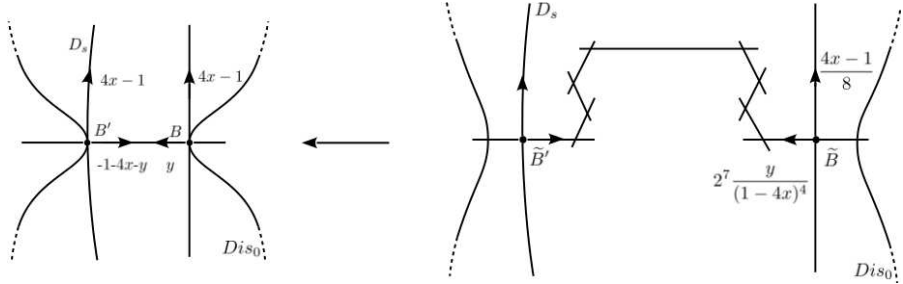}\caption{Fig.3. Blowing-up at $B$ and $B'$. For brevity of notation, $\widetilde{B}$
and $\widetilde{B}'$ are called $B$ and $B'$ again in the text.}

\end{figure}

\para{Picard-Fuchs equations and mirror symmetry.} Let us introduce
$(z_{B}^{1},z_{B}^{2})$ for the blow-up coordinate centered at $B$.
It is related to the affine coordinate $(x^{1},x^{2})=(x,y)$ centered
at $A$ by 
\begin{equation}
(z_{B}^{1},z_{B}^{2})=\big(\frac{1}{8}(1-4x),\,\,2^{7}\frac{y}{(1-4x)^{4}}\big).\label{eq:affine-relation-x-zB}
\end{equation}
It is straightforward to transform the Picard-Fuchs equations $\cD_{2}\omega=\cD_{3}\omega=0$
to this coordinate. The normalization factors in this coordinate are
chosen so that we have a natural coordinate to describe mirror symmetry,
e.g. the integral power series $\omega_{0}^{B}$ in (\ref{eq:w0-w1-for-B})
below. 
\begin{prop}
The boundary points $B$ and $B'$ of the Picard-Fuchs equations over
$\mbP_{\Delta}$ are mirror symmetric to Calabi-Yau manifolds{}{{}
$Y:=V_{8w}^{1}/\mbZ_{8}\times\mbZ_{8}$ and $Y':=V_{8w}^{2}/\mbZ_{8}\times\mbZ_{8}$,}
respectively. 
\end{prop}
We verify the above proposition by making local solutions of the Picard-Fuchs
equations around $B$ and $B'$. We can also confirm this by calculating
Gromov-Witten invariants from $B$ and $B'$. 

\para{Griffiths-Yukawa couplings.} In the same way as Proposition
\ref{prop:GY-Cijk-A}, we can determine the Griffiths-Yukawa couplings
$C_{ijk}^{B}$ up to a normalization constant. Our global description
of the family, however, enables us to determine them uniquely by 
\begin{equation}
C_{ijk}^{B}:=\sum_{l.m.n}C_{lmn}^{A}\frac{\partial x_{l}}{\partial z_{i}^{B}}\frac{\partial x_{m}}{\partial z_{j}^{B}}\frac{\partial x_{n}}{\partial z_{k}^{B}}.\label{eq:Cijk-B}
\end{equation}

We arrange the local solutions around the point $B$ into the canonical
form $\Pi_{B}(z)$ in (\ref{eq:App-Canonial-Pi}). Among the solutions,
the first half of $\Pi_{B}(z)$ is sufficient to define the mirror
map. These solutions have the following explicit forms:
\begin{equation}
\begin{alignedat}{1} & \omega_{0}^{B}(z)=1+8z_{1}+56z_{1}^{2}+384z_{1}^{3}+\cdots-8z_{1}^{4}z_{2}-96z_{1}^{5}z_{2}-\cdots,\\
 & \omega_{1}^{B}(z)=(\log z_{1})\omega_{0}^{B}(z)+(4z_{1}+44z_{1}^{2}+\frac{1120}{3}z_{1}^{3}+\cdots-2z_{1}^{4}z_{2}-\cdots),\\
 & \omega_{2}^{B}(z)=(\log z_{2})\omega_{0}^{B}(z)+(-12z_{1}-136z_{1}^{2}-\cdots+96z_{1}^{5}z_{2}+\cdots).
\end{alignedat}
\label{eq:w0-w1-for-B}
\end{equation}
Here and hereafter, we omit the superscript $B$ in $z_{i}$ for brevity,
unless confusions arise. Mirror map is defined in the same way as
Definition \ref{def:MirrorMap-A} by inverting the relations 

\[
q_{1}=e^{\frac{\omega_{1}}{\omega_{0}}}=C_{1}z_{1}\,\exp\big(\frac{\omega_{1}^{reg}}{\omega_{0}}\big),\,\,\,\,q_{2}=e^{\frac{\omega_{2}}{\omega_{0}}}=C_{2}z_{2}\,\exp\big(\frac{\omega_{2}^{reg}}{\omega_{0}}\big),
\]
where $\omega_{0}=\omega_{0}^{B}$ and $\omega_{i}=\omega_{i}^{B}$
with some constant $C_{k}$. We write the mirror map by $z_{a}=z_{a}(q_{1},q_{2})$
for $a=1,2$. Then the quantum corrected Yukawa couplings are given
by 
\begin{equation}
Y_{ijk}^{B}=\left(\frac{1}{N_{B}\omega_{0}^{B}}\right)^{2}\sum_{a,b,c}C_{abc}^{B}\frac{dz_{a}}{dt_{i}}\frac{dz_{b}}{dt_{j}}\frac{dz_{c}}{dt_{k}},\label{eq:Yijk-B}
\end{equation}
where $N_{B}$ is a constant which we will identify in Proposition
\ref{prop:Conection-Matrices} with the normalization constant of
the local solutions in $\Pi_{B}(z)$. 
\begin{prop}
\label{prop:q-Yijk-B}When we set $N_{B}=\frac{1}{2}$, and $C_{1}=C_{2}=1$
in the definition of the mirror map, we have {\small{}
\[
\begin{alignedat}{2} & Y_{111}^{B}=128+4096q_{1}^{8}+180224q_{1}^{16}q_{2}+\cdots, &  & Y_{112}^{B}=16+512q_{1}^{8}q_{2}+11264q_{1}^{16}q_{2}+\cdots,\\
 & Y_{122}^{B}=64q_{1}^{8}q_{2}+704q_{1}^{16}q_{2}+4160q_{1}^{16}q_{2}^{2}+\cdots, & \, & Y_{222}^{B}=q_{2}+q_{2}^{2}+q_{3}^{3}+\cdots+8q_{1}^{8}q_{2}+.
\end{alignedat}
\]
}We have exactly same form for the corresponding expansions of $Y_{ijk}^{B'}$.
\end{prop}
\begin{table}
{\scriptsize{}
\[
\begin{array}{|c|ccccccccccccccc}
\hline \m j\diagdown i\m & 0 & .. & 8 & .. & 16 & .. & 24 & .. & 32 & .. & 40 & .. & 48 & .. & 56\\
\hline 0 & 0 & .. & 0 & .. & 0 & .. & 0 & .. & 0 & .. & 0 & .. & 0 & .. & 0\\
1 & 1 & .. & 8 & .. & 44 & .. & 192 & .. & 726 & .. & 2464 & .. & 7704 & .. & 22528\\
2 & 0 & .. & 0 & .. & 64 & .. & 1536 & .. & 19072 & .. & 168960 & .. & 1199616 & .. & 7255040\\
3 & 0 & .. & 0 & .. & 44 & .. & 3048 & .. & 98454 & .. & 1950464 & .. & 27697188 & .. & 308823552\\
4 & 0 & .. & 0 & .. & 0 & .. & 1536 & .. & 165120 & .. & 7480320 & .. & 206854144 & .. & 4083891200\\
5 & 0 & .. & 0 & .. & 0 & .. & 192 & .. & 98454 & .. & 11389080 & .. & 637461376 & .. & 22460931072\\
6 & 0 & .. & 0 & .. & 0 & .. & 0 & .. & 19072 & .. & 7480320 & .. & 918183744 & .. & 59571908608\\
7 & 0 & .. & 0 & .. & 0 & .. & 0 & .. & 726 & .. & 1950464 & .. & 637461376 & .. & 81827379400\\
: & : & .. & : & .. & : & .. & : & .. & : & .. & : & .. & : & .. & :
\end{array}
\]
}{\scriptsize \par}

(0) Genus zero BPS numbers $n_{0}^{B}(i,j)$.

{\scriptsize{}
\[
\hsp{0}\begin{array}{|c|cccccccccccccccccccc}
\hline \m j\diagdown i\m & 0 & \m1 & \m2 & \m3 & \m4 & \m5 & \m6 & \m7 & \m8 & \m9 & \m10 & \m11 & \m12 & \m13 & \m14 & \m15 & \m16 & \m17 & \m18 & \m19\\
\hline 0 & 0 & \m8 & \m8 & \m8 & \m8 & \m8 & \m8 & \m8 & \m8 & \m8 & \m8 & \m8 & \m8 & \m8 & \m8 & \m8 & \m8 & \m8 & \m8 & \m8\\
1 & 0 & \m0 & \m0 & \m0 & \m14 & \m16 & \m48 & \m64 & \m152 & \m128 & \m288 & \m256 & \m728 & \m592 & \m1248 & \m1216 & \m2748 & \m2112 & \m4224 & \m3584\\
2 & 0 & \m0 & \m0 & \m0 & \m0 & \m0 & \m0 & \m0 & \m8 & \m24 & \m112 & \m224 & \m704 & \m848 & \m2368 & \m2752 & \m7872 & \m7944 & \m19520 & \m21472\\
3 & 0 & \m0 & \m0 & \m0 & \m0 & \m0 & \m0 & \m0 & \m0 & \m0 & \m0 & \m0 & \m14 & \m48 & \m288 & \m656 & \m2748 & \m3648 & \m12528 & \m17408\\
4 & 0 & \m0 & \m0 & \m0 & \m0 & \m0 & \m0 & \m0 & \m0 & \m0 & \m0 & \m0 & \m0 & \m0 & \m0 & \m0 & \m8 & \m72 & \m600 & \m1728\\
5 & 0 & \m0 & \m0 & \m0 & \m0 & \m0 & \m0 & \m0 & \m0 & \m0 & \m0 & \m0 & \m0 & \m0 & \m0 & \m0 & \m0 & \m0 & \m0 & \m0\\
6 & 0 & \m0 & \m0 & \m0 & \m0 & \m0 & \m0 & \m0 & \m0 & \m0 & \m0 & \m0 & \m0 & \m0 & \m0 & \m0 & \m0 & \m0 & \m0 & \m0
\end{array}
\]
}{\scriptsize \par}

(1) Genus one BPS numbers $n_{1}^{B}(i,j)$.

{\scriptsize{}
\[
\begin{array}{|c|ccccccccccccccccc}
\hline \m j\diagdown i\m & 0 & \m1 & \m.. & \m5 & \m6 & \m7 & \m8 & \m9 & \m10 & \m11 & \m12 & \m13 & \m14 & \m15 & \m16 & \m17 & \m18\\
\hline 0 & 0 & \m0 & \m.. & \m0 & \m0 & \m0 & \m0 & \m0 & \m0 & \m0 & \m0 & \m0 & \m0 & \m0 & \m0 & \m0 & \m0\\
1 & 0 & \m0 & \m.. & \m0 & \m4 & \m32 & \m155 & \m456 & \m1304 & \m2784 & \m6336 & \m11120 & \m22440 & \m35360 & \m66056 & \m97752 & \m172768\\
2 & 0 & \m0 & \m.. & \m0 & \m0 & \m0 & \m0 & \m128 & \m696 & \m2304 & \m7416 & \m18704 & \m51168 & \m110400 & \m265376 & \m504976 & \m1115408\\
3 & 0 & \m0 & \m.. & \m0 & \m0 & \m0 & \m0 & \m0 & \m0 & \m0 & \m0 & \m768 & \m5112 & \m19872 & \m74760 & \m211208 & \m645364\\
4 & 0 & \m0 & \m.. & \m0 & \m0 & \m0 & \m0 & \m0 & \m0 & \m0 & \m0 & \m0 & \m0 & \m0 & \m0 & \m2944 & \m23552\\
5 & 0 & \m0 & \m.. & \m0 & \m0 & \m0 & \m0 & \m0 & \m0 & \m0 & \m0 & \m0 & \m0 & \m0 & \m0 & \m0 & \m0\\
: & : & : & : & : & : & : & : & : & : & : & : & : & : & : & : & : & :
\end{array}
\]
}{\scriptsize \par}

(2) Genus two BPS numbers $n_{2}^{B}(i,j)$.

~

\caption{Table 2. \label{tab:Table-B} BPS numbers $n_{g}^{B}(i,j)$ of $Y=V_{8,w}^{1}/\protect\mbZ_{8}\times\protect\mbZ_{8}$.
The blanks ``$..$'' represent vanishing invariants. These are equal
to $n_{g}^{B'}(i,j)$ of $V_{8,w}^{2}/\protect\mbZ_{8}\times\protect\mbZ_{8}$.}
\end{table}

\subsection{Genus one Gromov-Witten potentials $F_{1}^{B}$ and $F_{1}^{B'}$ }

We determine the genus one Gromov-Witten potentials $F_{1}^{B}$ and
$F_{1}^{B'}$ using the general form $F_{1}^{M}$ given in (\ref{eq:F1-X-formula}).
Since $B$ and $B'$ are isomorphic locally, we only describe $F_{1}^{B}$. 

\para{Calculating $F_1^B$.} To apply the general formula (\ref{eq:F1-X-formula}),
we introduce the following definitions:
\[
\begin{alignedat}{1}disB_{0}=1-(1-4z_{1})(1-8z_{1})z_{2}-16z_{1}^{4}(1-8z_{1})z_{2}^{2}\qquad\\
disB_{1}=1-4z_{1}+16z_{1}^{4}z_{2},\,\,\,disB_{2}=1-4z_{1},\,\,\,disB_{3}=1-8z_{1}.
\end{alignedat}
\]
As in Subsection \ref{sub:F1-A-A'}, knowing that $B$ corresponds
to the free quotient $V_{8,w}^{1}/\mbZ_{8}\times\mbZ_{8}$, we can
determine the parameters $r_{0}$ and $r_{k}$ in the BCOV formula. 
\begin{prop}
Near the boundary point $B$, the BCOV potential function has the
following form
\[
F_{1}^{B}(q)=\frac{1}{2}\log\Big\{\big(\frac{1}{N_{B}\omega_{0}^{B}}\big)^{5}\frac{\partial(z_{1},z_{2})}{\partial(t_{1},t_{2})}disB_{0}^{-\frac{1}{6}}\,disB_{1}^{-1}\,disB_{2}^{-\frac{2}{3}}\,disB_{3}^{-\frac{19}{3}}z_{1}^{-1-\frac{8}{12}}z_{2}^{-1}\Big\}.
\]

\end{prop}
Assuming mirror symmetry, we read the genus one Gromov-Witten invariants
from the potential function. In Table \ref{tab:Table-B} (1), we have
listed the resulting BPS numbers. 

\para{Connecting property for $F_1^A$ and $F_1^B$.}\label{para: connecting-prop-g1}
Two potential functions $F_{1}^{A}$ and $F_{1}^{B}$ are defined
independently near the corresponding boundary points. Here we describe
how these two functions are related on the parameter space. Let us
note that the potential function $F_{1}^{M}$ in general consists
three parts; (1) the Hodge factor $\left(\omega_{0}\right)^{-(3+h^{1,1}+\frac{\chi}{12})}$
(coming from Hodge bundle), (2) the Jacobian part, and (3) rational
functions. Rational functions are easily transformed from $A$ to
$B$. For the Hodge factor (1) and the Jacobian factor (2), we set
the following transformation rules 
\begin{equation}
\begin{alignedat}{1} & ({\rm t1})\;\;\left(N_{A}\omega_{0}^{A}\right)^{-(3+h_{A}^{1,1}+\frac{\chi_{A}}{12})}=\left(N_{B}\omega_{0}^{B}\right)^{-(3+h_{B}^{1,1}+\frac{\chi_{B}}{12})},\\
 & ({\rm t}2)\;\;\frac{\partial(x_{1},x_{2})}{\partial(t_{1}^{A},t_{2}^{A})}=\frac{\partial(z_{1}^{B},z_{2}^{B})}{\partial(t_{1}^{B},t_{2}^{B})}\frac{\partial(x_{1},x_{2})}{\partial(z_{1}^{B},z_{2}^{B})},
\end{alignedat}
\label{eq:connect-g1}
\end{equation}
where we attached superscripts to indicate the two boundaries points.
We say that $F_{1}^{A}$ is connected to $F_{1}^{B}$, or vice versa,
if two are transformed under the above transformation rules.
\begin{prop}
The potential functions $F_{1}^{A}$ and $F_{1}^{B}$ defined near
the boundary point $A$ and $B$, respectively, are connected to each
other. \end{prop}
\begin{proof}
Because of our definition, the connection property may be verified
simply by transforming the discriminants in the coordinate $(x,y)=(x_{1},x_{2})$
to the blow-up coordinate $(z_{1},z_{2})=(z_{1}^{B},z_{2}^{B})$,
and by evaluating the Jacobian. For the discriminants, we have the
following relations:
\[
\begin{aligned}dis_{0}(x,y)=z_{1}^{4}\,disB_{0}(z_{1},z_{2}),\,\,\,\,dis_{1}(x,y)=2\,disB_{1}(z_{1},z_{2})\qquad\qquad\\
dis_{2}(x,y)=2\,disB_{2}(z_{1},z_{2}),\,\,dis_{3}(x,y)=8\,z_{1},\,\,x=\frac{1}{4}\,disB_{3}(z_{1},z_{2}),\,\,y=32\,z_{1}^{4}z_{2}.
\end{aligned}
\]
For the Jacobian, we have $\frac{\partial(x,y)}{\partial(z_{1},z_{2})}=z_{1}^{4}$.
Under these relations, we can verify that the exponents $r_{k}$ in
$F_{1}^{A}$ exactly match those in $F_{1}^{B}$. This verifies the
claimed connection property between $F_{1}^{A}$ and $F_{1}^{B}$.\end{proof}
\begin{rem}
We can also argue that $F_{1}^{B}$ is calculated by the same family
$\cV_{\mbZ_{8}}^{1}\rightarrow\mbP_{\Delta}$ as $F_{1}^{A}$ by looking
at the conifold factor in $disB_{0}^{-\frac{1}{6}}$. See Remark \ref{rem:Conifold-A}.
$\hfill\square$ 
\end{rem}
\para{Contraction to $(2,2,2,2)/G$.} \label{para:contR-B-Z8Z8}Calculations
for the boundary point $B$ apply word by word to another boundary
point $B'$; we arrive at the same results as $B$ since the local
properties are isomorphic. Under mirror symmetry, these boundary points
can be identified with the birational models in the upper line of
the diagram\def\contRii{\begin{xy}
(-30,0)*++{V^1_{8,w}/G}="ViG",
( 30,0)*++{V^2_{8,w}/G}="ViiG",
(  0,0)*++{X_{2,2,2,2}^{sing}/G}="ZiG",
(-30,-15)*++{V^1_{8,w}}="Vi",
( 30,-15)*++{V^2_{8,w}}="Vii",
(  0,-15)*++{X_{2,2,2,2}^{sing}}="Zi",
\ar^{{64 \,\mbP^1}\hsp{10}} "Vi";"Zi"
\ar_{\hsp{20}{64 \,\mbP^1}} "Vii";"Zi"
\ar^{{1 \,\mbP^1}\hsp{10}} "ViG";"ZiG"
\ar_{\hsp{20}{1 \,\mbP^1}} "ViiG";"ZiG"
\ar "Vi";"ViG"
\ar "Vii";"ViiG"
\end{xy} }
\[
\begin{matrix}\contRii\end{matrix}
\]
where $G=\mbZ_{8}\times\mbZ_{8}$. We have identified the boundary
points $A$ and $A'$ with the birational models in the lower line.
For the lower line, we remark that there is a smoothing to general
$(2,2,2,2)$ complete intersections in $\mbP^{7}$. This explained
the sum-up properties (\ref{eq:sumup-rel-AA'}) and (\ref{eq:sumup-rel-g1-AA'})
of the BPS numbers.

In contrast to $X_{2,2,2,2}^{sing}$, the singular Calabi-Yau variety
$X_{2,2,2,2}^{sing}/G$ does not have a smoothing \cite{Namikawa}.
This is consistent to the classification result in \cite{Brown,Hua}
which says that there is no free quotient of $(2,2,2,2)$ complete
intersection by the group $\mbZ_{8}\times\mbZ_{8}$ which admits a
smooth Calabi-Yau manifold. Nevertheless we observe the following
sum-up property 

\[
\sum_{d_{2}}n_{0}^{B}(d,d_{2})=\frac{1}{|G|}n_{0}^{(2,2,2,2)}(d),
\]
which indicates that the number $n_{0}^{(2,2,2,2)/G}(d):=$$\frac{1}{|G|}n_{0}^{(2,2,2,2)}(d)$
has some meaning as BPS numbers of the singular variety $X_{2,2,2,2}^{sing}/G$.
Interestingly, we can also verify the sum-up property at genus one
\[
\sum_{d_{2}}n_{1}^{B}(d,d_{2})=n_{1}^{(2,2,2,2)/G}(d),
\]
with the numbers $n_{1}^{(2,2,2,2)/G}(d)$ which we obtain from the
BCOV formula, see Appendix \ref{sub:Appendix-free-quot-Z8Z8}. 

\para{Generators of $H_2(V^1_{8,w}/G,\mbZ)$.} \label{para:PicG-Z8Z8}
{}{Let us write $Y=V_{8,w}^{1}/G$ with $G=\mathbb{Z}_{8}\times\mathbb{Z}_{8}$,
and recall that $Y$ is isomorphic to the dual fibration by $(1,8)$-polarized
abelian surfaces \cite[Lem.5.4]{Sch}. We denote the class of a dual
abelian surface by $A_{Y}$, and set $H_{Y}$ to be a relative ample
divisor on $Y/\mathbb{P}^{1}$ which comes from the relative ample
divisor $H_{X}$ on $V_{8,w}^{1}/\mathbb{P}^{1}$ by the fiberwise
Fourier-Mukai transformation (see \cite[Thm.4.1]{BL2}). Let us denote
by $\sigma_{Y}$ the section of the fibration $Y\rightarrow\mathbb{P}^{1}$
which is a single $G$-orbit of the 64 sections of the fibration $X=V_{8,w}^{1}\rightarrow\mathbb{P}^{1}$
in (4) of \ref{para: V1-summary}. As summarized in (3),(4) of \ref{para: V1-summary},
each singular fiber of $X$ is given by an elliptic translation scroll
of an elliptic curve $E_{X}$, and $G$ acts on each singular fiber
by the natural translations by 8-torsion points of $E_{X}$. We denote
by $E_{Y}$ an elliptic curve in $Y$ coming from a singular fiber
of $X$, i.e., $E_{X}/G$. From Table \ref{tab:Table-B}, we will
read below the BPS number $n_{0}^{B}(0,1)=1$ as counting the section
$\sigma_{Y}$, and also the number $n_{1}^{B}(1,0)=8$ as counting
the elliptic curves $E_{Y}$. We will start studying a basis of the
group $H_{2}(V_{8,w}^{1}/G,\mbZ)$ modulo torsions.}

{}{Calabi-Yau manifold $V_{8,w}^{1}$ is simply connected,
since it is defined by a small resolution of a $(2,2,2,2)$ complete
intersection. Hence the fundamental group of the free quotient $V_{8,w}^{1}/G$
is isomorphic to $G=\mbZ_{8}\times\mbZ_{8}$. Now we have the exact
sequence due to Eilenberg-MacLane \cite{Eilenberg-Mac}, 
\[
1\rightarrow\pi_{2}(V_{8,w}^{1}/G)\rightarrow H_{2}(V_{8,w}^{1}/G,\mbZ)\rightarrow H_{2}(G)\rightarrow0,
\]
where $H_{2}(G)=H_{2}(\pi_{1}(V_{8,w}^{1}/G))$ is the group homology,
and we can calculate this as $H_{2}(G)\simeq\mbZ_{8}$ for $G=\mbZ_{8}\times\mbZ_{8}$.
In \cite{AM1}, by finding an elliptic curve as a generator of $H_{2}$,
it was argued that the corresponding exact sequence does not split
for a free quotient of a quintic Calabi-Yau threefold by $\mathbb{Z}_{5}\times\mathbb{Z}_{5}$.
As we will argue below, properties of $\sigma_{Y},E_{Y}$ indicate
that the above exact sequence does not split in our case, too. Here,
it should be useful to contrasted this to the case $V_{8,w}^{1}$
where we have an isomorphism $\pi_{2}(V_{8,w}^{1})\simeq H_{2}(V_{8,w}^{1},\mbZ)$,
since $H_{2}(\pi_{1}(V_{8,w}^{1}))=0$. In this case, we have a basis
of $H_{2}(V_{8,w}^{1},\mbZ)$ consisting of a section $\sigma_{X}$
and a line $\ell$ in a singular fiber (see \ref{para: V1-summary}),
which are both rational curves. }

{}{Let $\pi^{*}H_{Y},\pi^{*}A_{Y}$ be the pull-backs
of $H_{Y},\,A_{Y}$ by $\pi:V_{8,w}^{1}\rightarrow V_{8,w}^{1}/G$.}
\begin{prop}
\label{prop:Pull-back-B}It holds that $\pi^{*}H_{Y}=8\,H_{X}$ and
$\pi^{*}A_{Y}=A_{X}$.\end{prop}
\begin{proof}
Since the group action on $V_{8,w}^{1}$ comes from the Heisenberg
group $\cH_{8}$ acting on abelian surfaces, the fiber class $A_{Y}$
pull-backs to the fiber class $A_{X}$. For the first equality, let
us write $\pi^{*}H_{Y}=aH_{X}+bA_{X}$. We recall that $\sigma_{X}$
is a section of $X:=V_{8,w}^{1}\rightarrow\mbP^{1}$ (see \ref{para: V1-summary}).
Then we have $\pi^{*}\sigma_{Y}=64\sigma_{X}$ for their homology
classes. Now $b=0$ follows from $\pi^{*}\sigma_{Y}.\pi^{*}H_{Y}=(\deg\pi)\sigma_{Y}.H_{Y}=0$.
To determine $a$, we note that $V_{8,w}^{1}/G$ is fibered by dual
abelian surfaces which have (1,8) polarization again; hence we have
$H_{Y}^{2}.A_{Y}=H_{X}^{2}.A_{X}=16.$ Using this we have $(\deg\pi)H_{Y}^{2}A_{Y}=(\pi^{*}H_{Y})^{2}.\pi^{*}A_{Y}=a^{2}H_{X}^{2}.A_{X}$
with $\deg\pi=64$. This determines $a=8$. 
\end{proof}
{}{As described above, we have $E_{Y}=E_{X}/G$, hence
$\pi^{*}E_{Y}=E_{X}$. Then, calculating $(\deg\pi)\,H_{Y}.E_{Y}$
by $\pi^{*}H_{Y}.\pi^{*}E_{Y}=8\,H_{X}.E_{X}=64$ for example, we
obtain 
\begin{equation}
H_{Y}.\sigma_{Y}=0,\,\,H_{Y}.E_{Y}=1;\quad A_{Y}.\sigma_{Y}=1,\,\,A_{Y}.E_{Y}=0,\label{eq:HY-AY}
\end{equation}
which show that $\sigma_{Y},E_{Y}$ generate $H_{2}(V_{8,w}^{1}/G,\mbZ)$
modulo torsion, and also $H_{Y},A_{Y}$ generate $Pic(V_{8,w}^{1}/G)$
modulo torsions. In terms of these basis, we can read the BPS numbers
$n_{g}^{B}(i,j)$ in Table \ref{tab:Table-B} by 
\[
n_{g}^{B}(i,j)=n_{g}^{B}(\beta.H_{Y},\beta.A_{Y})\,\,(\beta=i\,E_{Y}+j\,\sigma_{Y}).
\]
We justify this by reproducing the constant terms of $Y_{ijk}^{B}$
in Proposition \ref{prop:q-Yijk-B} as
\[
H_{Y}^{3}=128,\,\,H_{Y}^{2}A_{Y}=16,\,\,H_{Y}A_{Y}^{2}=A_{Y}^{3}=0,
\]
where we use $\pi^{*}D_{i}\,\pi^{*}D_{j}\,\pi^{*}D_{k}=(\deg\pi)\,D_{i}D_{j}D_{k}$
for divisors $D_{i}$ on $Y$. }

{}{Now, let us note that both classes $\sigma_{Y}$ and
$8E_{Y}$ represent homology classes of rational curves modulo torsions
as follows: We have $H_{Y}.(\pi_{*}\ell)=(\pi^{*}H_{Y}).\ell=8$ and
$A_{Y}.(\pi_{*}\ell)=0$ for the class of a line $\ell$ contained
in a singular fiber of $V_{8,w}^{1}$, namely $\pi_{*}\ell=8\,E_{Y}$
is a class of rational curve. Since the lines $\pi_{*}\ell$ in (the
quotient of) each singular fiber are parametrized by an intersection
point with $E_{Y}$, the BPS number of curves of class $\beta=\pi_{*}\ell$
is counted by $(-1)^{1}\chi(E_{Y})=0$ according to the counting rule
of BPS numbers \cite{GV}. We identify this counting number with $n_{0}^{B}(8,0)=0$
in Table \ref{tab:Table-B}. We also read the number $n_{0}^{B}(8,1)=8$
as counting reducible curves of class $\beta=8E_{Y}+\sigma_{Y}$ which
come from eight singular fibers. }

{}{We can now argue that the exact sequence above does
not split as follows (as in \cite{AM1}): If the exact sequence will
split, then $\pi_{2}(V_{8,w}^{1}/G)$ modulo torsions is isomorphic
to $H_{2}(V_{8,w}^{1}/G,\mbZ)$ modulo torsions, hence the class of
$E_{Y}$ must belong to the image of $\pi_{2}(V_{8,w}^{1}/G)$ modulo
torsions. The last property seems unlikely, although we need a proof
to complete the argument. }

\para{Relating $g=0$ BPS numbers in Tables \ref{tab:Table-A} and \ref{tab:Table-B}.}\label{para:B-g0}
As summarized in (4) of \ref{para: V1-summary}, the group $G=\mbZ_{8}\times\mbZ_{8}$
acts on the 64 sections in $V_{8,w}^{1}$ making them into a single
orbit. Clearly, the number $n_{0}(0,1)=1$ counts this section. This
is also the case for any smooth rational curves in $V_{8,w}^{1}$,
i.e., the free action of $G$ must identify 64 of them as a single
rational curve in $V_{8,w}^{1}/G$ since any smooth rational curve
in $V_{8,w}^{1}$ cannot be stable under the free action. Let $C_{Y}$
be a rational curve in $V_{8,w}^{1}/G$ and write $\pi^{*}C_{Y}$$=64C$.
Then we have
\[
C_{Y}.H_{Y}=(\pi_{*}C).H_{Y}=C.(\pi^{*}H_{Y})=8\,C.H_{X}.
\]
In a similar way, we have $C_{Y}.A_{Y}=C.A_{X}$. These explain the
degree distribution of non-vanishing BPS numbers in Table \ref{tab:Table-B}
(0) and also the numbers which are exactly $\frac{1}{64}$ of Table
\ref{tab:Table-A} (0). 

\para{$g\geq 1$ BPS numbers in Tables \ref{tab:Table-A} and \ref{tab:Table-B}.}
Corresponding to (\ref{eq:F0-A}),(\ref{eq:F1-A}), let us introduce
the $q$-series $Z_{g,n}^{B}$ by 
\[
\begin{aligned} & F_{0}^{B}=\frac{128}{3!}\,t_{1}^{3}+\frac{16}{2!}t_{1}^{2}t_{2}+\sum_{n\geq1}Z_{0,n}^{B}(q_{1})q_{2}^{n},\\
 & F_{1}^{B}=-\frac{c_{2}.H_{Y}}{24}t_{1}+Z_{1,n}^{B}(q_{1})+\sum_{n\geq0}Z_{1,n}^{B}(q_{1})\,q_{2}^{n},
\end{aligned}
\]
where $t_{k}=\log q_{k}$. The observation in \ref{para:B-g0} implies
the following relation at $g=0$: 
\begin{equation}
Z_{0,n}^{B}(q)=\frac{1}{64}Z_{0,n}^{A}(q^{8})\,\,\,\,\,(n\geq1).\label{eq:Zg0B-A-relation}
\end{equation}
The relations between the BPS numbers in Table \ref{tab:Table-A}
(1) and those in Table \ref{tab:Table-B} (1) seem more complicated,
but it is easy to see a relation when $n=0$, 
\[
\widehat{Z}_{1,0}^{B}(q_{1})=-\frac{c_{2}.H_{Y}}{24}t_{1}-8\,\log\,\bar{\eta}(q_{1})=-8\,\log\,\eta(q_{1}).
\]
{}{By making $q$-series expansions up to sufficiently
high degrees, }we observe the following property (see \cite{HT-math-c}):

\begin{Xobs}{\label{obs:Z11-B}The $q$-series $Z_{1,1}^{B}(q)$
is expressed by 
\[
Z_{1,1}^{B}(q)=P_{1,1}^{B}(q)\,\frac{1}{\bar{\eta}(q^{8})^{8}}
\]
using exactly the same polynomial $P_{1,1}^{B}(q)=P_{1,1}^{A}(q)$
in terms of the quasi-modular forms in Observation \ref{obs:Z11-A}.
}\end{Xobs}

Corresponding to Conjecture \ref{conj:Zgn-A}, we naturally come to
the following 
\begin{conjecture}
\label{conj:Zgn-B}The $q$-series $Z_{g,n}^{B}(q)\,\,(n\geq1)$ are
expressed by 
\[
Z_{g,n}^{B}(q)=P_{g,n}^{B}(E_{2},S,T,U)\left(\frac{1}{\bar{\eta}(q^{8})^{8}}\right)^{n},
\]
where $P_{g,n}^{B}$ are polynomials of degree $2(g+n-1)$ of $E_{2},S,T,U$. 
\end{conjecture}
We verify the above conjecture determining polynomials $P_{1,n}^{B}$
for $n\leq3$ (see Appendix \ref{sec:Appendix-Pgn-AB} for some of
them). 
\begin{rem}
(1){}{{} In \cite{HT-math-c},} we verify $P_{1,n}^{B}\not=P_{1,n}^{A}$
for $2\leq n\leq3$, and observe that the equality $P_{1,1}^{B}=P_{1,1}^{A}$
in Observation \ref{obs:Z11-B} holds only for $n=1$. In the next
subsection, we will find that $P_{g,1}^{B}=P_{g,1}^{A}$ holds also
for $g=2$.

\noindent(2) We can verify Conjecture \ref{conj:Zgn-B} at $g=0$
up to $n=9$ {}{(see \cite{HT-math-c})}. For example,
assuming the general form of $Z_{0,n}^{B}(q)$, we obtain the following
polynomials 
\[
\begin{alignedat}{2}P_{0,2}^{B} &  &  & =\frac{E{}_{2}^{2}}{576}+\frac{E_{2}}{288}(S+2T+4U)+\frac{1}{384}(S^{2}+10ST+3T^{2}+4SU+20TU)\\
P_{0,3}^{B} &  &  & =\frac{7E{}_{2}^{4}}{2^{14}3^{4}}+\frac{7E{}_{2}^{3}}{2^{12}3^{4}}(S+2T+4U)+\m\frac{7E{}_{2}^{2}}{2^{13}3^{4}}(4S^{2}\m\,+\m\,50ST\m\,+\m\,13T^{2}\m\,+\,\m20SU\,\m+\,\m76TU)\\
 &  &  & \,\,\,\,+\frac{E_{2}}{2^{12}3^{4}}(16S^{3}\m\,+\m\,300S^{2}T\m\,+\m\,1605ST^{2}\m\,+\m\,68T^{3}\m\,+\m\,72S^{2}U\m\,+\m\,364STU\m\,+\m\,888T^{2}U)\\
 &  &  & \,\,\,\,+\frac{1}{2^{14}3^{4}}(37S^{4}\m\,+\m\,572S^{3}T\m\,+\m\,214S^{2}T^{2}\m\,+\m\,18704ST^{3}\m\,+\m\,160T^{4}\m\,+\m\,88S^{3}U\\
 &  &  & \hsp{150}+\m\,2152S^{2}TU\m\,+\m\,7328ST^{2}U\m\,+\m\,4160T^{3}U).
\end{alignedat}
\]
Then the observation (\ref{eq:Zg0B-A-relation}) relating $Z_{0,n}^{B}$
to $Z_{0,n}^{A}$ implies the following equalities
\begin{equation}
P_{0,n}^{B}(q)=(64)^{n-1}P_{0,n}^{A}(q^{8})\,\,\,(n\geq1),\label{eq:PB-PA-rel}
\end{equation}
which we can verify by using the identity 
\[
8E_{2}(q^{8})-E_{2}(q)=S+2T+4U
\]
and expressions of $E_{4}(q^{8})$ and $E_{6}(q^{8})$ in terms of
the theta functions $S,T$ and $U$.$\square$
\end{rem}

\subsection{\label{sub:Fg2-AB}Genus two Gromov-Witten potentials $F_{2}^{A}$
and $F_{2}^{B}$}

The BCOV potential $F_{1}$ in (\ref{eq:F1-X-formula}) has its higher
genus generalizations $F_{g}\,(g\geq2)$, which are determined recursively
with initial data $F_{0}$ and $F_{1}$. The recursion relations arise
as solutions of the so-called BCOV holomorphic anomaly equation \cite{BCOV2},
which describes $F_{g}$ up to unknown holomorphic (rational) function
$f_{g}$. There is no general recipe to determine $f_{g}$; however,
global boundary conditions may restrict its possible form and determine
it completely for lower $g$ in some cases. 

\para{BCOV recursion formula $F_2$.} For Calabi-Yau manifolds with
vanishing Euler numbers, the BCOV recursion relation simplifies. Suppose
that potential functions $F_{0}$ and $F_{1}$ are given near a boundary
point $P$, in the present case, $A$ or $B$. We calculate the Yukawa
couplings (three point functions) and also four point functions by
\[
Y_{abc}=\partial_{a}\partial_{b}\partial_{c}F_{0},\,\,\,Y_{abcd}=\partial_{a}\partial_{b}\partial_{c}\partial_{d}F_{0},
\]
where $\partial_{a}:=\frac{\partial\;}{\partial t_{a}}$. Similarly,
we define one point and two point functions, $\partial_{a}F_{1}$,
$\partial_{a}\partial_{b}F_{1}$ at $g=1$. Using these functions,
the recursion relation for $F_{2}$ is given by 
\begin{equation}
\begin{aligned} &  &  & F_{2}=\frac{1}{2}\sum S^{ab}(\partial_{a}\partial_{b}F_{1}+\partial_{a}F_{1}\partial_{b}F_{1})-\frac{1}{4}\sum S^{ab}S^{cd}(\frac{1}{2}Y_{abcd}+2Y_{abc}\partial_{d}F_{1})\\
 &  &  & +\frac{1}{8}\sum S^{ab}S^{cd}S^{rs}Y_{acd}Y_{brs}+\frac{1}{12}\sum Y_{acd}Y_{brs}S^{ab}S^{cr}S^{ds}+\big(N_{P}\omega_{0}(x)\big)^{2}f_{2}
\end{aligned}
\label{eq:F2-BCOV}
\end{equation}
where $S^{ab}$ is a certain (contra-variant) tensor called propagator
and $f_{2}=f_{2}(x,y)$ is the holomorphic (rational) function which
we need to determine. 

\para{Propagator $S^{ab}$ and $f_2$ at the boundary points.} \label{para:F2-t1-t2}To
apply the above BCOV formula, we have to find the propagator $S^{ab}$
by solving a curvature relation in Weil-Petersson geometry on the
moduli space of Calabi-Yau manifolds. To avoid going into the details,
we present the resulting forms $S_{A}^{ab}$ and $S_{B}^{ab}$ for
each boundary point in Appendix \ref{sec:Appendix-Propagator-f2}.
With the data of propagator $S^{ab}$, the potential functions $F_{0},F_{1}$
and the unique period integral $\omega_{0}(x)$ at a boundary point,
the BCOV formula gives the genus two potential function $F_{2}^{M}$
($M=A$ or $B$) up to unknown function $f_{2}^{M}$. If we find the
function $f_{2}^{M}$ in some way, the potential function $F_{2}^{M}$,
as the generating function of Gromov-Witten invariants, has the following
expansion:
\[
F_{2}^{M}(q_{1},q_{2})=\frac{\chi}{5760}+\sum_{d_{1},d_{2}\geq0}N_{2}^{M}(d_{1},d_{2})\,q_{1}^{d_{1}}q_{2}^{d_{2}}\,\,\big(=:\sum_{n\geq0}Z_{2,n}^{M}(q_{1})q_{2}^{n}\,\,\big)
\]
where $\chi(=0)$ is the Euler number of $M$ and $\frac{1}{5760}$
is the (orbifold) Euler number of the moduli space $\mathcal{M}_{2}$
of genus two stable curves. Then the BPS numbers $n_{2}^{M}(d_{1},d_{2})=:n_{2}^{M}(\beta)$
are read by the relation \cite{GV} 
\[
N_{2}^{M}(\beta)=\sum_{k|\beta}\Big\{ n_{0}^{M}(\beta/k)\frac{k}{240}+n_{2}^{M}(\beta/k)\,k\Big\}.
\]
To find $f_{2}^{M}$ for $M=A,B$, we read from the Tables \ref{tab:Table-A}
(1) and Table \ref{tab:Table-B} (1) that we may expect vanishing
BPS numbers $n_{2}^{M}(d_{1},d_{2})$=0 for $d_{1}\leq3$. Expecting
these vanishing BPS numbers, we set the following ansatz for the possible
form of $f_{2}^{A}$: 
\begin{equation}
f_{2}^{A}(x,y)=\frac{\sum_{0\leq i\leq14,1\leq j\leq3}a_{ij}x^{i}(y(1+4x+y))^{j}}{(dis_{0}dis_{2}dis_{3})^{2}}\label{eq:f2A-ansatz}
\end{equation}
with unknown constants $a_{ij}$, and a similar ansatz for $f_{2}^{B}(z_{1},z_{2})$
in terms of $disB_{k}$ and parameters $b_{ij}$. We may attempt to
impose further vanishing $n_{2}^{M}(d_{1},d_{2})=0$ for $d_{1}\leq5$;
however even if we do so, it turns out that these vanishing conditions
$n_{2}^{M}(d_{1},d_{2})=0\,(d_{1}\leq5)$ imposed for $M=A$ and $B$
independently do not suffice to determine the forms $f_{2}^{A}$ and
$f_{2}^{B}$. However, if we assume that $f_{2}^{A}$ and $f_{2}^{B}$
represent the same rational function $f_{2}$ on $\mbP_{\Delta}$,
then it turns out that the vanishing conditions suffice to determine
the possible form of $f_{2}$.

To describe the requirements more precisely, we set up transformation
rules on the expression (\ref{eq:F2-BCOV}) and $F_{g}$ in general;
for $g\geq2$, we extend (t1), (t2) given in (\ref{eq:connect-g1})
to
\[
\begin{aligned} & ({\rm t1})'\,\,\,(N_{A}\omega_{0}^{A})^{2g-2}\,\,f_{g}^{A}(x,y)=(N_{B}\omega_{0}^{B})^{2g-2}\,\,f_{g}^{B}(z_{B}^{1},z_{B}^{2}),\\
 & ({\rm t2})'\,\,\text{ \ensuremath{S^{ab},Y_{abc},\partial_{a}F_{1},\,}etc. in }F_{g}\,\text{transform covariantly as tensors.}
\end{aligned}
\]
In BCOV theory, the covariance (t2)$'$ is explained based on the
fact that the coordinates $t_{A}^{a}$ and $t_{B}^{b}$ are the so-called
flat coordinate defined near boundary points \cite{BCOV2}. {}{We
say that}\textit{ $F_{g}^{A}$ and $F_{g}^{B}$ are related to each
other on $\mbP_{\Delta}$} if the form of $F_{g}^{B}$ is obtained
from $F_{g}^{A}$ by the above transformation rules and the rationality
requirement, i.e. $f_{2}^{A}(x,y)=f_{2}^{B}(z_{B}^{1},z_{B}^{2})$
for a rational function $f_{2}^{A}$ on $\mathbb{P}_{\Delta}$. 
\begin{prop}
\label{prop:FAB-related-on-PD}If we assume that $F_{2}^{A}$ and
$F_{2}^{B}$ are related to each other on $\mbP_{\Delta}$, then the
rational function $f_{2}^{A}(x,y)$ is determined uniquely from the
vanishing conditions $n_{2}^{M}(d_{1},d_{2})=0\,(d_{1}\leq5)$ for
$M=A,B$.\end{prop}
\begin{proof}
We first assume the form $f_{2}^{A}$ given in (\ref{eq:f2A-ansatz}).
Then $f_{2}^{B}$ follows from $f_{2}^{A}$ by the relation (\ref{eq:affine-relation-x-zB});
and we use the BCOV formula (\ref{eq:F2-BCOV}) with $N_{A}=1$ and
$N_{B}=\frac{1}{2}$. We read the BPS numbers $n_{2}^{M}$, which
contain unknown parameters, from the $q,p$-series expansion of $F_{2}^{M}$
using mirror maps for $M=A$ and $B$. Imposing the vanishing conditions,
we find the unique form as claimed. In Appendix \ref{sec:Appendix-Propagator-f2},
we record the resulting form of $f_{2}^{A}(x,y)$.\end{proof}
\begin{rem}
{}{Determining the unknown functions $f_{g}$ is one
of the main difficulties in BCOV theory for $F_{g}$. For lower $g$,
we can observe that the vanishing conditions as in the above proposition
determine $f_{g}$ completely in some special cases, but they are
not sufficient in general. Regarding to this, there is another type
of vanishing conditions, called gap conditions, which arise from the
singular behavior of $F_{g}$ near the conifold loci of mirror families
(see \cite{Fg2222},\cite{AS} for details). We expect that we can
determine $F_{3}^{A}$ and $F_{3}^{B}$ if we combine these two vanishing
conditions.} $\hfill\square$
\end{rem}
\para{BPS numbers.} In Table \ref{tab:Table-A} (2) and Table \ref{tab:Table-B}
(2), we have listed the BPS numbers determined from $F_{2}^{A}$ and
$F_{2}^{B}$. Introducing $q$-series $Z_{2,n}^{A}(q)$ and $Z_{2,n}^{B}(q)$
as before for $g=0,1$, we verify Conjectures \ref{conj:Zgn-A}, \ref{conj:Zgn-B}
at $g=2$ for lower $n$ and sufficiently large degree in $q$. For
example, we obtain 
\[
\begin{alignedat}{2}P_{2,1}^{A} &  &  & =\frac{1}{2^{13}3^{3}}E{}_{2}^{4}+\frac{1}{2^{12}3^{3}}E{}_{2}^{3}(S+2T+4U)\\
 &  &  & +\frac{1}{2^{11}3^{3}}E{}_{2}^{2}(2S^{2}+37ST+8T^{2}-8SU-16TU)\\
 &  &  & +\frac{1}{2^{12}3^{3}}E_{2}(11S^{3}-66S^{2}T+96ST^{2}-32T^{3}-12S^{2}U+176STU)\\
 &  &  & +\frac{1}{2^{13}3^{2}5}(41S^{4}+756S^{3}T+15168S^{2}T^{2}+3904ST^{3}+256T^{4}-48S^{3}U\\
 &  &  & \hsp{125}-6624S^{2}TU-11648ST^{2}U-1024T^{3}U).
\end{alignedat}
\]
Interestingly, we can also verify the sharing property of $P_{g,1}$,
i.e., $P_{2,1}^{A}=P_{2,1}^{B}$ for $Z_{2,1}^{A}(q)$ and $Z_{2,1}^{B}(q)$
as observed for $g=1$ in Observation \ref{obs:Z11-B} {}{(see
\cite{HT-math-c}).} We conjecture that this sharing property holds
in general as follows: 
\begin{conjecture}
\label{conj:Conj.Zg1-AB}The $q$-series $Z_{g,1}^{A}(q)$ and $Z_{g,1}^{B}(q)$
are given by 
\[
Z_{g,1}^{A}(q)=P_{g,1}\,\frac{64}{\bar{\eta}(q)^{8}}\,\,\,\text{and \,\,}Z_{g,1}^{B}(q)=P_{g,1}\,\frac{1}{\bar{\eta}(q^{8})^{8}}
\]
with common polynomials $P_{g,1}(=P_{g,1}^{A}=P_{g,1}^{B})$ of degree
$2g$ of quasi-modular forms $E_{2},S,T$ and $U$. \end{conjecture}
\begin{rem}
The above conjecture reminds us the so-called S-duality between counting
sections of elliptic surfaces and counting Euler numbers of the corresponding
moduli spaces of sheaves on the surfaces \cite{VafaWitten}. S-duality
on elliptic surfaces is explained by fiberwise Fourier-Mukai transforms
\cite{VafaWitten,Yos}. In the present case, fiberwise Fourier-Mukai
transforms result in Fourier-Mukai partners, i.e., the geometry of
$V_{8,w}^{1}$ is transformed to $V_{8,w}^{1}/\mbZ_{8}\times\mbZ_{8}$
and vice versa. Note that, under fiberwise Fourier-Mukai transformations,
$n$-sections are transformed to sheaves of rank $n$. Then the simplification
we observe in the above conjecture {}{should be related
to} the fact that the relevant sheaves are of rank one. Corresponding
to the $q$-series $Z_{g,1}^{A}(q)$ (resp. $Z_{g,1}^{B}(q)$), there
should be some nice geometry of the moduli space of rank one sheaves
on $V_{8,w}^{1}/\mbZ_{8}\times\mbZ_{8}$ (resp. $V_{8,w}^{1}$). For
the counting of $n$-sections $Z_{g,n}^{A}$ and $Z_{g,n}^{B}$ (Conjecture
\ref{conj:Zgn-A} and Conjecture \ref{conj:Zgn-B}), the equality
$P_{g,n}^{A}=P_{g,n}^{B}$ does not hold anymore for $n\geq2$. However,
our results motivate us studying geometry of moduli spaces of rank
$n$ stable sheaves in general on both $V_{8,w}^{1}/\mbZ_{8}\times\mbZ_{8}$
and $V_{8,w}^{1}$. {}{We can find a study in this direction
in \cite{Bak}.} $\hfill\square$

\vskip3cm
\end{rem}

\section{\label{sec:MS-more-C}\textbf{Exploring the parameter space $\protect\mbP_{\Delta}$
more}}

We continue our study on the parameter space $\mbP_{\Delta}$ focusing
on the degeneration point $C$ in Fig.\ref{fig:Fig1}. Mirror symmetry
arises from this boundary point as in the preceding section. We will
identify this with the mirror symmetry between $V_{8,w}^{1}/\mbZ_{8}$
and its mirror family $\cV_{\mbZ_{8}\times\mbZ_{8}}^{1}\rightarrow\mbP_{\Delta}$.

\subsection{Blowing-up at the boundary point $C$ }

As shown in Fig.\ref{fig:Fig1}, the point $C$ is located at the
intersection point of three divisors $L_{1},L_{2}$ and $D_{s}$.
In terms of an affine coordinate $s_{1}:=4x+1$ and $y$, the relevant
components of three divisors are given by 
\[
\left\{ y=0\right\} ,\,\,\left\{ s_{1}=0\right\} \text{ and }\left\{ s_{1}+y=0\right\} ,
\]
respectively for $L_{1},L_{2}$ and $D_{s}$. We take $(s_{1},y)$
as an affine coordinate centered at $C$. As shown in Fig.\ref{fig:Fig-C-Cp},
after blowing-up at the origin, three intersection points of divisors
become normal crossing. To save our notation, we call two of the intersection
points $C$ and $C'$ (see Fig.\ref{fig:Fig-C-Cp}). 
\begin{prop}
The boundary points $C$ and $C'$, respectively, of the Picard-Fuchs
equations over $\mbP_{\Delta}$ are mirror symmetric to Calabi-Yau
manifolds{}{{} $Z:=V_{8,w}^{1}/\mbZ_{8}$ and $Z':=V_{8,w}^{2}/\mbZ_{8}$.}\end{prop}
\begin{proof}
We make the local solutions of Picard-Fuchs equations $\cD_{1}\omega=\cD_{2}\omega=0$
in terms of the blow-up coordinates, and arrange them into the canonical
form (\ref{eq:App-Canonial-Pi}). We observe the claimed mirror symmetry
by identifying the parameter $d_{ijk}$ in (\ref{eq:App-Canonial-Pi})
with the cubic forms of Calabi-Yau manifolds $V_{8,w}^{1}/\mbZ_{8}$
and $V_{8,w}^{2}/\mbZ_{8}$ (see \ref{para:potF-kappa-gen} and Remark
\ref{rem:Kappa-C-prob} below).
\end{proof}
\para{Griffiths-Yukawa couplings.} \label{para:Yukawa-C}Since the
local solutions near the boundary points $C$ and $C'$ are isomorphic,
as it is the case for $A$ and $A'$, we restrict our attentions to
$C$ using the blow-up coordinate
\begin{equation}
(s_{1},s_{2})=\big(\frac{1}{8}(4x+1),\frac{y}{4x+1}\big).\label{eq:affine-relation-s-x}
\end{equation}
Recall that the Griffiths-Yukawa couplings $C_{ijk}^{A}$ (\ref{eq:Cijk-A})
are determined by Picard-Fuchs equations, and related to $C_{ijk}^{B}$
by (\ref{eq:Cijk-B}). Since Picard-Fuchs equations are the same as
before, we have Griffiths-Yukawa couplings $C_{ijk}^{C}$ in the same
way by 
\[
C_{ijk}^{C}=\sum_{l,m,n}C_{lmn}^{A}\frac{\partial x_{l}}{\partial s_{i}}\frac{\partial x_{m}}{\partial s_{j}}\frac{\partial x_{n}}{\partial s_{k}}
\]
where $(x_{1},x_{2})=(x,y)$. 

\begin{figure}
~~~~~\includegraphics[scale=0.5]{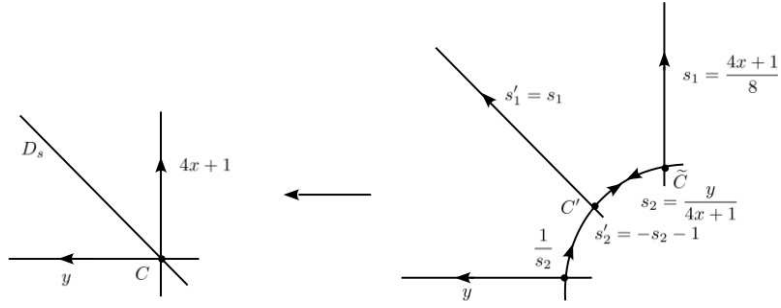}

\caption{\label{fig:Fig-C-Cp}Fig.4 Blowing-up at $C$. For brevity of notation,
$\tilde{C}$ is called $C$ again in the text. }

\end{figure}

\para{Mirror map and Gromov-Witten invariants.} Calculations of mirror
map and Gromov-Witten invariants are parallel to the previous cases.
In the present case the regular local solution $\omega_{0}^{C}(s)$
has the from 
\[
\omega_{0}^{C}(s)=1+8s_{1}+56s_{1}^{2}+384s_{1}^{3}+\cdots+s_{2}(-16s_{1}^{2}-256s_{1}^{3}-\cdots)+\cdots,
\]
and we arrange all solutions into the canonical form $\Pi_{C}(s)$
in (\ref{eq:App-Canonial-Pi}). Mirror map is described by 
\[
q_{1}=e^{\frac{\omega_{1}}{\omega_{0}}}=C_{1}s_{1}\,\exp\big(\frac{\omega_{1}^{reg}}{\omega_{0}}\big),\,\,\,\,q_{2}=e^{\frac{\omega_{2}}{\omega_{0}}}=C_{2}s_{2}\,\exp\big(\frac{\omega_{2}^{reg}}{\omega_{0}}\big),
\]
in terms of $\omega_{0}=\omega_{0}^{C}$ and $\omega_{i}=\omega_{i}^{C}$
and some constants $C_{k}$. Inverting these, we have $s_{a}=s_{a}(q_{1},q_{2})$
for $a=1,2$. Then the quantum corrected Yukawa couplings are given
by 
\begin{equation}
Y_{ijk}^{C}=\left(\frac{1}{N_{C}\omega_{0}^{C}}\right)^{2}\sum_{a,b,c}C_{abc}^{C}\frac{ds_{a}}{dt_{i}}\frac{ds_{b}}{dt_{j}}\frac{ds_{c}}{dt_{k}}.\label{eq:Yijk-C}
\end{equation}

\begin{prop}
\label{prop:q-Yijk-C}When we set $N_{C}=\frac{1}{\sqrt{2}}$, and
$C_{1}=\sqrt{-1},C_{2}=1$ in the definition of the mirror map, we
have {\small{}
\[
\begin{alignedat}{2} & Y_{111}^{C}=16+512q_{1}^{2}q_{2}+22528q_{1}^{4}q_{2}+\cdots, &  & Y_{112}^{C}=8+256q_{1}^{2}q_{2}+5632q_{1}^{4}q_{2}+\cdots,\\
 & Y_{122}^{C}=128q_{1}^{2}q_{2}+1408q_{1}^{4}q_{2}+9216q_{1}^{6}q_{2}^{2}+\cdots, & \, & Y_{222}^{C}=8q_{2}+8q_{2}^{2}+8q_{2}^{3}+\cdots+64q_{1}^{2}q_{2}+.
\end{alignedat}
\]
}We have exactly same form for the corresponding expansions of $Y_{ijk}^{C'}$.
\end{prop}
We observe in the above $q$-series expansions of $Y_{ijk}^{C}$ that
non-vanishing coefficients appear only for even-powers of $q_{1}$. 

\begin{table}
{\scriptsize{}
\[
\begin{array}{|c|ccccccccc}
\hline \m j\diagdown i\m & 0 & 1 & 2 & 3 & 4 & 5 & 6 & 7 & 8\\
\hline 0 & 0 & 0 & 0 & 0 & 0 & 0 & 0 & 0 & 0\\
1 & 8 & 64 & 352 & 1536 & 5808 & 19712 & 61632 & 180224 & 498696\\
2 & 0 & 0 & 512 & 12288 & 152576 & 1351680 & 9596928 & 58040320 & 310097920\\
3 & 0 & 0 & 352 & 24384 & 787632 & 15603712 & 221577504 & 2470588416 & 22896066536\\
4 & 0 & 0 & 0 & 12288 & 1320960 & 59842560 & 1654833152 & 32671129600 & 502295842816\\
5 & 0 & 0 & 0 & 1536 & 787632 & 91112640 & 5099691008 & 179687448576 & 4551683595656\\
6 & 0 & 0 & 0 & 0 & 152576 & 59842560 & 7345469952 & 476575268864 & 20119442417664\\
7 & 0 & 0 & 0 & 0 & 5808 & 15603712 & 5099691008 & 654619035200 & 47548652073448\\
: & : & \m: & \m: & \m: & \m: & \m: & \m: & \m: & :
\end{array}
\]
}{\scriptsize \par}

(0) BPS numbers $n_{0}^{C}(i,j)$ 

{\scriptsize{}
\[
\begin{array}{|c|cccccccc}
\hline \m j\diagdown i\m & 0 & 1 & 2 & 3 & 4 & 5 & 6 & 7\\
\hline 0 & 0 & 8 & 8 & 8 & 8 & 8 & 8 & 8\\
1 & 0 & 96 & 1824 & 16384 & 104480 & 533120 & 2332608 & 9074688\\
2 & 0 & 8 & 5072 & 216480 & 4207776 & 53459792 & 513205888 & 4022629824\\
3 & 0 & 0 & 1824 & 462176 & 27062912 & 798811136 & 15453590048 & 222719072256\\
4 & 0 & 0 & 8 & 216480 & 47771216 & 3439792064 & 134032422304 & 3481149295168\\
5 & 0 & 0 & 0 & 16384 & 27062912 & 5471284128 & 449545296384 & 21314156435456\\
6 & 0 & 0 & 0 & 8 & 4207776 & 3439792064 & 664078158608 & 60005323039296\\
7 & 0 & 0 & 0 & 0 & 104480 & 798811136 & 449545296384 & 84043909825632\\
: & : & : & : & : & : & : & : & :
\end{array}
\]
}{\scriptsize \par}

(1) BPS numbers $n_{1}^{C}(i,j)$ 

~

\caption{Table 3. \label{tab:Table-C} BPS numbers $n_{g}^{C}(i,j)=n_{g}^{C}(\beta.H_{Z},\beta.A_{Z})$
of $Z=V_{8,w}^{1}/\protect\mbZ_{8}$. These are equal to the corresponding
$n_{g}^{C'}(i,j)$ of $Z'=V_{8,w}^{2}/\protect\mbZ_{8}$.}
\end{table}

\subsection{Genus one potential $F_{1}^{C}$}

We calculate genus one Gromov-Witten invariants by using BCOV formula
(\ref{eq:F1-X-formula}). 

\para{BCOV potential function $F_1^C$.} Let us introduce the following
definitions:
\[
\begin{alignedat}{1}disC_{0}=(1-4s_{1})^{4}+256s_{1}^{2}(1-8s_{1})s_{2}(1+s_{2}),\quad\quad\;\;\\
disC_{1}=1+s_{2},\,\,\,disC_{2}=1-4s_{1},\,\,\,disC_{3}=1-8s_{1}.
\end{alignedat}
\]
As in Subsection \ref{sub:F1-A-A'}, knowing that the boundary point
$C$ is mirror symmetric to a free quotient $V_{8,w}^{1}/\mbZ_{8}$,
we can determine the parameters $r_{0}$ and $r_{k}$ contained in
the BCOV formula. 
\begin{prop}
Around the boundary point $C$, the BCOV potential function has the
following form
\[
F_{1}^{C}(q_{1},q_{2})=\frac{1}{2}\log\Big\{\big(\frac{1}{\omega_{0}^{C}}\big)^{5}\frac{\partial(s_{1},s_{2})}{\partial(t_{1},t_{2})}disC_{0}^{-\frac{8}{6}}\,disC_{1}^{-1}\,disC_{2}^{-1}\,disC_{3}^{-\frac{5}{3}}s_{1}^{-1-\frac{16}{12}}s_{2}^{-1}\Big\}.
\]

\end{prop}
Assuming mirror symmetry, we read the genus one Gromov-Witten invariants
from these potential functions. In Table \ref{tab:Table-C} (1), we
have listed the resulting BPS numbers (see also \ref{para:potF-kappa-gen}
below). Also, from the from of $F_{1}^{C}$, it is easy to verify
the following property:
\begin{prop}
\label{prop:F1-C-A-NOT-rel}The potential functions $F_{1}^{C}$ and
$F_{1}^{A}$ are NOT related to each other on $\mbP_{\Delta}$ by
the transformation rules $($t1$)$ and $($t2$)$ in $($\ref{eq:connect-g1}$)$.
\end{prop}
We will interpret the above proposition in Subsection \ref{sub:MS-ABC}
(Proposition \ref{prop:MS-Result-Main}).
\begin{rem}
\label{rem:Conifold-Factor-C}Recall that the conifold factor of $F_{1}^{M}$
in Remark \ref{rem:Conifold-A} counts the number of ODPs or vanishing
$S^{3}$'s which appear over the principal component of the discriminant
of a family. In many examples, it is observed that the conifold factor
is generalized to vanishing lens space $S^{3}/\mbZ_{N}$ by counting
$-\frac{N}{6}$ for each lens space $S^{3}/\mbZ_{N}$ (eg. \cite{AM1},\cite{HT-JAG}).
Based on this, we read the BCOV formula $F_{1}^{C}$ as the one applied
for the family $\cV_{\mbZ_{8}\times\mbZ_{8}}^{1}\rightarrow\mbP_{\Delta}$.
In fact, we have one vanishing cycle $S^{3}$ in each fiber $V_{8,w}^{1}/\mbZ_{8}$
of the family $\cV_{\mbZ_{8}}^{1}$ over $\left\{ disC_{0}=0\right\} =\left\{ dis_{0}=0\right\} $,
i.e., the $\tau$-orbit in Proposition \ref{prop:ODP-8}. This vanishing
cycle gives rise to a lens space $S^{3}/\langle\sigma\rangle$ in
the full quotient $\cV_{\mbZ_{8}\times\mbZ_{8}}^{1}\rightarrow\mbP_{\Delta}$
by $\mbZ_{8}\times\mbZ_{8}=\langle\sigma,\tau\rangle.$ This explains
the conifold factor in $disC_{0}^{-\frac{8}{6}}$ and also the claimed
property in Proposition \ref{prop:F1-C-A-NOT-rel}, because $F_{1}^{C}$
and $F_{1}^{A}$ are defined for different families $\cV_{\mbZ_{8}\times\mbZ_{8}}^{1}$
and $\cV_{\mbZ_{8}}^{1}$, respectively. $\hfill\square$ 
\end{rem}
\para{Generators of $H_2(V^1_{8,w}/\mbZ_8)$.} \label{para:PicG-Z8}
Recall that we fixed a subgroup $\mbZ_{8}=\langle\tau\rangle$ to
define a quotient $V_{8,w}^{1}/\mathbb{Z}_{8}$. Since $V_{8,w}^{1}$
is simply connected, the fundamental group of the free quotient $V_{8,w}^{1}/\mbZ_{8}$
is isomorphic to $\mbZ_{8}$. In the same way as in \ref{para:PicG-Z8Z8},
{}{we can argue} the generators of $H_{2}(V_{8,w}^{1}/\mbZ_{8},\mbZ)$
from the exact sequence 
\[
1\rightarrow\pi_{2}(V_{8,w}^{1}/\mbZ_{8})\rightarrow H_{2}(V_{8w}^{1}/\mbZ_{8},\mbZ)\rightarrow H_{2}(\mbZ_{8})\rightarrow0,
\]
where{}{{} $H_{2}(\mbZ_{8})=H_{2}(\pi_{1}(V_{8,w}^{1}/\mbZ_{8}))$
is the group homology. Since $H_{2}(\mbZ_{8})=0$ this time, we see
the isomorphism $\pi_{2}(V_{8,w}^{1}/\mbZ_{8})\simeq H_{2}(V_{8w}^{1}/\mbZ_{8},\mbZ)$,
which indicates that $H_{2}$ is generated by rational curves. In
fact, we have generators as follows: Let us note that both $V_{8,w}^{1}$
and $V_{8,w}^{1}/\mbZ_{8}$ have fibrations by abelian surfaces. We
take divisors $H_{X}$ and $A_{X}$ to be the restriction of the hyperplane
class and the fiber class of $V_{8,w}^{1}(=X)$, respectively.} Denote
the free quotient by {}{$\pi:V_{8,w}^{1}\rightarrow V_{8,w}^{1}/\mbZ_{8}=Z$.}
Since the free $\mbZ_{8}$-action acts on each fiber, we have $\pi^{*}A_{Z}=A_{X}$
for the pull-back of the fiber class $A_{Z}$ of $V_{8,w}^{1}/\mbZ_{8}$.
The group $\mbZ_{8}$ acts diagonally on the coordinates $x_{i}$.
Hence, restricting the divisor $\left\{ x_{i}=0\right\} $ to $V_{8,w}^{1}$,
we have a $\mbZ_{8}$-invariant divisor of the class $H_{X}$, which
is the pull-back of a divisor $H_{Z}$ on $Z$. To summarize, we have
\begin{equation}
\pi^{*}H_{Z}=H_{X},\,\,\,\,\,\pi^{*}A_{Z}=A_{X}.\label{eq:pull-back-C}
\end{equation}
Now, for these two divisors $H_{Y}$ and $A_{Y}$ on $V_{8,w}^{1}/\mbZ_{8}$,
we calculate 
\[
\begin{aligned}1=H_{X}.\ell=H_{Z}.(\pi_{*}\ell),\,\,\,\,\,\, & 0=H_{X}\cdot\sigma_{X}=H_{Z}.(\pi_{*}\sigma_{X}),\\
0=A_{X}.\ell=A_{Z}.(\pi_{*}\ell),\,\,\,\,\,\, & 1=A_{X}.\sigma_{X}=A_{Z}.(\pi_{*}\sigma_{X}),
\end{aligned}
\]
where $\sigma_{X}$ is a section of abelian surface fibration $V_{8,w}^{1}\rightarrow\mbP^{1}$
and $\ell$ is a line contained in a singular fiber of $V_{8,w}^{1}$
(see \ref{para: V1-summary}). These relations show that, modulo torsion
elements, the classes of $\pi_{*}\ell$ and $\pi_{*}\sigma_{X}$ generate
$H_{2}(V_{8,w}^{1}/\mbZ_{8},\mbZ)\simeq\pi_{2}(V_{8,w}^{1}/\mbZ_{8})$
and also $H_{Z}$ and $A_{Z}$ generate $Pic(V_{8,w}^{1}/\mbZ_{8})$. 

Note that, from (\ref{eq:pull-back-C}), we can determine the following
invariants of $Z{\color{black}}=V_{8,w}^{1}/\mbZ_{8}$,
\begin{equation}
H_{Z}^{3}=H_{Z}^{2}A_{Z}=2,\,\,\,\,H_{Z}A_{Z}^{2}=A_{Z}^{3}=0;\,\,\,\,\,\,c_{2}.H_{Z}=8,\,\,c_{2}.A_{Z}=0.\label{eq:Kcl-c2j-C}
\end{equation}

\para{Gromow-Witten invariants from potential functions $F^C_0,F^C_1$.}
\label{para:potF-kappa-gen}Recall that Gromov-Witten potential $F_{0}^{C}(t_{1},t_{2})$
at genus $0$ is defined as a function which satisfies 
\[
Y_{ijk}^{C}=\frac{\partial^{3}\;}{\partial t_{i}\partial t_{j}\partial t_{k}}F^{C}(t_{1},t_{2})
\]
for the quantum corrected Yukawa couplings in Proposition \ref{prop:q-Yijk-C}
with $q_{i}=e^{t_{i}}$. As a generating function of Gromov-Witten
invariants, this function takes the following general form:
\[
F_{0}(t_{1},t_{2})=\frac{1}{3!}\int_{Z}\kappa_{t}^{3}+\sum_{{\beta\in H_{2}(Y,\mbZ)\atop \beta\not=0}}N_{0}(\beta)e^{\beta.\kappa_{t}}
\]
with $\kappa_{t}:=t_{1}H_{1}+t_{2}H_{2}$ in terms of some nef-divisors
$H_{i}$. Comparing the invariants in (\ref{eq:Kcl-c2j-C}) with the
constant terms of $Y_{ijk}^{C}$, we see that, for the potential function
$F_{0}^{C}(t_{1},t_{2})$ from $Y_{ijk}^{C}$ in Proposition \ref{prop:q-Yijk-C},
we should have 
\begin{equation}
\kappa_{t}=t_{1}(2H_{Z})+t_{2}A_{Z}.\label{eq:kappa-C-2HA}
\end{equation}
Once we know the above identification of divisors, the genus one Gromov-Witten
invariants $N_{1}(\beta)$ are read by identifying the BCOV formula
for $F_{1}^{C}$, up to a constant term, with the following general
form: 

\[
F_{1}(t)=-\frac{c_{2}.\kappa_{t}}{24}+\sum_{{\beta\in H_{2}(Y,\mbZ)\atop \beta\not=0}}N_{1}(\beta)e^{\beta.\kappa_{t}}.
\]
The BPS numbers $n_{g}(i,j)=n_{g}(\beta.H_{Z},\beta.A_{Z})$ listed
in Table \ref{tab:Table-C} are determined from $N_{g}(\beta)=N_{g}(\beta.H_{Z},\beta.A_{Z})$
which are read from $F_{0}^{C}(t)$ and $F_{1}^{C}(t)$. 
\begin{rem}
\label{rem:Kappa-C-prob}When reading Gromov-Witten invariants from
(\ref{eq:Yijk-A}) and (\ref{eq:Yijk-B}), we implicitly identified
the $\kappa_{t}$ in the potential function with $\kappa_{t}=t_{1}H_{X}+t_{2}A_{X}$
and $\kappa_{t}=t_{1}H_{Y}+t_{2}A_{Y}$, respectively. We note that
$H_{X},A_{X}$ and $H_{Y},A_{Y}$ in these equations are generators
of the Picard groups, while $2H_{Z},A_{Z}$ in (\ref{eq:kappa-C-2HA})
are not. The above example shows that a non-trivial identification
of $\kappa_{t}$ like (\ref{eq:kappa-C-2HA}) introduces a slight
subtlety when reading $N_{g}(\beta)$ from the potential functions
$F_{g}(t)$ which we calculate by using mirror symmetry. $\hfill\square$
\end{rem}
\para{$Z_{g,n}(q)$ by modular forms.} We define the $q$-series $Z_{g,n}^{C}(q)$
from the potential functions $F_{g}^{C}(q,p)$ as in the preceding
sections. From Table \ref{tab:Table-C}, it is easy to observe that
\begin{equation}
Z_{0,1}^{C}(q)=\frac{8}{\bar{\eta}(q^{2})^{8}},\label{eq:Zg0C-A-relation}
\end{equation}
where $q^{2}$ comes from the relation (\ref{eq:kappa-C-2HA}). Also,
corresponding to (\ref{eq:Zg0B-A-relation}), we find the following
relations: 
\[
P_{0,n}^{C}(q)=8^{n-1}P_{0,n}^{A}(q^{2})\,\,\,\,(n\geq1),
\]
for the polynomials defined by $Z_{g,n}^{C}(q)=P_{g,n}^{C}(q)\left(\frac{8}{\bar{\eta(q^{2})^{8}}}\right)^{n}$.
Furthermore, at $g=1$, we can calculate explicitly the following
forms of $P_{1,n}^{C}(q)$ in terms of elliptic quasi-modular forms
$\tilde{E}_{k}:=E_{k}(q^{2})\,(k=2,4,6)$ for lower $n$ {}{(see
also \cite{HT-math-c})}: 
\[
\begin{alignedat}{1}P_{1,1}^{C}= & \frac{1}{36}(\tilde{E}_{2}^{2}+2\tilde{E}_{4}),\\
P_{1,2}^{C}= & \frac{1}{10368}(8\tilde{E}_{2}^{4}+13\tilde{E}_{2}^{2}\tilde{E}_{4}+17\tilde{E}_{2}\tilde{E}_{6}+16\tilde{E}_{4}^{2}),\\
P_{1,3}^{C}= & \frac{1}{2^{14}3^{7}}(1048\tilde{E}_{2}^{6}+1995\tilde{E}_{2}^{4}\tilde{E}_{4}+3642\tilde{E}_{2}^{2}\tilde{E}_{4}^{2}+1711\tilde{E}_{4}^{3}+2444\tilde{E}_{2}^{3}\tilde{E}_{6}\\
 & \hsp{170}+3948\tilde{E}_{2}\tilde{E}_{4}\tilde{E}_{6}+764\tilde{E}_{6}^{2}).
\end{alignedat}
\]

\para{Contraction to $(2,2,2,2)/\mbZ_8$.} Birational geometry of
$V_{8,w}^{1}/\mbZ_{8}$ is quite parallel to the cases of $V_{8,w}^{1}$
and $V_{8,w}^{1}/\mbZ_{8}\times\mbZ_{8}$; the diagram (\ref{eq:Birat-i})
is valid with the free $\mbZ_{8}$ actions;\def\contRiii{\begin{xy}
(-15,0)*++{V^1_{8,w}/\mbZ_8}="Vi",
( 15,0)*++{V^2_{8,w}/\mbZ_8}="Vii",
(  0,-13)*++{X_{2,2,2,2}^{sing}/\mbZ_8}="Zi",
(-11,-7)*++{\,^{\,_{8\,\mathbb{P}^1}}},
( 11,-7)*++{\,^{\,_{8\,\mathbb{P}^1}}},
\ar@{<-->}^{} "Vi";"Vii"
\ar "Vi";"Zi"
\ar "Vii";"Zi"
\end{xy} }
\begin{equation}
\begin{matrix}\contRiii\end{matrix}\label{eq:Birat-iii}
\end{equation}
The birational model $V_{8,w}^{2}/\mbZ_{8}$ corresponds to the other
boundary point $C'$ in Fig.\ref{fig:Fig-C-Cp}. We note that the
singular Calabi-Yau complete intersection $X_{2,2,2,2}^{sing}/\mbZ_{8}$
admits a smoothing to a smooth free quotient $X_{2,2,2,2}/\mbZ_{8}$
in \cite{Brown,Hua}. We can observe the following sum-up properties
for $g=0$ and $g=1$ BPS numbers:
\begin{equation}
\sum_{d_{2}}n_{0}^{C}(d,d_{2})=\frac{1}{|\mbZ_{8}|}n_{0}^{(2,2,2,2)}(d),\,\,\,\,\sum_{d_{2}}n_{1}^{C}(d,d_{2})=n_{1}^{(2,2,2,2)/\mbZ_{8}}(d).\label{eq:sum-up-C}
\end{equation}
The number $\frac{1}{|\mbZ_{8}|}n_{0}^{(2,2,2,2)}(d)$ can be interpreted
as BPS numbers of the smooth free quotient $X_{2,2,2,2}/\mbZ_{8}$.
The number $n_{1}^{(2,2,2,2)/\mbZ_{8}}(d)$ has been determined in
Appendix \ref{sub:Appendix-free-quot-Z8} from BCOV formula assuming
a mirror family of the smooth free quotient $X_{2,2,2,2}/\mbZ_{8}$.

\subsection{Genus two potential $F_{2}^{C}$ and $F_{g}^{C}\,(g\geq2)$. }

It is quite parallel to the cases of $F_{2}^{A}$ and $F_{2}^{B}$
to determine the genus two potential function $F_{2}^{C}$ up to a
rational function $f_{2}^{C}(s_{1},s_{2})$. It is easy to find a
suitable ansatz on $f_{2}^{C}$ and impose the most reasonable vanishing
conditions on some BPS numbers; however, there still remain undetermined
parameters. If $F_{2}^{C}$ were related to $F_{2}^{A}$ or $F_{2}^{B}$
on $\mbP_{\Delta}$ by the transformation rules (t1)$'$ and (t2)$'$,
then $f_{2}^{C}$ would be determined from $f_{2}^{A}$. However,
we find that vanishing conditions for $A$ and $C$ are not compatible,
i.e., we encounter non-integral BPS numbers if we impose them. At
genus one, we have already encountered the corresponding situation
in Proposition \ref{prop:F1-C-A-NOT-rel}. 

Our verifications of the following conjecture are restricted to genus
zero and one; but we expect that it holds in general because BCOV
recursion formulas for $F_{g}(g\geq2)$ start with $F_{0}$ and $F_{1}$
as initial data. 
\begin{conjecture}
\label{conj:Zgn-C}The $q$-series $Z_{g,n}^{C}(q)$ of Gromov-Witten
invariants of $V_{8,w}^{1}/\mbZ_{8}$ are written by quasi-modular
modular forms as
\[
Z_{g,n}^{C}(q)=P_{g,n}^{C}(E_{2}(q^{2}),E_{4}(q^{2}),E_{6}(q^{2}))\left(\frac{8}{\bar{\eta}(q^{2})^{8}}\right)^{n},
\]
where $P_{g,n}^{C}(E_{2}(q^{2}),E_{4}(q^{2}),E_{6}(q^{2}))$ are quasi-modular
forms of weight $4(g+n-1)$. \end{conjecture}
\begin{rem}
In physics, the whole set $\left\{ F_{g}^{M}\right\} $ of potential
functions define the so-called topological string theory on a Calabi-Yau
manifold $M$. From the above (conjectural) simple structure on $P_{g,n}^{C}$,
we may naturally expect that the topological string on $Z=V_{8,w}^{1}/\mbZ_{8}$
is completely integrable, i.e., the whole set $\left\{ F_{g}^{M}\right\} $
may be determined completely. Mathematically, the simplification in
the quasi-modular property of $Z_{g,n}^{C}(q)$ may be explained by
the fact that the fiber abelian surfaces are principally polarized
for the fibration $V_{8,w}^{1}/\mbZ_{8}\rightarrow\mbP^{1}$ as we
see in (\ref{eq:Kcl-c2j-C}). $\hfill\square$
\end{rem}

\subsection{Mirror symmetry\label{sub:MS-ABC}}

In Section \ref{sec:MSbyPF-A}, with a subgroup $\mbZ_{8}=\langle\tau\rangle\subset\mbZ_{8}\times\mbZ_{8}$,
we have introduced two families 
\[
\frak{X}=\cV_{\mbZ_{8}}^{1}\rightarrow\mbP_{\Delta}\,\,\text{and }\,\,\frak{X}=\cV_{\mbZ_{8}\times\mbZ_{8}}^{1}\rightarrow\mbP_{\Delta}
\]
over the same parameter space $\mbP_{\Delta}$. The local systems
$R^{3}\pi_{*}\mbC_{\frak{X}}$ associated to these are represented
by the same Picard-Fuchs differential equations $\cD_{2}\omega=\cD_{3}\omega=0$
(\ref{eq:PFeqs-D2-D3}) on $\mbP_{\Delta}$. We can now summarize
our results from each boundary point of $A,B;A',B';C,C'$ as follows. 
\begin{prop}
\label{prop:MS-Result-Main}The following two different pictures of
mirror symmetry are encoded in the same Picard-Fuchs differential
equations $($\ref{eq:PFeqs-D2-D3}$)$ : 

\begin{myitem}

\item{$(1)$} When we read $($\ref{eq:PFeqs-D2-D3}$)$ as representing
the local system of the family $\cV_{\mbZ_{8}}^{1}\rightarrow\mbP_{\Delta}$,
mirror symmetry of the family to Calabi-Yau manifolds $V_{8,w}^{1}$
and $V_{8,w}^{1}/\mbZ_{8}\times\mbZ_{8}$ are identified at the boundary
points $A$ and B, respectively. Mirror symmetry to birational models
$V_{8,w}^{2}$ and $V_{8,w}^{2}/\mbZ_{8}\times\mbZ_{8}$ are also
identified at $A'$ and $B'$. 

\item{$(2)$} When we read $($\ref{eq:PFeqs-D2-D3}$)$ as representing
the local system of the family $\cV_{\mbZ_{8}\times\mbZ_{8}}^{1}\rightarrow\mbP_{\Delta}$,
mirror symmetry of the family to a Calabi-Yau manifold $V_{8,w}^{1}/\mbZ_{8}$
is identified at the boundary point C. Mirror symmetry to the birational
model $V_{8,w}^{2}/\mbZ_{8}$ is identified at $C'$.

\end{myitem}
\end{prop}
At this moment, our identifications of the mirror families $\cV_{\mbZ_{8}}^{1}$
and $\cV_{\mbZ_{8}\times\mbZ_{8}}^{1}$ as above are based on the
genus one potential functions $F_{1}^{A},F_{1}^{B}$ and $F_{1}^{C}$
(Remarks \ref{rem:Conifold-A} and \ref{rem:Conifold-Factor-C}).
In the next section (Proposition \ref{prop:Int-str-AB-C}), the difference
between the two families will be explained further by integral structures
from the solutions of (\ref{eq:PFeqs-D2-D3}). 

The above proposition is our affirmative answer to Conjecture \ref{conj:GrossPav}. 
\begin{rem}
Calabi-Yau manifolds $V_{8,w}^{1}$ and $V_{8,w}^{1}/\mbZ_{8}\times\mbZ_{8}$
are Fourier-Mukai partners to each other. The above proposition shows
that they are mirror symmetric to the family $\cV_{\mbZ_{8}}^{1}\rightarrow\mbP_{\Delta}$.
Conversely, Calabi-Yau manifold $V_{8,w}^{1}/\mbZ_{8}$ should be
mirror symmetric to both the family $\cV_{\mbZ_{8}\times\mbZ_{8}}^{1}\rightarrow\mbP_{\Delta}$
and a family of $V_{8,w}^{1}$ over some parameter space, although
we haven't constructed the latter.$\hfill\square$
\end{rem}
\newpage

\section{\label{sec:Degens-CY}\textbf{Degenerations and analytic continuations }}

We further study our results in Proposition \ref{prop:MS-Result-Main}
by calculating the connection matrices for the local solutions at
each boundary point $A,B$ and $C$. We also describe the degenerations
of Calabi-Yau manifolds over these points. We expect that categorical
and geometric aspects of mirror symmetry, including Fourier-Mukai
partners, will appear in explicit and concrete forms from these degenerations,
but we leave the details for future investigations.

\subsection{Analytic continuations of period integrals}

We consider the connection problem of local solutions around the boundary
points $A,B$ and $C$ of the Picard-Fuchs differential equations
(\ref{eq:PFeqs-D2-D3}). To set up the connection problem, we arrange
local solutions into the canonical form (\ref{eq:App-Canonial-Pi})
in Appendix \ref{sec:App-Canonical-Form}; and write them by 
\[
\Pi_{A}(x,y),\,\,\,\Pi_{B}(z_{1},z_{2}),\,\,\,\Pi_{C}(s_{1},s_{2}),
\]
where we use the affine coordinates centered at each boundary point
which are related to each other by (\ref{eq:affine-relation-x-zB})
and (\ref{eq:affine-relation-s-x}). Here we set the parameters $a_{ij}$
to zero in the canonical forms $\Pi_{P}$ for all $P=A,B,C$ (and
$\tilde{A}$). To simplify our calculations, in this paper, we will
restrict our calculations for the above three points (and $\tilde{A}$
additionally). 

\para{Making local solutions.} To do analytic continuations, we need
to generate local solutions in the forms of power series up to sufficiently
high degrees. In the present case, the following property of the Picard-Fuchs
system enables us to do it efficiently. Below we sketch our calculations
for the case $A$, but calculations for other cases $B,C$ and also
$\tilde{A}$ are quite parallel. 

Let us recall the forms of local solutions $\omega_{0}(x),\omega_{i}(x),\omega_{2,i}(x),\omega_{3}(x)$
in (\ref{eq:Appendix-omega-k}). Arrange the regular solution $\omega_{0}(x)$
by powers of $y$,
\[
\omega_{0}(x,y)=f_{0}(x)+f_{1}(x)y+f_{2}(x)y^{2}+f_{3}(x)y^{3}\cdots,
\]
and substitute this into $\cD_{2}\omega_{0}=0$, then we obtain polynomial
relations 
\[
f_{n}+G_{n}(x,f_{n-1},f_{n-1}',f_{n-1}'')=0\,\,\,\,(n=1,2,\cdots),
\]
which we can solve recursively with an initial data $f_{0}(x)$. From
$\cD_{2}\omega_{0}=\cD_{3}\omega_{0}=0,$ it is easy to find the initial
data $f_{0}(x)=\sum_{n\geq0}\frac{(2n)!(2n)!}{(n!)^{4}}x^{2n}$. Basically,
this method works for other local solutions assuming their forms in
(\ref{eq:Appendix-omega-k}). For the case of $\omega_{1}$, for example,
having $\omega_{0}(x)$ up to desired orders in $x$ and $y$, arrange
$\omega_{1}(x)$ as 
\[
\omega_{1}(x,y)=\omega_{0}(x,y)\log x+g_{0}(x)+g_{1}(x)y+g_{2}(x)y^{2}+\cdots,
\]
and substitute this into $\cD_{2}\omega_{1}=0$. Then we obtain polynomial
relations 
\[
g_{n}+K_{n}(x,g_{n-1},g_{n-1}',g_{n-1}'')=0\,\,\,\,(n=1,2,\cdots),
\]
which we can solve recursively once we determine $g_{0}(x)$. To determine
$g_{0}(x)$, we use the equations $\cD_{2}\omega_{1}=\cD_{3}\omega_{1}=0$
to find a linear differential equation of $g_{0}(x)$ described by
the data $f_{0}(x)$ and $f_{1}(x)$. 

Though the solutions $\omega_{2,i}(x),\omega_{3}(x)$ contain polynomials
of higher powers of $\log x$ and $\log y$, we can continue the same
process after having solutions $\omega_{0}(x)$ and $\omega_{i}(x)$. 

\para{Connection matrices.} \label{para:AnalyticC}Analytic continuation
of local solutions is tedious in general for differential equations
of multi-variables. However, note that the boundary points $A,B,C$
and $A'$ are aligned on the real line of a single rational boundary
divisor in $\mbP_{\Delta}$ (see Fig.\ref{fig:Analytic-Path}). This
reduces our connection problem essentially to that of one variable.
In fact, we can obtain the following results by making local solutions
up to the first order in $y$ but sufficiently higher order in $x$. 

To describe the results, let us write by $\Pi_{A}(x),\Pi_{B}(z),\Pi_{\tilde{A}}(\tilde{x})$
and $\Pi_{C}(s)$ the local solutions in the canonical form (\ref{eq:App-Canonial-Pi})
with $a_{ij}=0$ for $A,B,\tilde{A}$ and $C$, respectively. Four
points $A,B,\tilde{A}$ and $C$ are aligned on the real coordinate
line of a boundary divisor $\left\{ y=0\right\} $ with their coordinates
$(x,0)=(0,0),(\frac{1}{4},0),(\infty,0)$ and $(-\frac{1}{4},0)$
in order. We define connection matrices along the real coordinate
line (see Fig.\ref{fig:Analytic-Path}) by 
\[
\Pi_{A}=U_{AB}\Pi_{B},\,\,\,\,\Pi_{B}=U_{B\tilde{A}}\Pi_{\tilde{A}},\,\,\,\,\,\Pi_{\tilde{A}}=U_{\tilde{A}C}\Pi_{C},\,\,\,\,\,\Pi_{C}=U_{CA}\Pi_{A},
\]
where $U_{PQ}$ represents the connection matrix of the analytic continuation
of $\Pi_{P}$ to the point $Q$ along a path $P\rightarrow Q$ shown
in Fig.\ref{fig:Analytic-Path}. Note that the canonical forms $\Pi_{P}$
of local solutions contain normalization constants $N_{P}$. 
\begin{prop}
\label{prop:Conection-Matrices}When we normalize the local solutions
by 
\[
N_{A}=1,\,\,\,\,N_{B}=\frac{1}{2},\,\,\,\,N_{C}=\frac{1}{\sqrt{2}},\,\,\,\,N_{\tilde{A}}=2^{4},
\]
then the connection matrices are represented by symplectic matrices
with respect to $\Sigma$ in (\ref{eq:App-Sigma}); explicitly, they
are given by 
\[
\begin{aligned} & U_{AB}=\left(\begin{matrix}\,\,0 & 0 & 0 & \m\m-1 & 0 & 0\\
\,\,0 & 1 & 0 & 0 & 0 & 0\\
\,\,0 & 0 & 0 & 0 & 0 & 1\\
\m\m\,-1 & 0 & 0 & 0 & 0 & 0\\
\,\,0 & 0 & 0 & 0 & 1 & 0\\
\,\,0 & 0 & 1 & 0 & 0 & 0
\end{matrix}\right),\,\,\,\,U_{\tilde{A}C}=\frac{1}{\sqrt{2}}\left(\begin{matrix}-2 & 8 & 0 & \m\m-2 & 0 & 0\\
\,\,\frac{1}{2} & \m\m-3 & 0 & 1 & 0 & 0\\
\,\,\frac{3}{2} & \m\m-3 & \m\m-2 & 0 & 1 & 2\\
-1 & 8 & 0 & \m\m-4 & 0 & 0\\
-3 & 0 & 8 & 0 & \m\m-6 & \m\m-16\\
-\frac{5}{2} & 10 & 1 & \m\m-4 & \m\m-1 & \m\m-4
\end{matrix}\right),\\
 & U_{B\tilde{A}}=\left(\begin{matrix}\,\,0 & 0 & 0 & \m\m-1 & 0 & 0\\
\,\,0 & 1 & 0 & \frac{1}{2} & 0 & 0\\
\,\,0 & 0 & 0 & 2 & 0 & 1\\
\m\m\,-1 & \m\m-8 & 0 & \m\m-2 & 0 & 0\\
\,\,0 & \m\m-96 & 0 & \m\m-32 & 1 & \m\m-8\\
\m\m\,-2 & \m\m-32 & 1 & \m\m-7 & \frac{1}{2} & \m\m-2
\end{matrix}\right),U_{CA}=\frac{1}{\sqrt{2}}\left(\begin{matrix}\m\m\,-4 & 16 & 0 & \m\m-2 & 0 & 0\\
\m\m\,-1 & 2 & 0 & 0 & 0 & 0\\
\,\,4 & 0 & \m\m-4 & 0 & 1 & 2\\
\m\m\,-1 & 0 & 0 & 0 & 0 & 0\\
\,\,6 & \m\m-16 & \m\m-8 & 1 & 1 & 0\\
\,\,\frac{1}{2} & \m\m-1 & 1 & 0 & 0 & 0
\end{matrix}\right).
\end{aligned}
\]
\end{prop}
\begin{rem}
(1) The values of the normalization constants $N_{P}$ are exactly
the same as those used in (\ref{eq:Yijk-A}),(\ref{eq:Yijk-B}) and
(\ref{eq:Yijk-C}) to have right quantum corrected Yukawa couplings
with right normalizations. (2) An ordered product of the above connection
matrices represents a trivial loop on the divisor $\left\{ y=0\right\} \simeq\mbP^{1}$.
We observe that our choice of path for this loop is twisted by the
monodromy $\log y\mapsto\log y-4\times2\pi\sqrt{-1}$;
\[
U_{AB}U_{B\tilde{A}}U_{\tilde{A}C}U_{CA}=\left(\begin{smallmatrix}1 & 0 & 0 & 0 & 0 & 0\\
0 & 1 & 0 & 0 & 0 & 0\\
\m-1 & 0 & 1 & 0 & 0 & 0\\
0 & 0 & 0 & 1 & 0 & 0\\
0 & \m-16 & 0 & 0 & 1 & 0\\
0 & 0 & 0 & 1 & 0 & 1
\end{smallmatrix}\right)^{4}.
\]
{}{This twist may be explained} by the fact that we have
resolved 4-th order tangency at the intersection point $B$ of $\left\{ y=0\right\} $
and $\left\{ dis_{0}=0\right\} $. $\hfill\square$
\end{rem}
\begin{figure}
\includegraphics[scale=0.5]{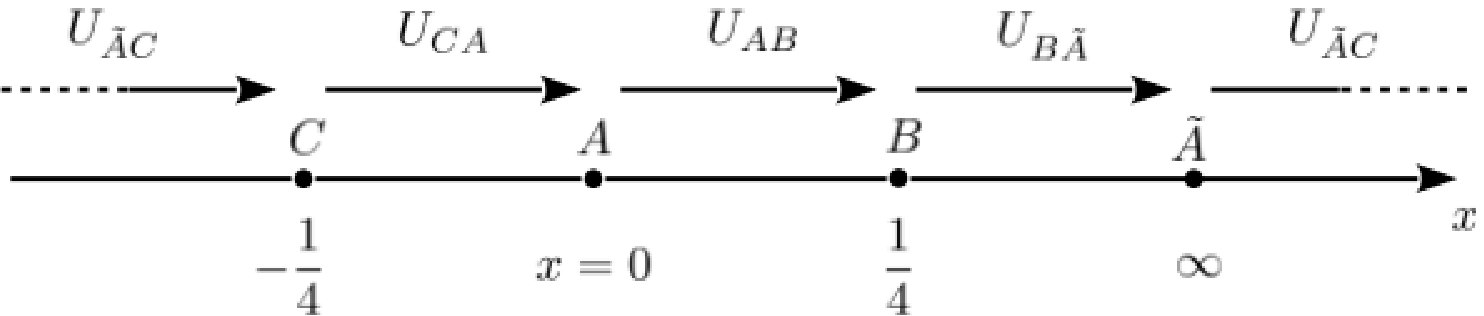}

\caption{\label{fig:Analytic-Path}Fig.5 }
\end{figure}

\para{Integral and symplectic structures.} As summarized in Appendix
\ref{sec:App-Canonical-Form}, mirror symmetry arises from Picard-Fuchs
differential equations of a family in the canonical forms $\Pi_{P}$
of local solutions at special boundary points $P$. It is conjectured
in general \cite{HosIIA,CCarge} that the canonical form $\Pi_{P}$
introduces an integral and symplectic structure on the space of solutions
which are compatible with mirror symmetry (see \cite{Galk-InIritani}
for a generalization to Fano varieties). Here, to go into further
implications of mirror symmetry, let us assume this conjecture and
denote by $\cH_{\mbZ,P}\simeq\mbZ^{6}$ the integral structure generated
by the local solutions $\Pi_{P}$. As described in Appendix \ref{sec:App-Canonical-Form},
the solutions in $\Pi_{P}$ introduce a symplectic basis of $H_{3}$,
\begin{equation}
\left\{ \gamma_{k}^{P}\right\} =\left\{ \alpha_{0}^{P},\alpha_{1}^{P},\alpha_{2}^{P},\beta_{2}^{P},\beta_{1}^{P},\beta_{0}^{P}\right\} \label{eq:symp-basis-gamma}
\end{equation}
with $\langle\alpha_{i}^{P},\beta_{j}^{P}\rangle=-\langle\beta_{j}^{P},\alpha_{i}^{P}\rangle=\delta_{ij}$.
We note that the connection matrix $U_{AB}$ in Proposition \ref{prop:Conection-Matrices}
is symplectic and integral, while $U_{CA}$ is symplectic but not
integral. From this, we observe the following 
\begin{prop}
\label{prop:Int-str-AB-C}The integral structures $\cH_{\mbZ,A}$
and $\cH_{\mbZ,B}$ are isomorphic to each other, while they are not
isomorphic to $\cH_{\mbZ,C}$. 
\end{prop}
In Proposition \ref{prop:MS-Result-Main}, three boundary points $A,B,C$
are grouped into $A,B$ and $C$ from the transformation properties
(connection properties) of the higher genus potential functions at
$g=1$ and $2$. Combining this with the above result, we deduce again
that the integral structure $\cH_{\mbZ,A}\simeq\cH_{\mbZ,B}$ comes
from the local system $R^{3}\pi_{*}\mbZ_{\frak{X}}$ associated with
the family $\frak{X}=\cV_{\mbZ_{8}}^{1}\rightarrow\mbP_{\Delta}$,
while $\cH_{\mbZ,C}$ comes from that of the other family $\frak{X}=\cV_{\mbZ_{8}\times\mbZ_{8}}^{1}\rightarrow\mbP_{\Delta}$.

Our result shows that BCOV theory of higher genus potential functions
$F_{g}$ (topological string for physicists) depends on the integral
structure $R^{3}\pi_{*}\mbZ_{\frak{X}}$ although it is encoded in
the theory in an implicit way. It is argued in physics that a topological
string $\left\{ F_{g}^{M}\right\} $ is equivalent to a wave function
of a quantum mechanical system on $H^{3}(M,\mbC)$ with the natural
symplectic structure on $H^{3}(M,\mbZ)$ (see e.g. \cite{W.BI,Ag}).
We expect that our examples motivate investigations towards a global
mathematical theory of these interesting ideas as well as BCOV recursion
formulas. 
\begin{rem}
In the above proposition, we have excluded the boundary point $\tilde{A}$
from our consideration. This is because $A$ and $\tilde{A}$ are
related by an involutive symmetry (\ref{eq:inv-sym}), which actually
is a symmetry of the Picard-Fuchs differential equations. One may
start all our calculations with $\tilde{A}$ and obtain the same results,
e.g. Proposition \ref{prop:Conection-Matrices}, with $A$ replaced
by $\tilde{A}$.$\hfill\square$
\end{rem}
\para{Derived equivalence and mirror symmetry.} The mirror Calabi-Yau
manifolds for the boundary points $A$ and $B$ are identified, respectively,
with $V_{A}:=V_{8,w}^{1}$ and $V_{B}:=V_{8,w}^{1}/\mbZ_{8}\times\mbZ_{8}$,
which are derived equivalent. The equivalence is described by the
fiberwise Fourier-Mukai transformation \cite{Bak,Sch}. In the mirror
side, i.e. in the family of $\cV_{\mbZ_{8}}^{1}\rightarrow\mbP_{\Delta}$,
we can read this Fourier-Mukai transformation in the form of the connection
matrix $U_{AB}$. 

Let us recall that, under homological mirror symmetry \cite{Ko},
the integral symplectic structure $(\cH_{\mbZ,A},\left\{ \gamma_{k}^{A}\right\} )\simeq(\cH_{\mbZ,B},\left\{ \gamma_{k}^{B}\right\} )$
from the boundary points is transformed to the corresponding integral
symplectic structures on the Grothendieck group $K(V_{A})\simeq K(V_{B})$
of {}{coherent sheaves }(see e.g. \cite{CCarge}). Note
that the isomorphism follows from the derived equivalence $D^{b}(V_{A})\simeq D^{b}(V_{B})$,
and there the symplectic structure is naturally introduced by $\chi(\cE,\cF)=\sum_{i}(-1)^{i}\dim Ext^{i}(\cE,\cF)$.
Mirror symmetry predicts that there are integral symplectic bases
$\left\{ \cE_{\gamma_{k}}^{A}\right\} ,\left\{ \cE_{\gamma_{k}}^{B}\right\} ,$
which are mirror duals to the bases $\left\{ \gamma_{k}^{A}\right\} $
and $\left\{ \gamma_{k}^{B}\right\} $, respectively. Not so much
is known about the integral symplectic basis $\left\{ \cE_{\gamma_{k}}\right\} $;
however, for the symplectic basis (\ref{eq:symp-basis-gamma}) coming
from the canonical form of period integrals $\Pi_{P}$, it is expected
that the following naive correspondences hold:
\begin{equation}
\left\{ \gamma_{k}\right\} =\left\{ \alpha_{0},\alpha_{1},\alpha_{2},\beta_{2},\beta_{1},\beta_{0}\right\} \underset{\text{\text{MD}}}{\longleftrightarrow}\left\{ \cO_{p},\cO_{C_{1}},\cO_{C_{2}},\cO_{D_{2}},\cO_{D_{1}},\cO_{V}\right\} =\left\{ \cE_{\gamma_{k}}\right\} ,\label{eq:HMS}
\end{equation}
where $\cO_{p},\cO_{V}$ are the skyscraper sheaf supported on $p\in V$
and the structure sheaf of a mirror Calabi-Yau manifold $V$, respectively.
$\cO_{C_{i}},\cO_{D_{i}}$ are torsion sheaves supported on curves
$C_{i}$ and dual divisors $D_{i}$. More precisely, we need suitable
twists on these sheaves, but we omit these details for simplicity. 

Assuming the above mirror duality (MD), we read the connection matrix
$U_{AB}$ as in the following diagram:\def\XXX{
\begin{xy}
(-50,0)*++{\left\{\gamma_k^A\right\}=}="gA",
(-43,0)*+{\{}="x1",
(-6,0)*++{\}}="x2",
(-40,0)*+{\alpha_0,}="a0",
(-34,0)*+{\alpha_1,}="a1",
(-28,0)*{\alpha_2,}="a2",
(-21,0)*{\beta_2,}="b2",
(-15,0)*{\beta_1,}="b1",
(-9,0)*{\beta_0}="b0",
(6.5,0)*+{\{}="y1",
(53.5,0)*++{\}}="y2",
(10,0)*+{\cO_p,}="p",
(18,0)*+{\cO_C,}="C1",
(26,0)*{\cO_\sigma,}="C2",
(34,0)*{\cO_{A},}="D2",
(42,0)*{\cO_D,}="D1",
(50,0)*{\cO_{V} }="V",
(-43,-15)*+{\{}="xx1",
(-5,-15)*++{\}}="xx2",
(-50,-15)*++{\left\{\gamma_k^B\right\}=}="gB",
(-39,-15)*+{-\beta_2,}="Ba0",
(-32,-15)*+{\alpha_1,}="Ba1",
(-27,-15)*+{\beta_0,}="Ba2",
(-21,-15)*+{-\alpha_0,}="Bb2",
(-14,-15)*+{\beta_1,}="Bb1",
(-8.5,-15)*+{\alpha_2}="Bb0",
(6.5,-15)*+{\{}="yy1",
(54.5,-15)*++{\}},
(11,-15)*+{ \cO_{p'},}="Bp",
(19,-15)*+{\cO_{C'},}="BC1",
(27,-15)*+{\cO_{\sigma'},}="BC2",
(35,-15)*+{\cO_{A'},}="BD2",
(43,-15)*+{\cO_{D'},}="BD1",
(51,-15)*+{\cO_{V'}}="BV",
(38,-10)*{\text{\tiny $\times$(-1)}},
(16,-10)*{\text{\tiny $\times$(-1)}},
\ar^{U_{AB}} "gA";"gB"
\ar @(rd,u) "a0";"Bb2"
\ar @(ld,u) "b2";"Ba0"
\ar @(ld,u) "b0";"Ba2"
\ar @(rd,u) "a2";"Bb0"
\ar @(rd,u) "p";"BD2"
\ar @(ld,u) "D2";"Bp"
\ar @(ld,u) "V";"BC2"
\ar @(rd,u) "C2";"BV"
\ar@{<->}_{\text{MD}} "x2";"y1"
\ar@{<->}_{\text{MD}} "xx2";"yy1"
\end{xy} }

\begin{equation}
\begin{matrix}\XXX\end{matrix}\label{eq:Mab-HomologicalMS}
\end{equation}
where $\alpha_{i},\beta_{j}$ represent the symplectic bases in $(\cH_{\mbZ,A},\left\{ \gamma_{k}^{A}\right\} )$.
On the mirror side, $\sigma,A$ and $\sigma',A'$ represent a section
and a fiber of the abelian surface fibrations $V:=V_{A}\rightarrow\mbP^{1}$
and its dual fibration $V':=V_{B}\rightarrow\mbP^{1}$, respectively.
The relations between the sheaves indicated in the above diagram exactly
match with the actions of the fiberwise Fourier-Mukai transformation
(cf. \cite{Bak}).

\subsection{\label{sub:degen-A-B}Degenerations of the family $\protect\cV_{\protect\mbZ_{8}}^{1}$
over $A$ and $B$}

Since the degeneration point $A$ is the origin of the affine coordinate
$(x,y)$ which is related to $[w_{0},w_{1},w_{2}]\in\mbP_{w}^{2}$
by (\ref{eq:xy-definition}), we may represent it by a limit of $[t,t,1]$
($t\rightarrow0$). Similarly, it is easy to find a suitable limit
which represents the point $B$ by writing the blow-up coordinate
$(z_{B}^{1},z_{B}^{2})$ in (\ref{eq:affine-relation-x-zB}) in terms
of $[w_{0},w_{1},w_{2}]$ as 
\[
(z_{B}^{1},z_{B}^{2})=\big(\frac{w_{2}^{2}-w_{0}^{2}}{8w_{2}^{2}},-\frac{256w_{1}^{4}w_{2}^{5}}{w_{0}(w_{0}^{2}-w_{2}^{2})^{4}}\big).
\]
Let us denote by $p_{AB}(t)\,(0<t<1)$ a path in $\mbP_{w}^{2}$ for
the analytic continuation $U_{AB}$, satisfying $(x(p_{AB}(t)),y(p_{AB}(t))\rightarrow A$
and $(z_{B}^{1}(p_{AB}(t)),z_{B}^{2}(p_{AB}(t)))\rightarrow B$ for
$t\rightarrow0$ and $1$, respectively, on the blow-up of $\mbP_{\Delta}$. 
\begin{prop}
For a large integer $n$, the following $path$ in $\mbP_{w}^{2}$
\[
p_{AB}(t)=[t,t^{n}(t-1)^{n},1]\,\,\,(0<t<1)
\]
describes a path used for the analytic continuation from the degeneration
point $A$ to $B$. \end{prop}
\begin{proof}
In terms of the affine coordinate $(x,y)$ in (\ref{eq:xy-definition}),
$A$ and $B$ are represented by $(0,0)$ and $(\frac{1}{4},0)$,
respectively. From the definition (\ref{eq:xy-definition}), we have
\[
(x(p_{AB}(t)),y(p_{AB}(t))=\big(\frac{t^{2}}{4},-2\,t^{n-1}(t-1)^{n}\big).
\]
Since $y(p_{AB}(t))$ is sufficiently close to zero when $n$ is large,
we obtain a path which we use for the analytic continuation.
\end{proof}
From $p_{AB}(0)=[0,0,1]$ and $p_{AB}(1)=[1,0,1]$, we read the ideals
for the degenerations over $A$ and $B$ as follows:
\[
\begin{aligned}I_{A} & =\langle x_{2}x_{6},\,\,x_{3}x_{7},\,\,x_{4}x_{0},\,\,x_{5}x_{1}\rangle,\\
I_{B} & =\langle\,\,x_{0}^{2}+x_{4}^{2}+2x_{2}x_{6},\,\,x_{1}^{2}+x_{5}^{2}+2x_{3}x_{7},\,\,\\
 & \hsp{20}x_{2}^{2}+x_{6}^{2}+2x_{0}x_{4},\,\,x_{3}^{2}+x_{7}^{2}+2x_{1}x_{5}\rangle.
\end{aligned}
\]
It is easy to see from $I_{A}$ that the family $\cV_{\mbZ_{8}}^{1}$
degenerates to 16 $\mbP^{3}$s, consisting coordinate subspaces, whose
configuration is displayed in the dual intersection diagram in Fig.\ref{fig:dual-Int-Diag}
(where 16 $\mbP^{3}$ are represented by vertices). This degeneration
was first appeared in \cite{Pav} and was studied in detail there.
As the proposition below shows, the degeneration over $B$ has exactly
the same type as $A$, but none of 16 $\mbP^{3}$s is given by coordinate
subspace.

\begin{figure}
\includegraphics[scale=0.4]{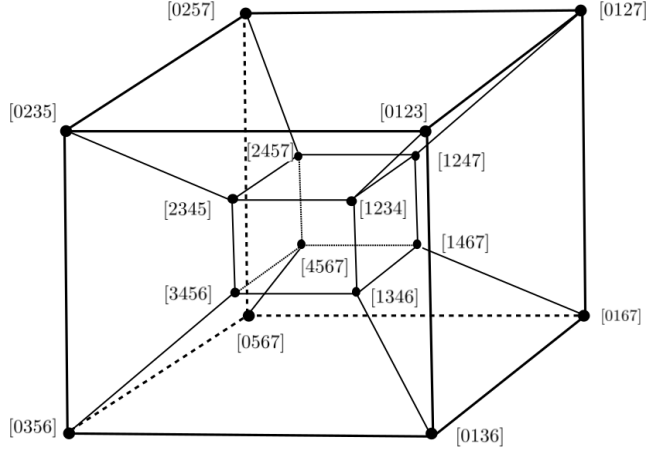}

\caption{Fig.6 \label{fig:dual-Int-Diag}Dual intersection diagram of $V(I_{A})$.
The bracket $[ijkl]$ assigned to each vertex represents a linear
subspace ($\simeq\protect\mbP^{3}$) in $\protect\mbP^{7}$ defined
by the ideal $\langle x_{i},x_{j},x_{k},x_{l}\rangle$. This diagram
should be identified with the three dimensional faces of a four dimensional
cube, which gives a polyhedral decomposition of three sphere $S^{3}$.
The other degenerations $V(I_{B})$ and $V(I_{C})$ are described
by the same diagram under the correspondences $x_{k}\leftrightarrow z_{k}\leftrightarrow u_{k}$. }
\end{figure}

\begin{prop}
\label{prop:Degen-B}By the following linear change of variables with
$i=\sqrt{-1}$,
\begin{equation}
\begin{aligned}z_{0} & =x_{0}\,+\,\,x_{2}-x_{4}\,-\,\,x_{6}, &  &  & z_{1} & =x_{1}\,+\,\,x_{3}-x_{5}\,-\,\,x_{7},\\
z_{2} & =x_{0}-i\,x_{2}+x_{4}-i\,x_{6}, &  &  & z_{3} & =x_{1}-i\,x_{3}+x_{5}-i\,x_{7},\\
z_{4} & =x_{0}\,-\,\,x_{2}-x_{4}\,+\,\,x_{6}, &  &  & z_{5} & =x_{1}\,-\,\,x_{3}-x_{5}\,+\,\,x_{7},\\
z_{6} & =x_{0}+i\,x_{2}+x_{4}+i\,x_{6}, &  &  & z_{7} & =x_{1}+i\,x_{3}+x_{5}+i\,x_{7},
\end{aligned}
\label{eq:phaseB-linear-zx}
\end{equation}
the four generators of $I_{B}$ are expressed by 
\begin{equation}
\begin{alignedat}{2}x_{0}^{2}+x_{4}^{2}+2\,x_{2}x_{6}=\frac{1}{2}(z_{2}z_{6}+z_{0}z_{4}), & \,\,\, & x_{1}^{2}+x_{5}^{2}+2\,x_{3}x_{7}=\frac{1}{2}(z_{3}z_{7}+z_{1}z_{5}),\\
x_{2}^{2}+x_{6}^{2}+2\,x_{0}x_{4}=\frac{1}{2}(z_{2}z_{6}-z_{0}z_{4}), & \,\,\, & x_{3}^{2}+x_{7}^{2}+2\,x_{1}x_{5}=\frac{1}{2}(z_{3}z_{7}-z_{1}z_{5}).
\end{alignedat}
\label{eq:phaseB-equation}
\end{equation}

\end{prop}
From this proposition, we see that the degenerations of the family
over $A$ and $B$ are isomorphic under correspondences of linear
subspace $\left\{ x_{k}=0\right\} $ and $\left\{ z_{k}=0\right\} $,
while the transformation properties under the Heisenberg group $\cH_{8}$
differ from each other (see Proposition \ref{prop:actions-on-ABC}
below).

\subsection{Degeneration of the family $\protect\cV_{\protect\mbZ_{8}\times\protect\mbZ_{8}}^{1}$
over $C$. }

The degeneration point $C$ is located at the origin of $(s_{1},s_{2})=(\frac{1}{8}(4x+1),\frac{y}{4x+1})$
as described in \ref{para:Yukawa-C}. As in the preceding subsection,
we have
\[
(s_{1},s_{2})=\big(\frac{w_{0}^{2}+w_{2}^{2}}{8w_{2}^{2}},-\frac{2w_{1}^{4}}{w_{0}w_{2}(w_{0}^{2}+w_{2}^{2})}\big),
\]
and realize the origin by a suitable limit in $\mbP_{w}^{2}$. We
denote by $p_{AC}(t)\,(0<t<1)$ a path in $\mbP_{w}^{2}$ which represent
the path adapted for the analytic continuation $U_{AC}=(U_{CA})^{-1}$
in \ref{para:AnalyticC}; it satisfies $(s_{1}(p_{AC}(t)),s_{2}(p_{AC}(t)))\rightarrow A$
and $C$ when $t\rightarrow0$ and $1$, respectively. 
\begin{prop}
For a large integer $n$, the following $path$ in $\mbP_{w}^{2}$
\[
p_{AC}(t)=[i\,t,(i)^{\frac{1}{4}}\,t^{n}(t-1)^{n},1]\,\,\,(0<t<1)
\]
describes a path for the analytic continuation from $A$ to $C$. 
\end{prop}
To show the above claim, we simply need to write the $(x,y)$ coordinate
of $p_{AC}(t)$, 
\[
(x(p_{AC}(t)),y(p_{AC}(t)))=\big(-\frac{t^{2}}{4},-2\,t^{n-1}(t-1)^{n}\big).
\]
using the definition (\ref{eq:xy-definition}). When $n$ is large,
we obtain a path for the analytic continuation. Now, corresponding
to Proposition \ref{prop:Degen-B}, we have 
\begin{prop}
\label{prop:Degen-C}If we change the variables by 
\begin{equation}
\begin{aligned}u_{0} & =\,\,\,x_{0}+\xi\,x_{2}\,-\,\,x_{4}+\xi\,x_{6}, &  &  & u_{1} & =\,\,\,x_{1}+\xi\,x_{3}\,-\,\,x_{5}-\xi\,x_{7},\\
u_{2} & =\xi\,x_{0}\,+\,\,x_{2}+\xi\,x_{4}\,-\,\,x_{6}, &  &  & u_{3} & =\xi\,x_{1}\,+\,\,x_{3}+\xi\,x_{5}\,-\,\,x_{7},\\
u_{4} & =-x_{0}+\xi\,x_{2}\,+\,\,x_{4}+\xi\,x_{6}, &  &  & u_{5} & =-x_{1}+\xi\,x_{3}\,+\,\,x_{5}+\xi\,x_{7},\\
u_{6} & =\xi\,x_{0}\,-\,\,x_{2}+\xi\,x_{4}\,+\,\,x_{6}, &  &  & u_{7} & =\xi\,x_{1}\,-\,\,x_{3}+\xi\,x_{5}\,+\,\,x_{7},
\end{aligned}
\label{eq:phaseC-linear-ux}
\end{equation}
where $\xi=(-1)^{\frac{1}{4}}$, then the four generators of $I_{C}$
are expressed by 
\[
\begin{alignedat}{3} & \frac{i}{2}(x_{0}^{2}+x_{4}^{2})+x_{2}x_{6}=u_{2}u_{6}-i\,u_{0}u_{4}, &  &  & \,\,\,\,\, & \frac{i}{2}(x_{1}^{2}+x_{5}^{2})+x_{3}x_{7}=u_{3}u_{7}-i\,u_{1}u_{5},\\
 & \frac{i}{2}(x_{2}^{2}+x_{6}^{2})+x_{0}x_{4}=u_{0}u_{4}-i\,u_{2}u_{6}, &  &  & \,\,\,\,\,\, & \frac{i}{2}(x_{3}^{2}+x_{7}^{2})+x_{1}x_{5}=u_{1}u_{5}-i\,u_{3}u_{7}.
\end{alignedat}
\]

\end{prop}
It is easy to verify the above proposition. From this proposition,
we see an isomorphisms of $V(I_{C})$ to $V(I_{A})$ under the correspondences
between the linear subspaces $\left\{ x_{k}=0\right\} $ and $\left\{ u_{k}=0\right\} $.

\subsection{Heisenberg group actions on the degenerations}

The equations in Propositions \ref{prop:Degen-B}, \ref{prop:Degen-C}
entail the same dual intersection diagram as that for $I_{A}$ (see
Fig.\ref{fig:dual-Int-Diag}), under the correspondences between linear
subspaces in $\mbP^{7}$; 
\[
P_{k}:=\left\{ x_{k}=0\right\} \,\,\leftrightarrow\,\,Q_{k}:=\left\{ z_{k}=0\right\} \,\,\leftrightarrow\,\,R_{k}:=\left\{ u_{k}=0\right\} .
\]
 We note that Heisenberg group $\cH_{8}=\langle\sigma,\tau\rangle$
acts naturally on the finite set $\left\{ P_{k}\right\} $ as permutations,
which we write 
\begin{equation}
\sigma:\,(P_{0}P_{7}P_{6}P_{5}P_{4}P_{3}P_{2}P_{1}),\,\,\,\,\,\tau:\,(P_{0})(P_{1})\cdots(P_{7})\label{eq:H8-act-A}
\end{equation}
by using cycle notation of permutations. It is easy to derive the
following results from the linear relations of $z_{k}$ and $u_{k}$
to $x_{k}$. 
\begin{prop}
\label{prop:actions-on-ABC}The generators $\sigma$ and $\tau$ of
$\cH_{8}$ act on the linear subspaces $\left\{ Q_{k}\right\} $ and
$\left\{ R_{k}\right\} $ for $I_{B}$ and $I_{C}$, respectively,
as
\begin{equation}
\begin{alignedat}{2} & \sigma:\,(Q_{0}Q_{5}Q_{4}Q_{1})(Q_{2}Q_{7}Q_{6}Q_{3}), & \,\,\,\, & \tau:\,(Q_{0}Q_{2}Q_{4}Q_{6})(Q_{1}Q_{3}Q_{5}Q_{7}),\\
 & \sigma:\,(R_{0}R_{7}R_{6}R_{5}R_{4}R_{3}R_{2}R_{1}), &  & \tau:\,(R_{0}R_{2}R_{4}R_{6})(R_{1}R_{3}R_{5}R_{7}).
\end{alignedat}
\label{eq:H8-act-BC}
\end{equation}
 
\end{prop}
For the mirror family $\cV_{\mbZ_{8}}^{1}$ of $V_{8,w}^{1}$ (and
also $V_{8,w}^{1}/\mbZ_{8}\times\mbZ_{8}$), we consider a $\mbZ_{8}$-quotient
of $\cV^{1}\rightarrow\mbP_{w}^{2}$ by the subgroup $\mbZ_{8}\simeq\langle\tau\rangle$
of $\mbZ_{8}\times\mbZ_{8}(\simeq\cH_{8}$ projectively). As described
above, for the degenerations over $A$ and $B$, the (dual) intersection
complex of the linear subspaces are isomorphic but the actions of
the $\mbZ_{8}$ subgroup differ from each other. We need to study
more to reveal how the mirror symmetry we observed in Sections \ref{sec:MSbyPF-A},\ref{sec:MS-B},\ref{sec:MS-more-C}
are encoded in the geometry of the degenerations over $A$ and $B$
(and also $C$). We leave more detailed analysis for future investigations. 
\begin{rem}
\label{rem:H8-action-ABC}We notice a similarity in the linear relations
in (\ref{eq:phaseB-linear-zx}) and (\ref{eq:phaseC-linear-ux}) as
well as the ideals $I_{B}$ and $I_{C}$. However, there is a sharp
difference between the two in their symmetry properties: The four
generators of the ideal $I_{C}$ follow from the symmetry relation
(\ref{eq:g-act-f}) 
\[
f_{i}(\rho(g).w,\,g.x)=\sum c_{ij}(g)f_{j}(w,x)
\]
for $g=S.S$ by setting $w=(0,0,1)$ and replacing $x$ with $g^{-1}.x$.
Then the linear relation (\ref{eq:phaseC-linear-ux}) is identified
with $u=g^{-1}.x$. On the other hand, we can verify that no group
element $g$ in $\cN\cH_{8}$ explains the relation (\ref{eq:phaseB-equation}).
Since the element $\bar{S}.\bar{S}$ is contained in $G_{1}^{ex}\setminus G_{0}$
as we see in (\ref{eq:G1-ex}), two points $A$ and $C$ are identified
in the quotient $\mbP_{w}^{2}/G_{1}^{ex}$ as well as in the quotient
$\mbP_{w}^{2}/G_{\rho}\simeq\mbP^{2}(3,2,1)$. The fact that there
is no such element $g\in\cN\cH_{8}$ for $I_{B}$ indicates that $B$
is mapped to a point different from $A$ and $C$ in $\mbP^{2}(3,2,1)$.
In fact, using the invariants given in Proposition \ref{prop:invariants-G1},
we see that both the points $w=(0,0,1)$ and $w=(i,0,1)$ are mapped
to $[0,1,0]$ in $\mbP^{2}(3,2,1)$ while $w=(1,0,1)$ is mapped to
$[0,0,1]\in\mbP^{2}(3,2,1)$. $\hfill\square$
\end{rem}
\vskip2cm

\section{\textbf{Summary and Discussions}}

We have presented a thorough study on two families coming from a family
$\cV^{1}\rightarrow\mbP_{w}^{2}$ of Calabi-Yau manifolds with vanishing
Euler numbers. We have observed that mirror symmetry, including Fourier-Mukai
partners, arises nicely from the boundary points of the relevant family.
Assuming mirror symmetry, we have determined Gromov-Witten potentials
$F_{g}^{M}$ for $g=0,1,2\,(M=A,B)$ and $g=0,1$ ($M=C$); and found
that they are written in terms of quasi-modular forms. From these
results, one may expect that these topological string theories are
completely integrable in these cases. 

The $q$-series $Z_{g,n}^{M}(q)$ for $A$ and $B$ counts Gromov-Witten
invariants of Calabi-Yau manifolds which are derived equivalent to
each other. These counting problems of curves should have corresponding
moduli problems of stable sheaves on these Calabi-Yau manifolds. Deriving
our counting functions from this point of view is eagerly desired
but beyond the scope of this paper. We expect some interesting geometry
of moduli spaces of sheaves on Calabi-Yau manifolds will appear from
this problem. 

Studying derived equivalences of Calabi-Yau threefolds with Picard
numbers one in the context of mirror symmetry was initiated by an
interesting example found in \cite{Rod,BKvS} (and further studied
in \cite{BC,HoriT}) which is related to classical projective duality
between Grassmannians and Pfaffians. This example was extended later
to the other classical projective duality of the so-called symmetroids
in \cite{HT-JAG,HT-JDG}. These derived equivalences are now understood
in a framework called homological projective duality due to Kuznetsov
\cite{Ku}. In this framework, the derived equivalence in this paper
may be described by homological projective dualities for joins \cite{Ku-Perry}.
Note that $V_{8,w}^{1}$ degenerate to joins of two elliptic quartic
curves over the discriminant $L_{1}$ as summarized in Proposition
\ref{prop:V1-degeneration-loci}. Starting with joins of elliptic
curves, we can find some other examples of derived equivalent Calabi-Yau
manifolds. Recently, Inoue \cite{Inoue} has reported interesting
examples of derived equivalent Calabi-Yau manifolds from this construction
by joins. It should be noted that Calabi-Yau manifolds which come
from joins of elliptic curves have Picard numbers greater than one,
like $V_{8,w}^{1}$. Mirror families of these Calabi-Yau manifolds
are also discussed in \cite{Inoue} finding interesting relations
to fiber-products of elliptic modular surfaces. In a paper \cite{Kanpp},
these Calabi-Yau manifolds are formulated in physics languages called
gauged linear sigma models (GLSM), which describe the $A$-side of
the stringy moduli spaces (K\"ahler moduli) directly. These Calabi-Yau
manifolds are expected to give us interesting examples of mirror symmetry. 

~

In Section \ref{sec:Degens-CY}, we have described degenerations of
the families $\cV_{\mbZ_{8}}^{1}$ and $\cV_{\mbZ_{8}\times\mbZ_{8}}^{1}$
over the points $A,B$ and $C$, up to quotients by $\mbZ_{8}$ or
$\mbZ_{8}\times\mbZ_{8}$. We have found that the types of degenerations
are all isomorphic but the group actions differ from each other. We
need to construct a real torus fibration and take suitable quotients
to extract the corresponding mirror geometries, which we leave for
future works. For the degeneration over $A$, however, the geometric
mirror construction due to Gross and Siebert \cite{GS1,GS2} has been
performed in detail in \cite{Pav}. We have now discovered two additional
degenerations over $B$ and $C$ where mirror symmetry to different
Calabi-Yau manifolds emerges, in particular mirror symmetry to a Fourier-Mukai
partner appears from $B$. It is an interesting problem to see how
the geometry of Fourier-Mukai partners arises in the geometric approaches
\cite{SYZ,GS1,GS2} to mirror symmetry. 

\newpage

\selectlanguage{english}%
\appendix
\renewcommand{\themyparagraph}{{\Alph{section}.\arabic{subsection}.\alph{myparagraph}}}

\selectlanguage{american}%

\section{\label{sec:AppendixA-PF}\textbf{Picard-Fuchs differential operators
$\protect\cD_{2}$ and $\protect\cD_{3}$}}

Here we present the operators $\cD_{2}$ and $\cD_{3}$ in Proposition
\ref{prop:PFeqs-A-A'} in terms of the affine coordinate $x,y$ in
(\ref{eq:xy-definition}). They are given by 
\begin{equation}
\cD_{2}=\sum_{i+j\leq2}p_{ij}(x,y)\theta_{x}^{i}\theta_{y}^{j},\,\,\,\,\,\,\,\cD_{3}=\sum_{i+j\leq3}q_{ij}(x,y)\theta_{x}^{i}\theta_{y}^{j},\label{eq:App-PFeqsD2D3}
\end{equation}
where $\theta_{x}:=x\frac{\partial\;}{\partial x},\,\theta_{y}:=y\frac{\partial\;}{\partial y}$,
and $p_{ij},q_{ij}$ are polynomials of $x,y$. Non-vanishing polynomials
$p_{ij}$ and $q_{ij}$ are given by
\[
\begin{alignedat}{3} & p_{00}=16x(1-12x)y^{2}, &  & \,\,\, &  & p_{01}=-(1-40x+208x^{2})y^{2},\\
 & p_{10}=2(1+8x-112x^{2})y^{2}, &  &  &  & p_{11}=32x(1-4x)y^{2},\\
 & p_{20}=4(1-4x)(1+4x)y^{2}, &  &  &  & p_{02}=-(1-4x)y\big\{(1-4x)^{2}+(1-12x)y\big\},
\end{alignedat}
\]
and
\[
\begin{aligned}\begin{alignedat}{3} & q_{00}=16xy(1-36x-3y), &  &  &  & q_{01}=y\big((1-36x)(1+60x)-(7-292x)y\big),\\
 & q_{20}=q_{21}=-20(1-4x)y^{2}, &  & \,\,\,\,\,\,\, &  & q_{11}=-2y\big(3(1-4x)(1-36x)+(9-76x)y\big),
\end{alignedat}
\\
\begin{aligned} & q_{10}=-2y\big((1-4x)(1-36x)+(7-68x)y\big), & \,\\
 & q_{12}=-2(1-4x)\big((1+4x+3y)(1-36x)+2y^{2}\big),\\
 & q_{02}=-(1-36x)\big((1-4x)(1+12x)-3(1+28x)y\big)+2(7-200x)y^{2},\\
 & q_{03}=2(1-36x)(1+20x+32x^{2})+(3+20x)(3-76x)y+(7-156x)y^{2}.
\end{aligned}
\end{aligned}
\]

~

~

~

\section{\label{sec:Appendix-Griffiths-Yukawa}\textbf{Griffiths-Yukawa couplings}}

Griffiths-Yukawa couplings $C_{ijk}$ of the family $\cV_{\mbZ_{8}}^{1}\rightarrow\mbP_{\Delta}$
has been introduced in Subsection \ref{sub:Griffiths-Yukawa-couplings}.
Here we present explicit forms of the polynomials $P_{ijk}(x,y)$
in Proposition \ref{prop:GY-Cijk-A};
\[
\begin{alignedat}{2}P_{111} &  &  & =16(1-4x)^{2}(1+12x+64x^{2})\\
 &  &  & \;\;\;\;+32(1+4x)(1-10x+216x^{2})y+16(1-16x+432x^{2})y^{2},\\
P_{112} &  &  & =16(1-4x)^{3}+16(3-80x-240x^{2})y+32(1-36x)y^{2},\\
P_{122} &  &  & =64(1+20x+32x^{2})+64(1+12x)y,\\
P_{222} &  &  & =64(1+4x)+128y.
\end{alignedat}
\]

~

~

~

\section{\label{sec:App-Canonical-Form}\textbf{Canonical form of local solutions}}

The first (or ``classical'') form of mirror symmetry appears as
the isomorphisms between weight monodromy filtration on $H_{3}\otimes\mbR$
and the natural filtrations on $H^{even}=\oplus H^{i,i}\cap\mbR$
coming from the nil-potent action $L(\cdot)=\kappa\wedge\cdot$ in
the hard Lefschetz theorem, where $\kappa$ is a K\"ahler form of
mirror Calabi-Yau manifold. If we start with a local system associated
to a family of Calabi-Yau manifolds, this isomorphism, if exists,
appears in the local solutions arranged in a canonical form. If the
parameter space of the family is one dimensional, the canonical form
was reported in \cite[(2-5)]{HT-JAG}. Here we generalize it to the
form which is applicable to the present case, i.e., a family with
parameter space of dimension two.

Let $P$ be a boundary points given by normal crossing divisors. We
take an affine coordinate $(x_{1},x_{2})$ centered at $P$, and assume
that Picard-Fuchs equations are given {}{in terms of}
this coordinate. In this settings, we can recognize mirror symmetry
if the boundary point satisfies some monodromy properties to be a
large complex structure limit (LCSL)\cite{MoLCSL}. We can summarize
the required monodromy properties and mirror symmetry from $P$ in
the following special form of local solutions: Firstly there exists
unique regular solutions (up to constant), which we write $\omega_{0}(x)$.
Secondly, there exists a symmetric tensor $d_{ijk}$ by which all
other local solutions are written in the following form:
\begin{equation}
\begin{matrix}\begin{alignedat}{2} & \,\,\,\omega_{i}(x)=(\log x_{i})\,\omega_{0}+\omega_{i}^{reg},\\
 & \omega_{2,i}(x)=-\sum_{j,k}d_{ijk}(\log x_{j})(\log x_{k})\,\omega_{0}+2\sum_{j,k}d_{ijk}(\log x_{j})\,\omega_{k}+\omega_{2,i}^{reg}, &  & \quad\\
 & \,\,\omega_{3}(x)=\sum_{i,j,k}d_{ijk}(\log x_{i})(\log x_{j})(\log x_{k})\,\omega_{0}-3\sum_{i,j,k}d_{ijk}(\log x_{i})(\log x_{j})\,\omega_{k},\\
 & \qquad\qquad\qquad\qquad\qquad\qquad\qquad\qquad\qquad\;\;+3\sum_{i}(\log x_{i})\,\omega_{2,i}+\omega_{3}^{reg},
\end{alignedat}
\end{matrix}\label{eq:Appendix-omega-k}
\end{equation}
 where $\omega_{i}^{reg}$ represents a power series with no constant
terms if the power series of $\omega_{0}(x)$ starts from a constant
term, i.e. $\omega_{0}(x)=1+O(x)$. (These solutions coincide with
the local solutions which are obtained by the (generalized) Frobenius
method for hypergeometric series \cite{HKTY1,HKTY2,HLY}, when Picard-Fuchs
equations are given by certain hypergeometric systems.) Finally, when
the boundary point $P$ is mirror symmetric to a Calabi-Yau threefold
$X$ with some choice of nef divisors $H_{1},H_{2}$, the canonical
form of local solutions is given by
\begin{equation}
\Pi_{P}(x)=N_{P}\left(\begin{matrix}1 & 0 & 0 & 0 & 0 & 0\\
0 & 1 & 0 & 0 & 0 & 0\\
0 & 0 & 1 & 0 & 0 & 0\\
\beta_{2} & a_{11} & a_{12} & K/2 & 0 & 0\\
\beta_{1} & a_{21} & a_{22} & 0 & K/2 & 0\\
\gamma & \beta_{1} & \beta_{2} & 0 & 0 & -K/6
\end{matrix}\right)\left(\begin{matrix}\,\,\,\,\omega_{0}\\
n_{1}\,\omega_{1}\\
n_{1}\,\omega_{2}\\
\,\,n_{2}\,\omega_{2,2}\\
\,\,n_{2}\,\omega_{2,1}\\
n_{3}\,\omega_{3}
\end{matrix}\right),\label{eq:App-Canonial-Pi}
\end{equation}
where $n_{k}=\frac{1}{(2\pi i)^{k}}$, $\beta_{k}=-\frac{c_{2}(X).H_{k}}{24},\,\gamma=-\frac{\zeta(3)}{(2\pi i)^{3}}\chi(X)$
and $K$ is a constant satisfying $H_{i}.H_{j}.H_{k}=K\,d_{ijk}$. 

The overall constant $N_{P}$ is not determined, but $N_{P}$ and
$N_{P'}$ for two different boundary points are related by analytic
continuation of the local solutions. The $a_{ij}$ represent the so-called
quadratic ambiguity in the genus zero potential $F_{0}^{X}$, which
does not have any geometric meanings. It is conjectured in \cite{HosIIA,CCarge}
that, if mirror symmetry arises from the boundary point $P$, the
above canonical form $\Pi_{P}(x)$ with suitable choice of $a_{ij}$
gives an integral, symplectic basis of period integrals with respect
to a symplectic form 
\begin{equation}
\Sigma=\left(\begin{matrix}O & J\\
-J & O
\end{matrix}\right)\label{eq:App-Sigma}
\end{equation}
where $J=\left(\begin{smallmatrix}0 & 0 & 1\\
0 & 1 & 0\\
1 & 0 & 0
\end{smallmatrix}\right)$. Although it is implicit, we can transform the period integrals in
$\Pi_{P}(x)$ to the corresponding cycles $\left\{ \alpha_{0},\alpha_{1},\alpha_{2},\beta_{2},\beta_{1},\beta_{0}\right\} $
in $H_{3}$. Then the set of these cycles $\left\{ \alpha_{0},\alpha_{1},\alpha_{2},\beta_{2},\beta_{1},\beta_{0}\right\} $
forms an integral and symplectic basis of $H_{3}$ satisfying $\langle\alpha_{i},\alpha_{j}\rangle=\langle\beta_{i},\beta_{j}\rangle=0$
and $\langle\alpha_{i},\beta_{j}\rangle=\delta_{ij}=-\langle\beta_{j},\alpha_{i}\rangle$,
where $\langle A,B\rangle:=\int\mu_{A}\wedge\mu_{B}$ with $\mu_{C}$
representing the Poincar\'e dual of $C$. 

~

\newpage

\section{\label{sec:Appendix-Propagator-f2}\textbf{Propagator $S^{ab}$ and
$f_{2}(x,y)$ }}

BCOV recursion formulas for $F_{g}$ are simplified when the Euler
number of a Calabi-Yau manifold vanishes. As displayed for $F_{2}$
in (\ref{eq:F2-BCOV}), $F_{g}$ are determined by the propagator
$S^{ab}$ and the rational function $f_{g}$ together with derivatives
of $F_{h}\,\,(h<g)$. In this appendix, we denote by $(x^{i})=(x^{1},x^{2})$
an affine coordinate centered by a degeneration point (LCSL) $P$;
and let $(t^{a})=(t^{1},t^{2})$ be the coordinate introduced by making
ratios of period integrals (Definition \ref{def:MirrorMap-A}).

\subsection{Propagator $S^{ab}$}

The propagator $S^{ab}$ is determined by solving the so called special
geometry relation for the Weil-Petersson metric on the moduli space,
\begin{equation}
\Gamma_{ij}^{k}=\delta_{i}^{\,k}K_{j}+\delta_{j}^{\,k}K_{i}-C_{ijk}S^{mk}+f_{ij}^{k},\label{eq:App-Special-K-rel}
\end{equation}
where $K_{i}=\frac{\partial\;}{\partial x_{i}}K$ is a derivative
of the K\"ahler potential and $f_{ij}^{k}$ is a holomorphic (rational)
tensor which we determine imposing consistency conditions. To solve
the above relation, consider the inverse matrix $C_{i}^{-1}:=(C_{ijk})^{-1}$
for a fixed $i$ and introduce its matrix components by $C_{i}^{-1}=(C_{i}^{jk})$,
then we have 
\[
S^{mk}=\sum_{j}C_{i}^{mj}\left\{ -\Gamma_{ij}^{k}+K_{j}\delta_{i}^{k}+K_{i}\delta_{j}^{k}+f_{ij}^{k}\right\} .
\]
Following \cite{BCOV1,BCOV2}, when we take the so-called topological
limit ``$\bar{t}\rightarrow\infty"$, the special geometry relation
reduces to a holomorphic relation by 
\[
\Gamma_{ij}^{k}\rightarrow\sum_{a}\frac{\partial x^{k}}{\partial t^{a}}\frac{\partial\;}{\partial x^{i}}\frac{\partial t^{a}}{\partial x^{j}},\,\,\,\,\,K_{i}\rightarrow-\frac{\partial\;}{\partial x^{i}}\log\omega_{0}(x).
\]
Then, these limiting relations enable us to write the propagator near
a degeneration point (LCSL) in the following form up to unknown functions
$f_{ij}^{k}$:
\begin{equation}
S^{mk}=\sum_{j}C_{i}^{mj}\left\{ -\sum_{a}\frac{\partial x^{k}}{\partial t^{a}}\frac{\partial\;}{\partial x^{i}}\frac{\partial t^{a}}{\partial x^{j}}-\frac{1}{\omega_{0}(x)}\Big(\frac{\partial\omega_{0}}{\partial x^{j}}\delta_{i}^{k}+\frac{\partial\omega_{0}}{\partial x^{i}}\delta_{j}^{k}\Big)+f_{ij}^{k}\right\} .\label{eq:Appendix-Sxx}
\end{equation}
The propagator $S^{ab}$ in (\ref{eq:F2-BCOV}) is a transform of
the above $S^{ij}$ to the coordinate $(t^{1},t^{2})$ with a suitable
weight factor being introduced, i.e., $S^{ab}=(N_{P}\omega_{0}(x))^{2}\sum_{i,j}\frac{\partial t^{a}}{\partial x^{i}}\frac{\partial t^{b}}{\partial x^{j}}S^{ij}$. 

As we recognize, the r.h.s of (\ref{eq:Appendix-Sxx}) depends on
the choice of the index $i$, whereas the l.h.s does not. To be consistent,
we require that $S^{mk}$ are all equal for $i=1,2$. Also, as being
a propagator, we require that $S^{ij}$ is a symmetric tensor. Finally,
we impose that the unknowns $f_{ij}^{k}$ are given by rational functions
with possible poles along the discriminant loci of the family. 
\begin{prop}
Around the degeneration point $A$ with the coordinate $(x^{1},x^{2})=(x,y)$,
the above requirements on $f_{ij}^{k}$ are satisfied by {\small{}
\begin{equation}
\begin{alignedat}{3} & f_{11}^{1}=-\frac{2}{x}+\frac{32\,x}{dis_{1}dis_{3}}, & \,\,\,\, & f_{11}^{2}=\frac{16}{dis_{3}}-\frac{16\,dis_{2}}{dis_{1}dis_{3}}, & \,\,\,\, & f_{12}^{1}=\frac{8\,x}{dis_{1}dis_{3}},\\
 & f_{12}^{2}=-\frac{1}{x}+\frac{4}{dis_{3}}-\frac{4\,dis_{2}}{dis_{1}dis_{3}}, &  & f_{22}^{1}=-\frac{2\,x}{y\,dis_{3}}+\frac{2\,x}{dis_{1}dis_{3}}, & \,\,\,\, & f_{22}^{2}=-\frac{1}{y}-\frac{dis_{2}}{dis_{1}dis_{3}},
\end{alignedat}
\label{eq:Appendix-fijk}
\end{equation}
}where $f_{ij}^{k}=f_{ji}^{k}$ and $dis_{1}=1+4x+y$, $dis_{2}=1+4x,$
$dis_{3}=1-4x$.\end{prop}
\begin{rem}
The above forms of $f_{ij}^{k}$ are not unique. Here, we have chosen
the simplest possible forms. The choices of $f_{ij}^{k}$ will not
affect the final form $F_{2}$ since the unknown $f_{2}$ will be
fixed accordingly. 
\end{rem}

\subsection{Connecting $A$ and $B$ }

The construction of the propagator $S^{ab}$ (as a contra-variant
tensor) applies in a similar way for the boundary point $B$. Let
us distinguish the resulting propagators by putting indices $A$ and
$B$, 
\[
S_{A}^{ab}=\big(N_{A}\omega_{0}^{A}(x)\big)^{2}\sum_{i,j}\frac{\partial t_{A}^{a}}{\partial x^{i}}\frac{\partial t_{A}^{b}}{\partial x^{j}}S_{A}^{ij},\,\,\,S_{B}^{cd}=\big(N_{B}\omega_{0}^{B}(z)\big)^{2}\sum_{i,j}\frac{\partial t_{B}^{c}}{\partial z^{i}}\frac{\partial t_{B}^{d}}{\partial z^{j}}S_{B}^{ij}.
\]
Note that $S_{A}^{ij}$ contains $f_{A,ij}^{\,\,\,\,\,k}:=f_{ij}^{k}$
determined in (\ref{eq:Appendix-fijk}). Likewise $S_{B}^{ij}$ contains
$f_{B,ij}^{\,\,\,\,\,k}$ which we introduce to solve (\ref{eq:App-Special-K-rel})
in the form (\ref{eq:Appendix-Sxx}). The consistency requirements
for (\ref{eq:Appendix-Sxx}) restrict the forms of $f_{B,ij}^{\,\,\,\,\,k}$
but they are not unique.

In the BCOV formulation \cite{BCOV2}, these propagators supposed
to behave as contra-variant tensor. Hence, to have such property,
we need to have 
\begin{equation}
f_{B,ij}^{\,\,\,\,\,k}=\sum_{m,m,r}\frac{\partial x^{m}}{\partial z^{i}}\frac{\partial x^{n}}{\partial z^{j}}\frac{\partial z^{k}}{\partial x^{r}}f_{A,mn}^{\,\,\,\,\,r}-\sum_{m}\frac{\partial z^{k}}{\partial x^{m}}\frac{\partial\;}{\partial z^{i}}\frac{\partial x^{m}}{\partial z^{j}},\label{eq:fijk-A-B-rel}
\end{equation}
where $(x^{1},x^{2})=(x,y)$ and $(z^{1},z^{2})=(z_{1}^{B},z_{2}^{B})$
are the affine coordinates centered $A$ and $B$, respectively (see
(\ref{eq:affine-relation-x-zB})). We can verify the following property
by explicit calculations. 
\begin{prop}
Among the possible forms of $f_{B,ij}^{\,\,\,\,\,k}$, there exist
unique rational functions $f_{B,ij}^{\,\,\,\,\,k}$ which satisfy
the relation (\ref{eq:fijk-A-B-rel}). 
\end{prop}
For the above $f_{B,ij}^{\,\,\,\,\,k}$ satisfying (\ref{eq:fijk-A-B-rel}),
assuming the connection property (t1) in (\ref{eq:connect-g1}), i.e.
$N_{A}\omega_{0}^{A}=N_{B}\omega_{0}^{B}$, we have the desired covariance,
\begin{equation}
S_{B}^{cd}=\sum_{a,b}\frac{\partial t_{B}^{c}}{\partial t_{A}^{a}}\frac{\partial t_{B}^{d}}{\partial t_{A}^{b}}S_{A}^{ab}.\label{eq:Appendix-cov-Sab}
\end{equation}

\subsection{The forms of $f_{2}^{A}$ and $f_{2}^{B}$ }

The rational function $f_{2}^{A}(x,y)$ is determined by imposing
the transformation rules as described in Proposition \ref{prop:FAB-related-on-PD}.
The result is 
\[
f_{2}^{A}(x,y)=\frac{h_{1}\,y\,dis_{1}+h_{2}\,(y\,dis_{1})^{2}+h_{3}\,(y\,dis_{1})^{3}}{1440\,(dis_{0}\,dis_{2})^{2}dis_{3}},
\]
where $dis_{k}$ are the same as in (\ref{eq:Appendix-fijk}) and
$h_{k}$ are defined with $\bar{x}=4x,\,\bar{y}=64y$ by 
\[
\begin{aligned} &  & h_{1} & =-(1-\bar{x})^{5}(431-120\bar{x}-2512\bar{x}^{2}-1440\bar{x}^{3}+598\bar{x}^{4}\\
 &  &  & \qquad\qquad\qquad\qquad\qquad\qquad\qquad\quad-1080\bar{x}^{5}-1912\bar{x}^{6}-240\bar{x}^{7}+131\bar{x}^{8}),\\
 &  & h_{2} & =128\bar{x}(1-\bar{x})^{3}(431+825\bar{x}-555\bar{x}^{2}-1922\bar{x}^{3}-1095\bar{x}^{4}+105\bar{x}^{5}+131\bar{x}^{6}),\\
 &  & h_{3} & =4096\bar{x}^{2}(431+1315\bar{x}+1280\bar{x}^{2}+280\bar{x}^{3}-295\bar{x}^{4}-131\bar{x}^{5}).
\end{aligned}
\]
As described just above Proposition \ref{prop:FAB-related-on-PD},
by substituting the inverse relation $(x,y)=(\frac{1}{8}(1-8z_{1}),32\,z_{1}^{4}z_{2})$
of (\ref{eq:affine-relation-x-zB}) into $f_{2}^{A}$, we obtain the
function $f_{2}^{B}(z_{1},z_{2})$ in $F_{2}^{B}$. 

~

~

~

~

\section{\label{sec:Appendix-2222}\textbf{Gromov-Witten invariants of (2,2,2,2)$\subset\protect\mbP^{7}$}}

Here we summarize mirror symmetry of complete intersections of four
quadrics in $\mbP^{7}$. If we take four quadrics in general, the
complete intersection is a smooth Calabi-Yau variety which we denote
by $W$. Its topological invariants are 
\[
H_{W}^{3}=16,\,\,c_{2}.H_{W}=64,\,\,h_{W}^{1,1}=1,\,\,h_{W}^{21}=65.
\]
As we encountered in the text, there are quadrics which admit free
$\mbZ_{8}$ and $\mbZ_{8}\times\mbZ_{8}$ actions; for the former
case, we still have smooth complete intersections, while for the latter
case there is no smooth complete intersection which is compatible
with the symmetry \cite{Brown,Hua}. We denote by $W_{1}:=W_{\mbZ_{8}}$
and $W_{2}:=W_{\mbZ_{8}\times\mbZ_{8}}$ the quotients of the complete
intersections of general quadrics which respect the specified symmetry.
Then topological invariants are evaluated to be
\[
\begin{aligned}H_{W_{1}}^{3}=2,\,\,c_{2}.H_{W_{1}}=8,\,\,h_{W_{1}}^{1,1}=1,\,\,h_{W_{1}}^{21}=9,\\
H_{W_{2}}^{3}=128,\,\,c_{2}.H_{W_{2}}=8,\,\,h_{W_{2}}^{1,1}=1,\,\,h_{W_{2}}^{21}=2,
\end{aligned}
\]
where, for the second line, we adopt the numbers from the contraction
of $V_{8,w}^{1}/\mbZ_{8}\times\mbZ_{8}$ in \ref{para:contR-B-Z8Z8}.

\subsection{\label{sub:Appendix-2222}Gromov-Witten invariants of $W$}

Since all the results are well-known in literatures, we summarize
only some of them for reader's convenience. In this case, Batyrev-Borisov
toric mirror construction \cite{BatBo} applies to have a family of
mirror Calabi-Yau manifolds, and we have a hypergeometric differential
equation,
\begin{equation}
\left\{ \theta_{x}^{4}-16x(2\theta_{x}+1)^{4}\right\} w(x)=0\,\,\,\,\,(\theta_{x}:=x\frac{d\;}{dx})\label{eq:AppE-PF2222}
\end{equation}
 which determines period integrals. Mirror symmetry of the family
to $W$ can be observed in the canonical form of local solutions at
the boundary points, $x=0$, as in (\ref{eq:App-Canonial-Pi}). Reducing
(\ref{eq:Appendix-omega-k}) to the present case with $d_{111}=1,$
we have local solutions about $x=0$ by 
\[
\omega_{0}(x),\,\,\omega_{1}(x),\,\,\omega_{2,1}(x),\,\,\omega_{3}(x),
\]
where $\omega_{0}(x)=\sum_{n\geq0}\frac{((2n)!)^{4}}{(n!)^{8}}x^{n}$.
Writing $\omega_{1}=(\log x)\omega_{0}+\omega_{1}^{reg}$ and inverting
the relation 
\[
q=z\,\exp\big(\frac{\omega_{1}^{reg}(x)}{\omega_{0}(x)}\big),
\]
we define the mirror map $x=x(q)$ as a $q$-series. Quantum corrected
Yukawa coupling is calculated by 
\begin{equation}
Y_{ttt}^{0}=\big(\frac{1}{\omega_{0}}\big)^{2}C_{xxx}\big(\frac{dx}{dt}\big)^{3}\,\,\text{with }\,\,\,C_{xxx}=\frac{16}{x^{3}(1-256x)}.\label{eq:App-Yttt0}
\end{equation}
The genus one potential function is given by 
\[
F_{1}=\frac{1}{2}\log\left\{ \big(\frac{1}{\omega_{0}}\big)^{3+1+\frac{128}{12}}(1-256x)^{-\frac{1}{6}}x^{-1-\frac{64}{12}}\frac{dx}{dt}\right\} .
\]
The BCOV potential function $F_{2}$ has been determined in literatures,
see e.g. \cite{Fg2222}.

\subsection{\label{sub:Appendix-free-quot-Z8Z8}Gromov-Witten invariants of $W_{\protect\mbZ_{8}\times\protect\mbZ_{8}}$}

The variety $W_{\mbZ_{8}\times\mbZ_{8}}$ is a free quotient by $W$
but is a singular Calabi-Yau variety. We do not know how to construct
its mirror family. However, naive applications of the results \cite{AM1}
obtained for a similar free quotient (by $\mbZ_{5}\times\mbZ_{5}$)
of quintic threefolds turn out to be consistent with the sum-up property
of our BPS numbers. In this case, we define local parameter $b$ by
$x=b^{8}$ and define the mirror map $b=b(q)$ by 
\[
q=b\,\exp\big(\frac{1}{8}\frac{\omega_{1}^{reg}(x)}{\omega_{0}(x)}\big).
\]
Quantum corrected Yukawa coupling is 
\[
Y_{ttt}=\big(\frac{1}{8\,\omega_{0}(x)}\big)^{2}C_{bbb}\big(\frac{db}{dt}\big)^{3}\,\,\,\text{with}\,\,\,C_{bbb}=C_{zzz}\big(\frac{dx}{db}\big)^{3},
\]
where the numerical factor $\frac{1}{8}$ is introduced to have right
normalization. BCOV formula for $F_{1}$ turns out to be 
\begin{equation}
\begin{alignedat}{1}F_{1} & =\frac{1}{2}\log\Big\{\big(\frac{1}{\omega_{0}(x)}\big)^{3+1-\frac{\chi}{12}}(1-2b)^{-\frac{64}{6}}(1+2b)^{-\frac{16}{6}}(1+4b^{2})^{-\frac{4}{6}}\\
 & \hsp{145}\times(1+16b^{4})^{-\frac{1}{6}}b^{-1-\frac{8}{12}}\frac{db}{dt}\,\Big\}
\end{alignedat}
\label{eq:F1-2222-Z8Z8}
\end{equation}
where $x=b^{8}$, and $\chi=\frac{-128}{|\mbZ_{8}\times\mbZ_{8}|}$
. With these definitions, we verify the sum-up properties of our BPS
numbers of $V_{\mbZ_{8}\times\mbZ_{8}}^{1}$ in Table \ref{tab:Table-B}.
We have also verified the sum-up property at genus two by solving
the BCOV recursion for $F_{2}$.

\subsection{\label{sub:Appendix-free-quot-Z8}Gromov-Witten invariants of $W_{\protect\mbZ_{8}}$}

Free quotient $W_{\mbZ_{8}}$ of $W$ is a smooth Calabi-Yau manifold.
Here, we observe that its Gromov-Witten invariants come from another
degeneration point of (\ref{eq:AppE-PF2222}), i.e., $\frac{1}{x}=0$,
and they satisfy the sum-up properties (\ref{eq:sum-up-C}). To describe
the degeneration point, we introduce a coordinate $z=\frac{1}{16^{4}}\frac{1}{x},$
which transforms the Picard-Fuchs differential equation (\ref{eq:AppE-PF2222})
to 
\[
\big\{(\theta_{z}-\frac{1}{2})^{4}-16z(2\theta_{z})^{4}\big\} w(z)=0.
\]
From this form, it is clear that the local solutions are given by
\[
\sqrt{z}\,\omega_{0}(z),\,\sqrt{z}\,\omega_{1}(z),\,\sqrt{z}\,\omega_{2,1}(z),\,\sqrt{z}\,\omega_{3}(z)
\]
in terms of the four solutions of (\ref{eq:AppE-PF2222}). Then using
the same series $\omega_{1}^{reg}$ as before, we define the mirror
map $z=z(q)$ by 
\[
q=z\,\exp\big(\frac{\sqrt{z}\omega_{1}^{reg}(z)}{\sqrt{z}\omega_{0}(z)}\big)=z\,\exp\big(\frac{\omega_{1}^{reg}(z)}{\omega_{0}(z)}\big),
\]
which has the same form as $x=x(q).$ However the quantum corrected
Yukawa coupling this time is given by 
\begin{equation}
Y_{ttt}^{\infty}=\big(\frac{1}{N_{\infty}\,\sqrt{z}\omega_{0}(z)}\big)^{2}C_{zzz}\big(\frac{dz}{dt}\big)^{3}\,\text{with }C_{zzz}=C_{xxx}\left(\frac{dx}{dz}\right)^{3},\label{eq:App-YtttInfty}
\end{equation}
where $N_{\infty}$ is the normalization factor in the canonical form
$\Pi_{\infty}(z)$ of the local solutions. As in Proposition \ref{prop:Conection-Matrices},
$N_{\infty}$ is determined by analytic continuation of the two local
solutions $\Pi_{0}(x)$ and $\Pi_{\infty}(z)$. As we summarize the
results below, it turns out that $N_{\infty}=32\sqrt{2}$ and we have
\[
Y_{ttt}^{\infty}=\big(\frac{1}{\omega_{0}(z)}\big)^{2}\frac{2}{z^{3}(1-256z)}\big(\frac{dz}{dt}\big)^{3},
\]
which is $\frac{1}{8}$ of $Y_{ttt}^{0}$. We remark that the same
result is derived in \cite{ESh} by evaluating the so-called two sphere
partition function. For genus {}{one potential} function,
we have the BCOV formula given by
\[
\begin{aligned}F_{1} & =\frac{1}{2}\log\Big\{\big(\frac{1}{\omega_{0}(z)}\big)^{3+1-\frac{\chi}{12}}(1-256\,z)^{-\frac{8}{6}}z^{-1-\frac{8}{12}}\frac{dz}{dt}\,\Big\},\end{aligned}
\]
where $\chi=\frac{-128}{|\mbZ_{8}|}=-16$ and the conifold factor
$-\frac{8}{6}$ is determined from the monodromy around the discriminant
$\left\{ 1-256z=0\right\} $ (see $M_{c}$ in Proposition \ref{prop:AppendixE-Monod}
below). 
\begin{table}
{\scriptsize{}
\[
\begin{array}{|c|cccccccc}
\hline \m g\diagdown d\m & \m1 & \m2 & \m3 & \m4 & \m5 & \m6 & \m7 & \m..\\
\hline 0 & \m512 & \m9728 & \m416256 & \m25703936 & \m1957983744 & \m170535923200 & \m16300354777600 & \m..\\
1 & \m0 & \m0 & \m0 & \m14752 & \m8782848 & \m2672004608 & \m615920502784 & \m..\\
2 & \m0 & \m0 & \m0 & \m0 & \m0 & \m1427968 & \m2440504320 & \m..
\\\hline \end{array}
\]
}{\scriptsize \par}

(1) BPS numbers of $W=$(2,2,2,2)

{\scriptsize{}
\[
\begin{aligned} & \begin{array}{|c|ccccccccccccccc}
\hline \m g\diagdown d\m & \m.. & \m8 & \m.. & \m16 & \m.. & \m24 & \m\,\,\m.. & \m\,\,\m32 & \m\,\,\m.. & \m\,\m40 & \m\,\m.. & \m\,\m48 & \m\,\m.. & \m\,\m56 & \m..\\
\hline 0 & \m.. & \m8 & \m.. & \m152 & \m.. & \m6504 & \m\,\,\m.. & \m\,\,\m401624 & \m\,\,\m.. & \m\,\m30593496 & \m\,\m.. & \m\,\m2664623800 & \m\,\m.. & \m\,\m254693043400 & \m..
\\\hline \end{array}\\
 & \begin{array}{|c|cccccccccccccccc}
\hline \m g\diagdown d\m & \m1 & \m2 & \m3 & \m4 & \m5 & \m6 & \m7 & \m8 & \m9 & \m10 & \m11 & \m12 & \m13 & \m14 & \m15 & ..\\
\hline 1 & \m8 & \m8 & \m8 & \m22 & \m24 & \m56 & \m72 & \m168 & \m160 & \m408 & \m488 & \m1454 & \m1496 & \m3912 & \m4632 & ..\\
2 & \m0 & \m0 & \m0 & \m0 & \m0 & \m4 & \m32 & \m155 & \m584 & \m2000 & \m5088 & \m13752 & \m30592 & \m78720 & \m165632 & ..
\\\hline \end{array}
\end{aligned}
\]
}{\scriptsize \par}

(2) BPS numbers of $W_{\mbZ_{8}\times\mbZ_{8}}=$(2,2,2,2)$/\mbZ_{8}\times\mbZ_{8}$

{\scriptsize{}
\[
\begin{array}{|c|ccccccc}
\hline \m g\diagdown d\m & 1 & 2 & 3 & 4 & 5 & 6 & 7\\
\hline 0 & 64 & 1216 & 52032 & 3212992 & 244747968 & 21316990400 & ..\\
1 & 112 & 8736 & 927920 & 110521568 & 14056476368 & 1863171853088 & ..
\\\hline \end{array}
\]
}{\scriptsize \par}

(3) BPS numbers of $W_{\mbZ_{8}}=$(2,2,2,2)$/\mbZ_{8}$

~

\caption{Tables E1. \label{tab:Appendix-2222} Three tables of BPS numbers.
Potential functions $F_{0}$ and $F_{1}$ are given in Appendix \ref{sec:Appendix-2222}.
We have also included $g=2$ BPS numbers for (1) and (2) by calculating
$F_{2}$ for these two cases. }
\end{table}

~

\subsection{Analytic continuation}

Picard-Fuchs equation (\ref{eq:AppE-PF2222}) has two degeneration
points, $x=0$ and $\infty$; from each we obtained quantum corrected
Yukawa couplings $Y_{ttt}^{0}$ and $Y_{ttt}^{\infty}$. These two
are related to each other in a quite parallel way to the relation
$Y_{ijk}^{A}$ and $Y_{ijk}^{C}$. We can see that these two couplings
are actually analytically continued over $\mbP^{1}$ by showing that
$N_{\infty}=32\sqrt{2}$ arises from the analytic continuation of
two local solutions $\Pi_{0}(x)$ and $\Pi_{\infty}(z)$. 

We define the canonical forms of period integrals around $x=0$ and
$\infty$ by
\[
\Pi_{0}(x)=N_{0}\,Z_{top}^{0}\left(\begin{smallmatrix}\omega_{0}(x)\\
n_{1}\omega_{1}(x)\\
n_{2}\omega_{2,1}(x)\\
n_{3}\omega_{3}(x)
\end{smallmatrix}\right),\,\,\Pi_{\infty}(z)=N_{\infty}\,Z_{top}^{\infty}\left(\begin{smallmatrix}\sqrt{z}\omega_{0}(z)\\
n_{1}\sqrt{z}\omega_{1}(z)\\
n_{2}\sqrt{z}\omega_{2,1}(z)\\
n_{3}\sqrt{z}\omega_{3}(z)
\end{smallmatrix}\right),
\]
where $n_{k}=\frac{1}{(2\pi i)^{k}}$ and $Z_{top}^{0}=Z_{top}(16,64,-128)$,
$Z_{top}^{\infty}=Z_{top}(2,8,-16)$ with 
\[
Z_{top}(H^{3},c_{2}.H,\chi)=\left(\begin{matrix}1 & 0 & 0 & 0\\
0 & 1 & 0 & 0\\
-\frac{c_{2}.H}{24} & 0 & \frac{H^{3}}{2} & 0\\
-\chi\frac{\zeta(3)}{(2\pi i)^{3}} & -\frac{c_{2}.H}{24} & 0 & -\frac{H^{3}}{3!}
\end{matrix}\right).
\]
To describe the analytic continuation, we take a path connecting two
points $P_{0}:=0+\ve\sqrt{-1}$ and $P_{\infty}:=\infty+\ve\sqrt{-1}$
$(\ve>0)$. 
\begin{prop}
\label{prop:AppendixE-Monod}$(1)$ In terms of the canonical forms
of period integrals $\Pi_{0}(x)$ and $\Pi_{\infty}(z)$, the monodromy
matrices $M_{0},M_{c},M_{\infty}$ around the three singular points
$x=0,\frac{1}{256},\infty$, respectively, of Picard-Fuchs equation
(\ref{eq:AppE-PF2222}) are given as follows: 
\[
\begin{array}{|c|ccc|}
\hline  & M_{0} & M_{c} & M_{\infty}\\
\hline \Pi_{0}(x) & \left(\begin{smallmatrix}1 & 0 & 0 & 0\\
1 & 1 & 0 & 0\\
8 & 16 & 1 & 0\\
-8 & -8 & -1 & 1
\end{smallmatrix}\right) & \left(\begin{smallmatrix}1 & 0 & 0 & 1\\
0 & 1 & 0 & 0\\
0 & 0 & 1 & 0\\
0 & 0 & 0 & 1
\end{smallmatrix}\right) & \left(\begin{smallmatrix}-7 & 8 & -1 & -1\\
-1 & 1 & 0 & 0\\
8 & -16 & 1 & 0\\
8 & -8 & -1 & 1
\end{smallmatrix}\right)\\
\hline \Pi_{\infty}(z) & \left(\begin{smallmatrix}-1 & 0 & 0 & 8\\
1 & -1 & 0 & -8\\
-1 & 2 & -1 & 8\\
-1 & 1 & -1 & 7
\end{smallmatrix}\right) & \left(\begin{smallmatrix}1 & 0 & 0 & 8\\
0 & 1 & 0 & 0\\
0 & 0 & 1 & 0\\
0 & 0 & 0 & 1
\end{smallmatrix}\right) & \left(\begin{smallmatrix}-1 & 0 & 0 & 0\\
-1 & -1 & 0 & 0\\
-1 & -2 & -1 & 0\\
1 & 1 & 1 & -1
\end{smallmatrix}\right)
\\\hline \end{array}
\]
$(2)$ All matrices $M_{0},M_{c},M_{\infty}$ are symplectic with
respect to $\Sigma=\left(\begin{smallmatrix}0 & 0 & 0 & 1\\
0 & 0 & 1 & 0\\
0 & -1 & 0 & 0\\
-1 & 0 & 0 & 0
\end{smallmatrix}\right)$.

\noindent $(3)$ If we set $N_{0}=1$ and $N_{\infty}=32\sqrt{2}$,
then the period integrals $\Pi_{0}(x)$ and $\Pi_{\infty}(z)$ are
analytically related by 
\[
\Pi_{0}(x)=U_{xz}\Pi_{\infty}(z),\,\,\,\,U_{xz}=\frac{1}{2\sqrt{2}}\left(\begin{smallmatrix}-1 & 0 & -2 & 4\\
0 & 1 & 0 & 2\\
0 & 0 & 8 & 0\\
0 & 0 & 0 & -8
\end{smallmatrix}\right)
\]
along a path $C=\left\{ (1-t)P_{0}+t\,P_{\infty}\mid0\leq t\leq1\right\} $
where $U_{xz}$ is a symplectic with respect to $\Sigma$. 
\end{prop}
It should be noted that $U_{xz}$ is symplectic but not integral as
in the case $U_{CA}$ in Proposition \ref{prop:Conection-Matrices}.
This indicates that two different local systems over $\mbP^{1}$,
one for a mirror family of $W$ and the other for a mirror family
of $W/\mbZ_{8}$, are represented by the same Picard-Fuchs equation
(\ref{eq:AppE-PF2222}). We remark that integral variations of Hodge
structure which underlie possible families of Calabi-Yau threefolds
over $\mbP^{1}\setminus\left\{ 0,1,\infty\right\} $ with $h^{2,1}=1$
are classified in \cite{DoMo}.

~

~

~

\section{\label{sec:Appendix-Pgn-AB}\textbf{Explicit forms of polynomials
$P_{g,2}^{A}$ and $P_{g,2}^{B}\,\,(g=1,2)$ }}

The explicit forms of $P_{g,n}^{A}$ and $P_{g,n}^{B}$ become complicated
even for lower $g$ and $n$. Here we present $P_{g,2}^{A}$ and $P_{g,2}^{B}$
for $g=1,2$ as examples, and refer to \cite{HT-math-c} for more
data. Here we recall the definitions 
\[
S:=\theta_{3}(q)^{4},\quad T:=\theta_{3}(q^{2})^{4},\quad U:=\theta_{3}(q)^{2}\theta_{3}(q^{2})^{2}\quad\Big(\theta_{3}(q)=\sum_{n\in\mbZ}q^{n^{2}}\Big)
\]
 which we introduced in Observation \ref{obs:Z11-A}.

\vskip0.3cm\noindent$\bullet$ \uline{$g=1$ and $n=2$.}
\[
\begin{alignedat}{1}P_{1,2}^{A} & =\frac{1}{2^{10}3^{4}}E_{2}^{4}+\frac{1}{2^{12}3^{3}}E{}_{2}^{3}(S\n+\n2T\n+\n4U)\\
 & +\frac{1}{2^{13}3^{4}}E{}_{2}^{2}(23S^{2}\n+\n244ST\n+\n128T^{2}\n+\n40SU\n-\n352TU)\\
 & +\frac{1}{2^{12}3^{3}}E_{2}(6S^{3}\n-\n53S^{2}T\n+\n40ST^{2}\n-\n80T^{3}\n+\n8S^{2}U\n-\n56STU\n+\n160T^{2}U)\\
 & +\frac{1}{2^{13}3^{4}}(35S^{4}\n+\n562S^{3}T\n+\n29768S^{2}T^{2}\n+\n14176ST^{3}\n+\n1664T^{4}\\
 & \hsp{50}\n-\n352S^{3}U\n-\n9928S^{2}TU\n-\n31040ST^{2}U\n-\n4736T^{3}U)
\end{alignedat}
\]

\[
\begin{alignedat}{2}P_{1,2}^{B} &  & \;\; & =\frac{1}{2^{10}3^{4}}E_{2}^{4}+\frac{13}{2^{12}3}E{}_{2}^{3}(S\n+\n2T\n+\n4U)\\
 &  &  & +\frac{1}{2^{12}3^{4}}E{}_{2}^{2}(31S\n+\n422ST\n+\n106T^{2}\n+\n32SU\n+\n280TU)\\
 &  &  & +\frac{1}{2^{12}3^{3}}E_{2}(23S^{3}\n-\n26S^{2}T\n+\n658ST^{2}\n+\n32T^{3}\n-\n28S^{2}U\n+\n552STU\n+\n256T^{2}U)\\
 &  &  & +\frac{1}{2^{12}3^{3}}(33S^{4}\n+\n378S^{3}T\n+\n5960S^{2}T^{2}\n+\n1232ST^{3}\n+\n64T^{4}\\
 &  &  & \hsp{45}\n-\n80S^{3}U\n-\n2736S^{2}TU\n-\n1648ST^{2}U\n-\n384T^{3}U).
\end{alignedat}
\]

\noindent$\bullet$ \uline{$g=2$ and $n=2$.}

\[
\begin{alignedat}{1} & P_{2,2}^{A}=\frac{13E_{2}^{6}}{2^{19}3^{6}}+\frac{19E_{2}^{5}}{2^{19}3^{6}}(S\n+\n2T\n+\n4U)+\frac{E{}_{2}^{4}}{2^{20}3^{5}}(49S^{2}\n+\n216ST\n+\n272T^{2}\n+\n264SU\n-\n384TU)\\
 & +\frac{E{}_{2}^{3}}{2^{18}3^{7}}(286S^{3}\n-\n675S^{2}T\n-\n4968ST^{2}\n-\n2032T^{3}\n-\n24S^{2}U\n+\n3520STU\n+\n6816T^{2}U)\\
 & +\frac{E{}_{2}^{2}}{2^{18}3^{5}}(75S^{4}\n+\n987S^{3}T\n-\n9208S^{2}T^{2}\n+\n16688ST^{3}\n+\n1792T^{4}\\
 & \hsp{50}\n-\n392S^{3}U\n+\n4296S^{2}TU\n-\n4768ST^{2}U\n-\n8448T^{3}U)\\
 & +\frac{E_{2}}{2^{19}3^{6}5}(5281S^{5}\n+\n161732S^{4}T\n-\n3489360S^{3}T^{2}\n+\n5255232S^{2}T^{3}\n-\n2745344ST^{4}\n-\n354304T^{5}\\
 & \hsp{50}\n-\n43276S^{4}U\n+\n1184672S^{3}TU\n-\n28288S^{2}T^{2}U\n-\n1897984ST^{3}U\n+\n2032640T^{4}U)\\
 & +\frac{1}{2^{20}3^{7}5}(39241S^{6}\n+\n3204540S^{5}T\n+\n81451536S^{4}T^{2}\n+\n672351872S^{3}T^{3}\\
 & \hsp{50}\n+\n379444992S^{2}T^{4}\n+\n41499648ST^{5}\n+\n4059136T^{6}\n-\n457608S^{5}U\n-\n11005408S^{4}TU\\
 & \hsp{50}\n-\n344517888S^{3}T^{2}U\n-\n675216384S^{2}T^{3}U\n-\n138291200ST^{4}U\n-\n11821056T^{5}U)
\end{alignedat}
\]

~

\[
\begin{alignedat}{1} & P_{2,2}^{B}=\frac{13E_{2}^{6}}{2^{19}3^{6}}\n+\n\frac{59E_{2}^{5}}{2^{19}3^{6}}(S\n+\n2T\n+\n4U)\n+\n\frac{E_{2}^{4}}{2^{18}3^{6}}(101S^{2}\n+\n1162ST\n+\n347T^{2}\n+\n256SU\n+\n1196TU)\\
 & +\frac{E_{2}^{3}}{2^{18}3^{7}}(971S^{3}\n+\n4410S^{2}T\n+\n37350ST^{2}\n+\n2080T^{3}\n+\n276S^{2}U\n+\n17312STU\n+\n21696T^{2}U)\\
 & +\frac{E_{2}^{2}}{2^{19}3^{5}}(599S^{4}\n+\n6088S^{3}T\n-\n4664S^{2}T^{2}\n+\n48816ST^{3}\n+\n1728T^{4}\\
 & \hsp{50}\n-\n1488S^{3}U\n-\n3664S^{2}TU\n+\n61776ST^{2}U\n+\n3712T^{3}U)\\
 & +\frac{E_{2}}{2^{18}3^{5}}(7429S^{5}\n+\n108626S^{4}T\n-\n1119372S^{3}T^{2}\n+\n3369120S^{2}T^{3}\n-\n185792ST^{4}\n-\n24064T^{5}\\
 & \hsp{50}\n-\n47740S^{4}U\n+\n337472S^{3}TU\n-\n478624S^{2}T^{2}U\n-\n224896ST^{3}U\n+\n207872T^{4}U)\\
 & +\frac{1}{2^{18}3^{7}5}(39494S^{6}\n+\n357390S^{5}T\n+\n24877263S^{4}T^{2}\n+\n165780568S^{3}T^{3}\\
 & \hsp{50}\n+\n83786304S^{2}T^{4}\n+\n12746880ST^{5}\n+\n277760T^{6}\n-\n207960S^{5}U\n-\n1572428S^{4}TU\\
 & \hsp{50}\n-\n101351544S^{3}T^{2}U\n-\n131884032S^{2}T^{3}U\n-\n38246272ST^{4}U\n-\n2500608T^{5}U)
\end{alignedat}
\]

~

~

~

~

\selectlanguage{english}%
\vspace{1cm}

\selectlanguage{american}%
{\small{}Shinobu Hosono}{\small \par}

{\small{}Department of Mathematics, Gakushuin University, }{\small \par}

{\small{}Mejiro, Toshima-ku, Tokyo 171-8588, Japan }{\small \par}

{\small{}e-mail: hosono@math.gakushuin.ac.jp}{\small \par}

\selectlanguage{english}%
{\small{}~}{\small \par}

\selectlanguage{american}%
{\small{}Hiromichi Takagi}{\small \par}

{\small{}Department of Mathematics, Gakushuin University, }{\small \par}

{\small{}Mejiro, Toshima-ku, Tokyo 171-8588, Japan }{\small \par}

{\small{}e-mail: hiromici@math.gakushuin.ac.jp}{\small \par}
\end{document}